\documentclass[12pt]{amsart}
\usepackage{amsmath, amsfonts, amssymb, latexsym, amsthm, amscd, mathrsfs, stmaryrd} 
\usepackage[all]{xy}
\usepackage[makeroom]{cancel}
\usepackage{yhmath}
\usepackage{hyperref}
\usepackage{imakeidx}
\makeindex[title = Index of Notation]

\usepackage{textgreek}

\numberwithin{equation}{subsection}

\theoremstyle{definition}
\newtheorem{prob}[subsubsection]{Problem}
\newtheorem{ex}[subsubsection]{Example}
\newtheorem{rem}[subsubsection]{Remark}
\theoremstyle{plain}
\newtheorem{prop}[subsubsection]{Proposition}
\newtheorem{thm}[subsubsection]{Theorem}
\newtheorem{thrm}{Theorem}

\newtheorem{conja}{Conjecture}

\newtheorem{Question}{Question}
\newtheorem{lem}[subsubsection]{Lemma}

\newtheorem{conj}[subsubsection]{Conjecture}
\newcommand{\Hom}{\mathrm{Hom}}
\newcommand{\mbf}{\mathbf}
\newcommand{\mbb}{\mathbb}
\newcommand{\mrm}{\mathrm}
\newcommand{\A}{\mathcal A}
\newcommand{\B}{\mbf B}
\newcommand{\C}{\mbf C}

\newcommand{\E}{\mbf E}
\newcommand{\F}{\mbf  F}
\newcommand{\G}{\mbf G}
\renewcommand{\H}{\mbf H}
\newcommand{\I}{\mbf I}
\newcommand{\GL}{\mrm{GL}}

\renewcommand{\S}{\mbf S}
\newcommand{\T}{\mbf T}
\newcommand{\tT}{\tilde T}
\newcommand{\mT}{\mathcal T}
\newcommand{\U}{\mbf U}
\newcommand{\V}{\mbf V}
\newcommand{\W}{\mbf W}

\newcommand{\f}{\mbf f}
\renewcommand{\j}{\mbf j}
\renewcommand{\k}{\mbf k}

\newcommand{\bI}{\widehat I}
\newcommand{\bO}{\mbf O}
\newcommand{\Fr}{\mrm{Fr}}
\newcommand{\Ind}{\mrm{Ind}}
\newcommand{\Res}{\mrm{Res}}
\newcommand{\sgn}{\mrm{sgn}}

\newcommand{\ve}{\varepsilon}

\usepackage{color}
\newcommand{\nc}{\newcommand}
\nc{\redtext}[1]{\textcolor{red}{#1}}
\nc{\bluetext}[1]{\textcolor{blue}{#1}}
\nc{\greentext}[1]{\textcolor{green}{#1}}
\nc{\yl}[1]{\redtext{ #1}}
\nc{\zb}[1]{\redtext{From zb: #1}}

\title[Quantum groups and edge contractions]{Quantum groups and  edge contractions}

\author{Yiqiang Li}
\address{Department of Mathematics\\ University at Buffalo, The State University of New York  \\Buffalo, NY 14260}
\email{yiqiang@buffalo.edu}

\date{\today}

\begin{document}
\maketitle


\begin{abstract}
We study the behaviors of quantum groups under an edge contraction. 
We show that there exists an explicit embedding induced by an edge contraction operation. 
We further conjecture that this explicit embedding is a section of an explicit subquotient. 
This conjecture is proved  when restricts to negative/positive half of a quantum group. 
The compatibility of the Hopf algebra structure of, and many other intrinsic structures associated with, a quantum group with the embedding and subquotient is  
studied along the way. 
The embedding phenomena are further observed in various representation theoretic objects such as Weyl groups and Chevalley groups. 
\end{abstract}

\section*{Introduction}

Graphs are used pervasively in addressing  problems in mathematics, many of classification nature. 
Edge contraction is  a simple operation  that produces a new graph $\Gamma/e$ by merging the two end vertices along a fixed edge, $e$, in  a given graph $\Gamma$. 
It is  natural to see how objects attached to graphs behave under edge contractions. 
Precisely, we ask

\begin{Question}
\label{main}
Let $A_\Gamma$ be an object attached to $\Gamma$. What is the relationship between $A_\Gamma$ and $A_{\Gamma/e}$? 
\end{Question}

For example, when $A_\Gamma$ is the chromatic polynomial of $\Gamma$, we have $A_\Gamma= A_{\Gamma \backslash e}- A_{\Gamma/e}$. 
In this article, we shall provide answers to Question~\ref{main} for various objects arising from representation theory, especially those related to quantum groups. 
We are drawn to study this question due to several results in representation theory recently can be rephrased as answers to this question. 
For example, an embedding of an affine $\mathfrak{gl}_n$
into an affine $\mathfrak{gl}_{n+1}$ in~\cite{M18} can be regarded as an outcome induced by  an edge contraction of an affine type $A^{(1)}_n$ graph to
an affine type $A^{(1)}_{n-1}$ graph. 
Similarly, an embedding of a quantum affine Schur algebra into its higher rank can be regarded as a result of the above-mentioned edge contraction.   
It is further shown in~\cite{Li21} that the two kinds of embeddings have a common generalization 
as  an embedding of the associated quantum affine 
$\mathfrak{gl}_n/\mathfrak{sl}_n$. 
The above embeddings play  important roles in resolving various problems of substantial interest in representation theory, e.g.,~\cite{M18, RW18, LS20}. 
As such,  it is natural to see if a similar phenomenon extends to  an arbitrary quantum group, say $\U_\Gamma$. 
In this article, we show  that
indeed it is the case.

\begin{thrm}[Theorem~\ref{Psi-U}]
\label{emb-U}
There exists an explicit embedding of $\U_{\Gamma/e}$ into  $ \U_{\Gamma}$ for an arbitrary $\Gamma$ induced by the edge contraction along $e$. 
\end{thrm}

We further show that the embedding is well behaved  with respect to the Hopf algebra structures on quantum groups. 
The embedding phenomenon further exists on the modified quantum groups, the generalized quantum Schur algebras, 
and tensor products of representations. 
It is compatible with other intrinsic structures of a quantum group such as   inner products,  canonical bases and braid group actions.

An interesting observation is that the above embedding is related to a standard embedding by Lusztig's braid group actions if the edge contraction is operated along a linear tree. 

The analysis on quantum groups forms the first step towards the formation of this article and leads to a systematic study of Question~\ref{main}
for other representation theoretic objects attached to a graph. 
In particular,  we establish the following embedding results. 

\begin{thrm}[Lemma~\ref{Emb-Weyl-1}, (\ref{emb-Root}), Theorem~\ref{Hall-emb}~\ref{Psi-group}]
\label{emb-WRH}
Edge contractions induce embeddings among  Weyl groups, root systems, Hall algebras, Lie algebras and Chevalley groups.
\end{thrm}

We note that when $\Gamma$ is of type $A_n$ or $A_n^{(1)}$, we recover conceptually the above mentioned embeddings $\mathfrak g_{\Gamma/e}\hookrightarrow \mathfrak g_\Gamma$ of affine $\mathfrak{gl}_n$ and 
the embedding on the group level
 $\G_{\Gamma/e, R}\hookrightarrow \G_{\Gamma, R}$, for  a commutative ring $R$. 
 
 A main tool in our approach to quantum groups is  the theory of Hall algebras. 
 In this setting, it naturally  reveals further that the embedding among Hall algebras is a section of a subquotient.  
 This remains true for the negative/positive half of the quantum groups. 
 Based on these facts, we further conjecture that

 \begin{conja}[Conjecture~\ref{subqa}]
 \label{thrm-subq}
 The explicit embedding  for quantum groups  in Theorem~\ref{emb-U} is a section of a subquotient. 
 \end{conja}

 
 When $A_\Gamma$ is a Khovanov-Lauda-Rouquier algebra $R_\Gamma$, it is shown in~\cite{M18} that  $R_{\Gamma/e}$ is a subquotient of $R_\Gamma$.  The subquotient result on Hall algebras  is in a similar spirit of {\it loc. cit}.
 Furthermore, our results  lead us to wonder if the  split subquotient phenomenon is universal. Or at least part of it. 
 Objects of particular interest   are Hecke algebras and their variants and quiver varieties. We shall address these topics in  separate publications together with  applications. 

 Variants of graphs, such as Satake diagrams, are used in representation theory for a parametrization. One can ask questions similar to Question~\ref{main} in these settings.
 In the case of Satake diagrams, the question is to address simple real groups/algebras and quantum symmetric pairs. 
 Partial results to this direction, e.g.,~\cite{BKLW, FLLLW}, suggest that the split subquotient phenomenon still exists, and further in light of the unreasonably effectiveness of this phenomenon in {\it loc. cit.}, it merits a further study. We hope to return to this topic in a near future.

 To give a graph is the same as to give a generalized symmetric Cartan matrix. The above Theorems remain holds in a broader setting, 
 i.e., for  any generalized symmetrizable Cartan matrix.
 
 Follows is the layout of this article. In Section~\ref{Prel}, we lay out some preliminaries and address the behaviors of Cartan data, graphs, root data and Weyl groups
 under an edge contraction.
 In Section~\ref{Hall}, we exhibit the split subquotient phenomenon in the Hall algebras. 
 In Section~\ref{QG}, we extend the phenomenon to quantum groups and we study the compatibility with the intrinsic Hopf algebra structures and canonical bases. 
 In Section~\ref{Braid}, we study the compatibilty of the split subquotient with Lusztig's braid group actions on quantum groups. 
 In Section~\ref{Linear}, we show that the embedding of quantum groups is related to a standard embedding by Lusztig's braid group actions if the edge contraction 
 is taken along a linear tree.
 In Section~\ref{cyclic}, we make connections between edge contractions on representation spaces with a natural embedding in affine flag varieties. 
 We study a finer structure of Hall algebra of a cyclic quiver under an edge contraction. 
 We briefly mention the edge contractions on quiver with loops. A detailed analysis will be appeared elsewhere.

\subsection*{Acknowledgements}
Twenty years ago, in Manhattan Kansas, Zongzhu gave me a xerox copy of Lusztig's article~\cite{L98} titled ``Canonical bases and Hall algebras'' that just came out. 
While getting stuck, and thus trying to avoid, checking the q-Serre relations  in this paper,  I dug out that copy, which is still hard to find nowadays. 
As it turned out, the framework  therein is the right solution.
I still remember that we spent hours and hours going  over  that article step by step by then. 
Without this training, it would not have been  possible for me to arrive at this solution. 
On that note, it is a pleasure  to thank him, yet again,  for what he had taught me.

I thank Adhish Rele for a careful proofreading of this paper. 

\tableofcontents

\section{Preliminaries}
\label{Prel}
In this section, we study the behaviors of Cartan data, oriented graphs (i.e., quivers) and Weyl groups under edge contractions.

\subsection{Cartan data}
\label{Cartan}
Let $\mbb N=\{0, 1, 2, \cdots\}$ be the set of natural numbers.
Let \index{$(I, \cdot)$} $(I, \cdot)$ be a Cartan datum, i.e.,  a finite set $I$ is given together with  a $\mbb Z$-valued symmetric bilinear form ``$\cdot$'' on the free abelian group $\mbb Z[I]$ 
such that 
\begin{align}
\label{I-cond}
 & i\cdot i \in 2\mbb N\backslash \{0\}, \quad \forall i\in I; \quad 
 2\frac{i\cdot j}{i\cdot i}  \in - \mbb N, \quad  \forall i \neq j \in I.
\end{align}
Assume that there is  a pair $(i_+, i_-)$ of elements in $I$ satisfying the following  condition:
\begin{align}
\label{i-comp} 
i_+\cdot i_+ = i_- \cdot i_- = - 2 i_+\cdot i_-. 
\end{align}
Let us fix forever such a pair. We write
\[
i_0 = i_++i_- \in \mbb Z[I].
\]
We let
\begin{align}
\label{I-hat}
\bI= I \cup \{ i_0\} - \{ i_+, i_-\}.
\end{align}
Clearly, the free abelian group $\mbb Z[\bI]$ is a subgroup of $\mbb Z[I]$. The bilinear form $\cdot$ on $\mbb Z[I]$ induces 
via restriction a symmetric bilinear form on $\mbb Z[\bI]$, still denoted by the same notation.
Thanks to the assumption (\ref{i-comp}),  we have   \index{($\bI, \cdot)$}

\begin{lem}
The pair $(\bI, \cdot)$ is a Cartan datum. 
\end{lem}

\begin{proof}
We only need to show that the condition (\ref{I-cond}) holds with $i_0$ involved. 
In light of the assumption (\ref{i-comp}), we have
\begin{align*}
i_0\cdot i_0 & = i_+\cdot i_+ + 2 i_+\cdot i_- + i_-\cdot i_- = i_+ \cdot i_+ \in 2\mbb N,\\
2\frac{i_0 \cdot j}{i_0\cdot i_0} & = 2 \frac{i_+\cdot j}{i_+\cdot i_+}  + 2\frac{ i_- \cdot j}{i_-\cdot i_-}  \in -\mbb N, \quad \forall j\in \bI -\{i_0\},\\
2\frac{j\cdot i_0} {j\cdot j}& = 2\frac{j \cdot i_+}{j\cdot j} + 2 \frac{j\cdot i_-}{j\cdot j} \in -\mbb N, \quad \forall j\in \bI-\{i_0\}.
\end{align*}
Therefore, the lemma holds. 
\end{proof}
The Cartan datum ($\bI, \cdot)$ is called the edge contraction of $(I, \cdot)$ along  the pair $\{ i_+, i_-\}$.

\subsection{Graphs}
\label{Graph}

A finite oriented graph consists of a quadruple $ (\I, \Omega, \ ':  \Omega \to \I, \ '': \Omega \to \I)$, where 
$\I$ and $\Omega$ are two finite sets with $\I$ nonempty and $'$ and $''$ are two maps.  
For a given edge $h\in \Omega$, we write $h'$ and $h''$ for its images of the maps $'$ and $''$ respectively. 
We assume that the finite oriented graph 
has no loops, i.e., $h'\neq h''$ for all $h\in \Omega$. 
An admissible automorphism $a$ of the oriented graph $(\I,\Omega)$ consists of a pair $(a: \I\to \I, a: \Omega \to \Omega)$ of bijections in the same notation such that
$
a(h') = a(h)', a(h'') = a(h)'' , \forall h\in \Omega, 
$
and $\{ a(h)', a(h)''\} \not \subseteq [\mbf i]$ for any $\mbf i\in\I$ where $[\mbf i]$ is the $a$-orbit of $\mbf i$.
Given the oriented graph $(\I,\Omega)$ with the admissible automorphism $a$, one can  define a Cartan datum $(\I/a, \cdot )$ where 
$\I/a$ the set of $a$-orbits in $\I$ and
$[\mbf i] \cdot [\mbf i]  = 2 \# [\mbf i]$ for all $[\mbf i]\in \I/a$ and  $[\mbf i]\cdot [\mbf j] = - \# \{ h\in \Omega| h', h'' \in [\mbf i]\cup [\mbf j]\}$, 
for all $[\mbf i] \neq[\mbf  j]\in \I/a$.

\index{$(\I, \Omega)$}

Assume that there exist two orbits $[\mbf i_+], [\mbf i_-]$ in $\I/a$ such that 
\begin{align}
\label{EC-1-b}
\# [\mbf i_+] = \# [\mbf i_-] = \#\{ h\in \Omega | h', h'' \in [\mbf i_+]\cup [\mbf i_-]\}.
\end{align}
From now on, we fix such a pair. 
Note that under the assumption, there is at most one edge between any pair $(\mbf i,\mbf j)$ in $[\mbf i_+]\times [\mbf i_-]$. 
Due to the compatibility of the automorphism $a$, if there is $h\in \Omega$ satisfing $h'\in [\mbf i_+], h'' \in[\mbf i_-]$, 
then any $\tilde h\in \Omega$ such that $\tilde h', \tilde h'' \in[\mbf i_+]\cup [\mbf i_-]$ must have
$\tilde h' \in [\mbf i_+], \tilde h''\in[\mbf i_-]$.
So we can, and we shall for simplicity, place the following assumption on our graph $(\I, \Omega)$.
\begin{align}
\label{EC}
\mbox{Any edge $h$ having $h',h''\in [\mbf i_+] \cup[\mbf i_-]$ must have 
$
h'\in [\mbf i_+]$.
}
%
\end{align}
In other words, any edge incident to both $[\mbf i_+]$ and $[\mbf i_-]$ 
starts from a vertex in $[\mbf i_+]$. 

For any $h\in \Omega$, we write $\bar h$ for a symbol. It will become clear that $\bar h$ denotes the opposite edge of $h$. But for now, it is simply a formal symbol.
Under the assumption (\ref{EC}), we define an oriented graph $(\widehat \I, \widehat \Omega)$ as follows.
\begin{align*}
\widehat \I  = \I - [\mbf i_-] , \quad
\widehat \Omega & = \Omega \cup \{ h_2h_1| (h'_1, h''_1) \in [\mbf  i_+] \times [\mbf  i_-], h''_1=h_2'\} \\
& \quad \quad \cup\{ \bar h_1 h_2 | (h'_1,h''_1) \in [\mbf i_+]\times [\mbf i_-], h_1'' = h_2'', h'_2\not \in [\mbf i_+] \}  \\
&\quad \quad  - \{ h\in \Omega| h' \ \mbox{or} \ h''\in [\mbf  i_-]\},
\end{align*}
and the maps  $': \widehat \Omega \to \widehat \I$ and $'': \widehat \Omega \to \widehat \I$ are induced by the ones on $(\I,\Omega)$ by defining
\begin{align*}
& (h_2h_1)' = h_1', (h_2h_1)''=h''_2, \quad \mbox{if} \ h_2h_1\in \widehat \Omega, h''_1=h_2', \\
 & (\bar h_1 h_2)' = h'_2, (\bar h_1h_2)''= h'_1,\quad \mbox{if} \ \bar h_1h_2\in \widehat \Omega,
  h_1'' = h_2''. 
\end{align*}
We define an automorphism $\widehat a: (\widehat \I,\widehat \Omega) \to (\widehat \I, \widehat \Omega)$ 
where $\widehat a: \widehat \I \to \widehat \I$ is the restriction of $a$ to $\widehat \I$ and
$\widehat a: \widehat \Omega \to \widehat \Omega$ is defined by
\begin{align*}
 \widehat a(h) &= a(h)  && \mbox{if}\  h\in \Omega,  \\
 \widehat a(h_2h_1) & = a(h_2) a(h_1)  &&  \mbox{if} \ h_2h_1\in \widehat \Omega,  h''_1=h_2', \\
 \widehat a(\bar h_1 h_2) & = \overline{a(h_1)} a(h_2) &&
\mbox{if} \ \bar h_1h_2 \in \widehat \Omega,  h_1'' = h_2''. 
\end{align*}
By definition, we see that $\widehat a$ is admissible.

We call $(\widehat \I,\widehat \Omega, \widehat a)$  the edge contraction of $(\I,\Omega, a)$ along $\{[\mbf  i_+],[\mbf  i_-]\}$.

\index{$(\bI, \widehat \Omega)$}

We say the Cartan datum $(I, \cdot)$ is isomorphic to  the  Cartan datum $(J, \circ)$ if there exists a bijection from $I$ to $J$ respecting the
bilinear forms.  

\begin{lem}
\label{EC-graph}
The Cartan datum of $(\widehat \I, \widehat \Omega, \widehat a)$ is isomorphic to the edge contraction of the Cartan datum 
of $(\I, \Omega, a)$ along the pair $\{ [\mbf i_+],[\mbf i_-]\}$
by $[\mbf i] \mapsto [\mbf i] $ if $[\mbf i] \neq [\mbf i_+]$ and $[\mbf i_+]\mapsto[\mbf i_+]+[\mbf i_-]$. 
\end{lem}

\begin{proof}
Let $(\widehat\I/\widehat a,\widehat  \circ)$ be the Cartan datum of $(\widehat \I, \widehat \Omega, \widehat a)$. We only need to check 
 that $[\mbf i_+] \widehat \circ[\mbf  j]  = ([\mbf i_+]+[\mbf i_-] )\circ[\mbf j]$ for all $[\mbf j] \neq [\mbf i_+]$. But we have
\[
[\mbf i_+] \widehat \circ[\mbf  j]  = - \# \{ h \in \Omega | h', h'' \in [\mbf i_+] \cup[\mbf  j]\}
- \# \{ h\in \Omega | h' , h'' \in [\mbf i_-] \cup [\mbf j]\} = [\mbf i_+] \circ [\mbf j] + [\mbf i_-] \circ [\mbf j].
 \]
So the lemma holds.
\end{proof}

One can realize any Cartan datum as the Cartan datum of a triple $(\mbf I,\Omega, a)$.   
As such, we can assume that the Cartan datum $(\I/a, \cdot) $ of $(\I,\Omega, a)$  provides a realization of the Cartan datum $(I, \cdot)$ in Section~\ref{Cartan}, i.e.,
\begin{align}
\label{I-ident}
(\mbf I/a, \circ) =(I,\cdot). 
\end{align}
Accordingly, we  assume that 
\begin{align}
\label{I-ident-2}
[\mbf i_+]=i_+, [\mbf i_-]=i_-. 
\end{align} 
so that the assumptions (\ref{i-comp}) and (\ref{EC-1-b}) are compatible. 
Thanks to Lemma~\ref{EC-graph}, we can, and shall, identify 
\begin{align} 
\label{I-indent-3}
(\widehat \I/\widehat a, \widehat \circ )=(\bI,\cdot), [\mbf i]\mapsto i , [\mbf i_+] \mapsto i_++i_-, \forall [\mbf i]\in \widehat{\mbf I}/\widehat a-[\mbf i_+].
\end{align}

Let $\bar \Omega=\{ \bar h| h\in \Omega\}$ and $H= \Omega \cup \bar \Omega$. 
We can extend the maps $', \ ''$ on $(\I, \Omega)$ to the pair $(\I, H)$ by declaring 
$(\bar h)'= h'', (\bar h)'' = h'$. 
Thus $\bar h$ is the opposite edge of $h$.
Then $(\I, H, ', '')$ is a finite graph in the sense of~\cite[9.1.1]{L10} (see also~\cite[1.7]{L98}).
If $a$ is an admissible automorphism of $(\I,\Omega)$, then $a$ can be extended to an automorphism 
on $(\I, H)$, denoted by the same notation,  by declaring $a(\bar h) = \overline{a(h)}$ for all $h\in \Omega$. 
The edge contraction of $(\I, H)$ is the new graph $(\widehat \I, \widehat H)$ where $\widehat H = \widehat \Omega \cup \overline{\widehat \Omega}$,
where $\overline{\widehat \Omega}=
\{ \bar h | h\in \widehat \Omega \cap \Omega\} \cup \{ \bar h_1 \bar h_2 | h_2h_1\in \widehat \Omega - \Omega\}$.
Note that for the edge contraction of $(\I, H)$, the assumption (\ref{EC}) is not needed.

\subsection{Root data}
\label{Edge-root}

\index{$(Y, X)$}

Let $ (Y, X, \langle-,-\rangle, I\hookrightarrow X (i\mapsto i'), I \hookrightarrow Y (i\mapsto i) )$ be a root datum of type $(I, \cdot)$, where
$\langle-,-\rangle : Y\times X \to \mbb Z$ is a perfect pairing of the finitely generated free abelian groups $Y$ and $X$ such that 
$\langle i, j'\rangle =2\frac{i\cdot j}{i\cdot i} $ for all $i, j\in I$. For simplicity, we write $(Y,X)_I$ for this root datum. 

Recall from Section~\ref{Cartan} that $(\bI,\cdot)$ be the edge contraction of $(I,\cdot)$ along the pair $\{ i_+,i_-\}$.
The embedding $I\hookrightarrow X$ defines a group homomorphism $\mbb Z[I]\to X$. 
Define an embedding $\bI \hookrightarrow X$ to be the restriction to $\bI$ of the  homomorphism $\mbb Z[I]\to X$, i.e., 
$i\mapsto i'$ if $i\in \bI-\{i_0\}$ and $i_0\mapsto (i_+)' + (i_-)'$. 
The embedding $I\hookrightarrow Y$ defines a group homomorphism $\mbb Z[I]\to Y$. 
Define an embedding 
$\bI  \mapsto Y$ to be the restriction to $\bI$ of the homomorphism $\mbb Z[I] \to Y$, i.e., 
$i\mapsto i$ if $i\in \bI-\{i_0\}$ and $i_0\mapsto i_++i_-$. 
Then the collection $(Y, X, \langle-,-\rangle, \bI\hookrightarrow X, \bI \hookrightarrow Y )$ is a root datum of type $(\bI, \cdot)$ as well.
For simplicity, we write $(Y, X)_{\bI}$ for this root datum. 
We call $(Y, X)_{\bI}$ the edge contraction of the root datum $(Y, X)_{I}$ along $\{ i_+, i_-\}$.

Let $X^+_I$ (resp. $X^+_{\bI}$) be the subset of all elements $\lambda\in X$ such that $\langle i,\lambda\rangle \in \mbb N$ for all $i\in I$ (resp. $i\in \bI$). Clearly we have 
$X^+_I\subseteq X^+_{\bI}$.

\subsection{Weyl groups}
\label{Emb-Weyl}
\index{$W_I$}
Let $(Y, X)_I$ be a root datum of $(I, \cdot)$. 
Assume that the root datum is $Y$-regular, that is, $I$ is linearly independent in $Y$. 
For any $i\in I$, we define a reflection $s_i: Y\to Y$ by $s_i (\mu) = \mu - \langle \mu, i'\rangle i$ for all $\mu\in Y$.
Let $W_I$ be the group generated by $s_i$ for all $i\in I$. 
It is well-known that $W_I$ is isomorphic to the Weyl group of the Cartan datum $(I, \cdot)$.

Let $(Y, X)_{\bI}$ be the edge contraction of $(Y, X)_I$ along $\{ i_+,i_-\}$. 
Since $(Y, X)_{I}$ is $Y$-regular, so is $(Y, X)_{\bI}$. 
Let $W_{\bI}$ be the group generated by the reflections  $s_i$ for all $i\in \bI$. 
This is the Weyl group of the Cartan datum $(\bI, \cdot)$.

\index{$\Psi_W$}

\begin{lem}
\label{Emb-Weyl-1}
The assignments $s_i \mapsto s_i$ for all $i\in \bI -\{i_0\}$ and $s_{i_0} \mapsto s_{i_+} s_{i_-} s_{i_+}$ define a group embedding
\begin{align*}
\Psi_W: W_{\bI} \to W_{I}. 
\end{align*}
\end{lem}

\begin{proof}
The generators $s_i$ for all $i\in \bI-\{i_0\}$ in $W_{\bI}$ and $W_I$ are the same. 
So we only need to show that $s_{i_0} =s_{i_+} s_{i_-} s_{i_+}$. 
By definition, we have
\begin{align*}
s_{i_+}s_{i_-} s_{i_+} (\mu ) 
& =s_{i_+}s_{i_-}  (\mu - \langle \mu, i'_+\rangle i_+)\\
& = s_{i_+} ( \mu - \langle \mu, i'_-\rangle i_- - \langle \mu, i'_+\rangle i_0) \\
& =\mu- \langle \mu, i'_+\rangle i_+ - \langle \mu, i'_-\rangle i_0 - \langle \mu, i'_+\rangle (  i_0- i_+ )\\
&=\mu - \langle \mu, i_0\rangle i_0= s_{i_0} (\mu), \quad \forall \mu \in Y.
\end{align*}
The lemma is proved.
\end{proof}

Note that a similar argument applies to an embedding in Coxeter groups under edge contractions.
Details will be appeared else where.

In the remaining part of this section, we assume that the Cartan datum $(I, \cdot)$ is of finite type, that is the matrix $(i\cdot j)_{i, j\in I}$ is positive definite.
Let $(Y, X)$ be the simply-connected root datum of $(I, \cdot)$.
Let $R_I$ be the set of elements in $Y$ of the form $w(i)$ for all $w\in W_I$ and $i\in I$. 
It is the root system of $(I, \cdot)$. 
Thanks to Lemma~\ref{Emb-Weyl-1}, we have 
\begin{align}
\label{emb-Root}
R_{\bI} \subseteq R_{I}. 
\end{align}
\index{$R_I$}
Clearly the set of positive roots with respect to the set of simple roots $\{ i\in \bI \}$ is included in the set of positive roots in $R_I$ with respect to 
the set of simple roots $\{i\in I\}$. 
Let $R^+_I$ be the subset in $R_I$ consisting of all elements of the form $\sum_{i\in I} a_i i$ where $a_i\in \mbb N$. 
Since $i_0= i_{i_+} + i_{i_-}$. We have $R^+_{\bI} \subset R^+_I$. 

Note that the root system of a Cartan datum can be defined without the finite-type assumption. In this case, we still have
\begin{align}
\label{Edge-root-a}
R^+_{\bI}\subseteq R^+_I.
\end{align}
Its proof is given in the following Remark ~\ref{flag-a}  (3).

\section{Embeddings among Hall algebras}
\label{Hall}

In this section, we establish an embedding among Hall algebras induced by edge contractions. 
We further show that the embedding is compatible with the bialgebra structures and inner products. 
We follow Lusztig's treatment of Hall algebras~\cite{R90} in~\cite{L98}.
In the end, we show that edge contractions induce a split subquotient among Hall algebras.

\subsection{Representation spaces}
\label{Hall-space}

\index{$\E_{\V, \Omega}$}
\index{$\E^\F_{\V, \Omega}$}
\index{$\E^{\heartsuit}_{\V,\Omega}$}
\index{$\G^\F_{\V}$}
\index{$\Fr$}
\index{$\F$}
\index{$\mathcal V_I$}
\index{$\mbf q_\V$}
\index{$\mu_\nu$}

Let $p$ be a prime number. Fix $q= p^e$ for some nonzero $e\in \mbb N$.
Let $\k$ be the algebraic closure field of the field $\mbb F_p$ of $p$ elements. 
Let $\Fr: \k\to \k$ be the Frobenius defined by $x\mapsto x^q$ for all $x\in \k$. 

To an oriented graph $(\I,\Omega, ', '')$  and an $\I$-graded finite dimensional $\k$-vector space $\V=\oplus_{\mbf i\in \I} \V_{\mbf i}$, we define the representation space
\[
\E_{\V,\Omega} = \oplus_{h\in \Omega} \Hom ( \V_{h'} , \V_{h''}).
\]
The group $\G_{\V} =\prod_{\mbf i\in \I} \GL(\V_{\mbf i})$ acts naturally from the left on $\E_{\V,\Omega}$ by conjugation. 

The Frobenius $\Fr$ on $\k$ induces naturally Frobenius maps on $\E_{\V,\Omega}$ and $\G_{\V}$, still denoted by the same notation, by raising
the entries to the $q$-th power when elements are regarded as matrices.

Recall from (\ref{I-ident}) that $I=\mbf I/a$. 
Let $\mathcal V_{I}$ be the set of all $\I$-graded vector spaces $\V$ over $\k$  such that $\mbf V_\mbf i=\mbf V_{a(\mbf i)}$  as vector spaces for all $\mbf i\in \I$.
For any $\nu\in \mbb N[I]$, let 
\begin{align}
\label{vector}
\mathcal V_{I,  \nu}=\{ \V\in \mathcal V_{I} | \dim \V_\mbf i=\nu_i, \forall \mbf i\in i, i\in I=\mbf I/a \}.
\end{align}
For any $\V$ in $\mathcal V_I$, it is equipped with an isomorphism $a$ defined by  permutation of the component vectors so that   $a(\V)_{\mbf i} =\V_{a(\mbf i)}$ for all $\mbf i\in \I$. 
It induces permutations on $\E_{\V,\Omega}$ and $\G_{\V}$ respectively, still denoted by the same notation. 

From the above analysis, we see that if $\V\in \mathcal V_I$, the variety  $\E_{\V,\Omega}$ is equipped with two commuting  automorphisms $\Fr$ and $a$. 
Let $\F= \Fr \circ a$ be the composition of automorphisms $\Fr$ and $a$  on $\E_{\V,\Omega}$. 
Similarly, if $\V\in \mathcal V_I$, we have automorphisms $\Fr$ and $a$ on $\G_{\V}$. 
We use $\F$ to denote the composition of $\Fr$ and $a$ on $\G_{\V}$ too. 
We call $\F$ a twisted Frobenius. 
Let $\E_{\V,\Omega}^\F$ and $\G_{\V}^\F$ be the fixed-point sets of $\F$. 
The $\G_{\V}$-action on $\E_{\V,\Omega}$ induces a $\G^{\F}_{\V}$-action on $\E^\F_{\V,\Omega}$.

Recall that the graph $(\I,\Omega)$ satisfies the condition (\ref{EC}) throughout the paper.
Recall the set $\bI$ from (\ref{I-hat}). We have $\mbb N[\bI]\subseteq \mbb N[I]$. 
For any $\V\in \mathcal V_{ I,\nu}$ where $\nu\in \mbb N[\bI]$, let
\begin{align*}
\E^{\heartsuit}_{\V,\Omega} =\{x\in \E_{\V,\Omega} | x_h \ \mbox{is an isomorphism} \ \forall h \ \mbox{such that}\  (h', h'') \in [\mbf i_+]\times [\mbf i_-]\}. 
\end{align*}
The space $\E^{\heartsuit}_{\V,\Omega}$ is a $\G_{\V}$-invariant open subvariety of $\E_{\V,\Omega}$.
Moreover, by definition, the subgroup $\G^{[\mbf i_-]}_{\V}=\prod_{\mbf i\in[\mbf  i_-]} \GL(\V_{\mbf i})$ of $\G_{\V}$ 
acts  on 
$\E^{\heartsuit}_{\V,\Omega}$ and additionally  the action is free. 
Let  
\begin{align}
\label{quotient}
\mbf q_\V: \E^{\heartsuit}_{\V,\Omega} \to \G_{\V}^{[\mbf i_-]} \backslash  \E^{\heartsuit}_{\V,\Omega} 
\end{align}
be the quotient map. 
Let $\widehat \V$ be the subspace of $\V$, i.e., $\widehat \V= \oplus_{\mbf i\in \I - [\mbf i_-]} \V_{\mbf i}$. 
We define a morphism of varieties 
\begin{align}
\label{quotient-mu}
 \mu_\nu : \E^{\heartsuit}_{\V,\Omega} \to \E_{\widehat \V,\widehat \Omega}, x\mapsto \mu_\nu (x)= (\widehat x_h)_{h\in \widehat \Omega}, 
\end{align}
by
\begin{align*}
\widehat x_h & = x_h && \mbox{if} \ h\in \widehat \Omega \cap  \Omega,\\
\widehat x_{h_2h_1} & = x_{h_2} x_{h_1} && \mbox{if} \ h_2h_1\in \widehat \Omega, h'_2=h''_1,\\
\widehat x_{\bar h_1 h_2}  & = x_{h_1}^{-1} x_{h_2} && \mbox{if} \ \bar h_1 h_2\in \widehat \Omega, h_1''= h''_2.
\end{align*}
There is also a natural projection $\hat\ : \G_{\V} \to \G_{\widehat \V}$ by forgetting $[\mbf i_-]$-components. Then we have a compatibility 
\begin{align}
\label{free}
\widehat{g. x}=\hat g.\hat x, \forall g\in \G_{\V}, x\in \E^{\heartsuit}_{\V,\Omega}. 
\end{align}

Indeed, 
if $h\in \widehat \Omega\cap \Omega$, then clearly we have $\widehat{g. x}_h=(\hat g.\hat x)_h$. 
Otherwise, there is 
\begin{align*}
\begin{split}
\widehat{g. x}_{h_2h_1}  &  = (g. x)_{h_2} (g. x)_{h_1} \\
& =  (g_{h_2''} x_{h_2} g_{h_2'}^{-1}) . (g_{h_1''} x_{h_1} g_{h_1'}^{-1})\\
& =g_{h_2''} x_{h_2}x_{h_1} g_{h_1'}^{-1}\\
&= g_{h_2''}  \hat x_{h_2h_1} g_{h_1'}^{-1}\\
& =(\hat g.\hat x)_{h_2h_1},   \quad  \quad \forall h_2h_1\in \widehat \Omega, h'_2=h''_1,\\
\widehat {g. x}_{\bar h_1 h_2} & = ((g.x)_{h_1})^{-1} (g.x)_{h_2} \\
& = (g_{h_1''} x_{h_1} g^{-1}_{h'_1})^{-1} (g_{h''_2} x_{h_2} g^{-1}_{h'_2} )\\
& =g_{h'_1} x^{-1}_{h_1} x_{h_2} g^{-1}_{h'_2} \\
& = (\hat g.\hat x)_{\bar h_1h_2}, \quad \quad \forall \bar h_1h_2 \in \widehat \Omega , h''_1=h''_2. 
\end{split}
\end{align*}
Thus we have the desired compatibility (\ref{free}).

Thanks to (\ref{free}), the map $\mu_\nu$  factors through
the quotient map $\mbf q_\V$:
\[
\xymatrix{
\E^{\heartsuit}_{\V,\Omega} \ar[r]^{\mbf q_\V} \ar[dr]_{\mu_\nu} & \G^{[\mbf i_-]}_\V \backslash \E^{\heartsuit}_{\V,\Omega} \ar[d]^{\tilde \mu_\nu} \\
& \E_{\widehat \V,\widehat \Omega} 
}
\]
Moreover,  $\mu_\nu(x) =\mu_\nu(x')$ if and only if there exists $g\in \G^{[\mbf i_-]}_\V$ such that $g.x=x'$. This implies that 
$\tilde \mu_\nu$ is a bijection.  
In addition, the differential of $\tilde \mu_\nu$ is always isomorphic. 
Since $\E_{\widehat \V, \widehat \Omega}$ is smooth and hence normal, we see that $\tilde \mu_\nu$ is an isomorphism.
 Therefore, we have 

\begin{lem}
\label{quotient-1}
The map $\mu_\nu$ in (\ref{quotient-mu}) gets identified with 
the quotient map $\mbf q_\V$ in (\ref{quotient})  via $\tilde \mu_\nu$.
\end{lem}

If $\V\in \mathcal V_{I, \nu}$, then 
the twisted Frobenius map $\F$ restricts to a twisted Frobenius $\F$ on $\E^{\heartsuit}_{\V,\Omega}$. 
We also have a twisted Frobenius $\widehat \F$ on $\E_{\widehat \V, \widehat \Omega}$ and $\G_{\widehat \V}$.
The twisted Frobenius maps are compatible with $ \mu_\nu$, and hence Lemma~\ref{quotient-1}  further infers that  

\begin{lem}
Let $\V\in \mathcal V_{I,\nu}$ where $\mathcal V_{I,\nu}$ is in (\ref{vector}) and $\nu\in \mbb N[\bI]$. 
The map $\mu_\nu$ induces a map on the fixed-point sets:
\begin{align}
\label{mu}
\mu_{\nu}: \E^{\heartsuit, \F}_{\V,\Omega} \to \E^{\widehat \F}_{\widehat \V,\widehat \Omega}. 
\end{align}
Moreover, the map is compatible with the projection $\G^{\F}_{\V}\to \G^{\widehat \F}_{\widehat \V}$ and can be identified with the quotient map
$\E^{\heartsuit, \F}_{\V,\Omega} \to \G^{[\mbf i_-], \F}_{\V} \backslash \E^{\heartsuit, \F}_{\V,\Omega}$. 
\end{lem}

We call $\mu_\nu$ a contraction map of $\E^{\F}_{\V,\Omega}$ along $\{[\mbf i_+], [\mbf i_-]\}$. 

\begin{rem}
\label{flag-a}
\begin{enumerate}
\item
In some special cases, the variety $\E^{\heartsuit}_{\V,\Omega}$ has appeared in~\cite{M15}. I noticed this while preparing this paper. 

\item
See Section~\ref{flag} for its connection with affine flag varieties. 

\item The quotient map (\ref{mu}) implies the statement (\ref{Edge-root-a}). 
By Kac's result, it is known that there is a correspondence between indecomposable quiver representations and positive roots. Now, the quotient $\mu_\nu$ implies that 
an indecomposable $\widehat \Omega$ representations produces  an indecomposable $\Omega$-representation and so we have (\ref{Edge-root-a}). 
\end{enumerate}
\end{rem}

\subsection{Induction and restriction diagrams}
\index{$\S_\W$}
\index{$x^\T$}
\index{$x^\W$}
\index{$\E'$}
\index{$\E''$}
Fix a decomposition of $\V = \W\oplus \T$ of the $\I$-graded vector space $\V$ over $\k$.
Let $\S_\W$ be the closed subvariety consisting of all elements $x\in \E_{\V,\Omega}$ such that $\W$ is $x$-invariant, i.e., $x_h(\W_{h'}) \subseteq \W_{h''}$ for all 
$h\in \Omega$.
To any $x\in \S_\W$, we can define $x^\T\in \E_{\T,\Omega}$ and $x^\W\in \E_{\W,\Omega}$ by passage to quotient $\T =\V/\W$ and restriction to $\W$ respectively.
Then, following Lusztig~\cite{L10}, we have the restriction diagram
\begin{align}
\label{restriction}
\begin{CD}
\E_{\V,\Omega} @<\iota<< \S_\W @>\kappa>> \E_{\T,\Omega} \times \E_{\W,\Omega},
\end{CD}
\end{align}
where $\iota$ is the inclusion and $\kappa$ is defined by  $x\mapsto (x^\T,x^\W)$.
Let $Q$ be the stabilizer of $\W$ in $\G_{\V}$ and $U$ the unipotent radical of $Q$. 
We have $Q/U\cong \G_\T\times \G_\W$.
The group $Q$ acts freely on $\G_\V\times \S_{\W}$ by $h. (g, x) = (gh^{-1}, h.x)$ for all $h\in Q$, $g\in \G_{\V}$ and $x\in \S_\W$.
Via restriction, the group $U$ acts freely on $\G_\V\times \S_\W$. Let
$\E' = \G_\V\times_U\S_\W$ and $\E''= \G_\V\times_Q \S_\W$ be the quotients of $\G_\V\times \S_\W$ by $U$ and $Q$ respectively.
We write $[g,x]$ for the $Q$-orbit and  $U$-orbit in $\E''$ and $\E'$. 
Then we have the following induction diagram.
\begin{align}
\label{induction}
\begin{CD}
\E_{\T,\Omega} \times \E_{\W,\Omega} @<p_1<< \E'@>p_2>> \E'' @>p_3>> \E_{\V,\Omega},
\end{CD}
\end{align}
where $p_3: [g, x] \mapsto g.x$, $p_2: [g, x]\mapsto [g, x]$ and $p_1: [g,x] \mapsto ( x^\T, x^\W)$ for any $g\in \G_\V$ and $x\in \S_\W$.

If $\V,\T,\W$ are in $\mathcal V_I$, then the diagrams (\ref{restriction}) and (\ref{induction}) are compatible with the twisted Frobenius $\F$. By taking fixed-point, we have the following Lusztig's restriction and induction  diagrams.
\begin{align}
\label{Res}
\E_{\V,\Omega}^\F  \overset{\iota}{\longleftarrow} \S_\W^\F \overset{\kappa}{\longrightarrow} \E_{\T,\Omega}^\F \times \E_{\W,\Omega}^\F,\\
%
\label{Ind}
\E_{\T,\Omega}^\F \times \E_{\W,\Omega}^\F \overset{p_1}{\longleftarrow} \E'^\F \overset{p_2}{\longrightarrow}  \E''^\F \overset{p_3}{\longrightarrow} \E_{\V,\Omega}^\F.
\end{align}

We now study the compatibility of the restriction diagram (\ref{Res}) with the map $j_{\nu} : \E^{\heartsuit,\F}_{\V,\Omega}\to \E^{\F}_{\V,\Omega}$ and 
$\mu_{\nu} : \E^{\heartsuit, \F}_{\V,\Omega} \to \E^{\widehat \F}_{\widehat \V,\widehat \Omega}$.
For any $\nu, \tau, \omega \in \mbb N [\bI]\subseteq \mbb N[I]$ 
such that $\nu=\tau+\omega$. 
Fix  a triple $(\V,\T, \W)\in \mathcal V_{I,\nu}\times \mathcal V_{I,\tau}\times \mathcal V_{I,\omega}$ such that 
$\V=\T\oplus \W$. 
We have the following diagram:
\begin{align}
\label{R-comp-1}
\begin{CD}
\E_{\V,\Omega}^\F  @<\iota<< \S_\W^\F @>\kappa>> \E_{\T,\Omega}^\F \times \E_{\W,\Omega}^\F,\\
@Aj_\nu AA @Aj'_{\nu} AA @AAj_{\tau} \times j_{\omega} A\\
\E^{\heartsuit,\F}_{\V,\Omega} @< \iota^0 <<  \S^{\heartsuit, \F}_\W @>\kappa^0 >>  \E^{\heartsuit, \F}_{\T,\Omega}\times \E^{\heartsuit,\F}_{\W,\Omega},\\
@V\mu_{\nu} VV @V\mu'_{\nu} VV @VV\mu_{\tau} \times \mu_{\omega} V\\
\E_{\widehat \V, \widehat \Omega}^{\widehat \F}  @<\widehat \iota <<  \S_{\widehat \W, \widehat \Omega}^{\widehat \F} @>\widehat \kappa >>  
\E_{\widehat \T, \widehat \Omega}^{\widehat \F} \times \E_{\widehat \W,\widehat \Omega}^{\widehat \F}\\
\end{CD}
\end{align}
where $\iota^0$ and $\kappa^0$ are restrictions of $\iota$ and $\kappa$ to the respective spaces and $\widehat \iota$ and $\widehat \kappa$ are the $\widehat \Omega$-version of $\iota$ and $\kappa$, respectively, the vertical $j$ maps are natural inclusions and $\mu$ maps are from (\ref{mu}) or their restrictions. 

\begin{lem}
\label{R-comp}
The  diagram (\ref{R-comp-1}) is commutative.  Moreover,  the top squares are cartesian. 
\end{lem}

\begin{proof}
The commutativity is clear. Note that if $h$ is such that $h', h'' \in [\mbf i_+]\cup [\mbf i_-]$, then $x_h$ is an isomorphism if and only if $x^\T_h$ and $x^\W_h$ are isomorphisms. Thus the top squares are cartesian. The lemma is proved. 
\end{proof}

\index{$ \C_{\widehat \kappa, \mu_{\tau}\times \mu_{\omega}} $}

\begin{lem}
\label{fiber-k}
Let $ \C_{\widehat \kappa, \mu_{\tau}\times \mu_{\omega}} $ be the cartesian product of the maps 
$\mu_{\tau}\times\mu_{\omega}$ and $\widehat \kappa$ in the diagram (\ref{R-comp-1}). Let 
 $\mu''_{\nu}$ and $\kappa^1$  be the projection from $ \C_{\widehat \kappa, \mu_{\tau}\times \mu_{\omega}}$ to 
$\S_{\widehat \W, \widehat \Omega}^{\widehat \F} $ and $ \E^{\heartsuit, \F}_{\T,\Omega}\times \E^{\heartsuit,\F}_{\W,\Omega}$ respectively. 
Then there exists a morphism $\tilde \kappa: \S^{\heartsuit, \F}_\W  \to  \C_{\widehat \kappa, \mu_{\tau}\times \mu_{\omega}}$ making the following diagram commute. Moreover the morphism $\tilde \kappa$ is a vector bundle whose fiber is isomorphic to 
$(\mbb F_{q^{i_-\cdot i_-}})^{\tau_{i_+} \omega_{i_-}}$, where $\mbb F_{q^{i_-\cdot i_-}}$ is the finite field of $q^{i_-\cdot i_-}$ elements.
\begin{align}
\label{fiber-k-1}
\begin{split}
\xymatrix{
\S^{\heartsuit, \F}_\W \ar[d]^{\mu'_\nu} \ar[r]^{\kappa^0} \ar[dr]^{\tilde \kappa}  & \E^{\heartsuit, \F}_{\T,\Omega}\times \E^{\heartsuit,\F}_{\W,\Omega}\\
 \S_{\widehat \W, \widehat \Omega}^{\widehat \F} & \C_{\widehat \kappa, \mu_{\tau}\times \mu_{\omega}} \ar[u]_{\kappa^1} \ar[l]^{\mu''_{\nu}}
}
\end{split}
\end{align}
\end{lem}

\begin{proof}
The existence of $\tilde \kappa$ is due to the commutative square in the bottom right of the diagram (\ref{R-comp-1}) and the universality of 
$\C_{\widehat \kappa, \mu_{\tau}\times \mu_{\omega}}$. 
By definition, the variety $\C_{\widehat \kappa, \mu_{\tau}\times \mu_{\omega}}$ consists of triples 
$(\widehat x, x', x'')\in \S^{\widehat \F}_{\widehat \W, \widehat \Omega} \times \E^{\heartsuit, \F}_{\T,\Omega}\times \E^{\heartsuit,\F}_{\W,\Omega}$ such that 
$\widehat \kappa (\widehat x) = \mu_{\tau}\times \mu_{\omega} (x', x'')$. 
Given such a triple $(\widehat x, x', x'')$, then the  $h$-component of  an element $x$ in its fiber under $\tilde \kappa$ are determined except 
the components $(x_{h_1}, x_{h_2})$ such that $h_2h_1\in \widehat \Omega$. Moreover, they are of the form
\begin{align*}
x_{h_a} = \begin{bmatrix}
x''_{h_a} & x^{\T\to \W}_{h_a} \\
0 & x'_{h_a}
\end{bmatrix}: 
\W_{h'_a}\oplus \T_{h'_a} 
\longrightarrow
\W_{h''_a} \oplus \T_{h''_a} ,
\quad \forall a=1, 2.
\end{align*}
They must satisfy the conditions
\begin{align*}
x^{\T\to \W}_{h_2} x'_{h_1} +x''_{h_2} x^{\T\to \W}_{h_1}  = \widehat x_{h_2h_1}^{\T\to \W}:\T_{h'_1} \to \W_{h''_2}, 
\mbox{if}\ h_2h_1\in \widehat \Omega, h''_1=h'_2,\\
(x''_{h_1})^{-1} x^{\T\to \W}_{h_2} - (x''_{h_1})^{-1} x^{\T\to \W}_{h_1} (x'_{h_1})^{-1} x'_{h_2} = \widehat x^{\T\to \W}_{\bar h_1h_2}:
\T_{h'_2} \to \W_{h'_1},  
\mbox{if} \ \bar h_1 h_2 \in \widehat \Omega , h''_1=h''_2, 
\end{align*}
or equivalently
\begin{align*}
x^{\T\to \W}_{h_2} &= (\widehat x^{\T\to \W}_{h_2h_1} - x''_{h_2} x^{\T\to \W}_{h_1} ) (x'_{h_1})^{-1},
&& \mbox{if}\ h_2h_1\in \widehat \Omega, h''_1=h'_2,\\
x^{\T\to \W}_{h_2} & = x^{\T\to \W}_{h_1}  \widehat x'_{\bar h_1 h_2} +x''_{h_1} \widehat x^{\T\to \W}_{\bar h_1 h_2},
&& \mbox{if} \ \bar h_1 h_2 \in \widehat \Omega , h''_1=h''_2.
\end{align*}
So the freedom of the fiber of $(\widehat x, x', x'')$ under $\tilde \kappa$ is that of $x^{\T\to \W}_{h_1}$ for  $h_1$ such that 
$h'_1,h''_1\in [\mbf i_+]\cup [\mbf i_-]$. 
Now the fact that $x$ is an $\F$-fixed point implies that it is determined only by one such $h_1$ and moreover 
$x^{\T\to \W}_{h_1}$, regarded as a matrix,  has entries in  $\mbb F_{q^{i_-\cdot i_-}}$, where $[\mbf i_-]=i_-$.
Therefore the fiber is isomorphic to $(\mbb F_{q^{i_-\cdot i_-}})^{\tau_{i_+} \omega_{i_-}}$.
The lemma is thus proved.
\end{proof}

Now we study the compatibility of the induction diagram (\ref{Ind}) with the maps $j_\nu$ and $\mu_\nu$.
For any $\nu, \tau, \omega \in \mbb N [\bI]\subseteq \mbb N[I]$ 
such that $\nu=\tau+\omega$. 
Fix  a triple $(\V,\T, \W)\in \mathcal V_{I,\nu}\times \mathcal V_{I,\tau}\times \mathcal V_{I,\omega}$ such that 
$\V=\T\oplus \W$.  We have the following diagram.
\begin{align}
\label{I-comp-1}
\begin{CD}
\E^\F_{\T,\Omega} \times \E^{\F}_{\W,\Omega} @<p_1<< \E'^\F @>p_2>>   \E''^\F @>p_3>>  \E^\F_{\V,\Omega}\\
@Aj_\tau \times j_\omega AA @AAj'_\nu A @AAj''_\nu A @AAj_\nu A\\
\E^{\heartsuit,\F}_{\T,\Omega} \times \E^{\heartsuit,\F}_{\W,\Omega} @<p^0_1<< \E'^{\heartsuit,\F} @>p^0_2>>   \E''^{\heartsuit,\F} @>p^0_3>>  \E^{\heartsuit,\F}_{\V,\Omega}\\
@V\mu_\tau\times \mu_\omega VV @VV\mu'_\nu V @VV\mu''_\nu V @VV\mu_\nu V \\
\E^{\widehat \F}_{\widehat \T,\widehat \Omega} \times \E^{\widehat \F}_{\widehat \W,\widehat \Omega} @<\widehat p_1<< \widehat \E'^{\widehat \F} @>\widehat p_2>>  \widehat  \E''^{\widehat \F} @>\widehat p_3>> \widehat  \E^{\widehat \F}_{\widehat \V,\widehat \Omega}
\end{CD}
\end{align}
where the first row is (\ref{Ind}), the second row is the restriction of the first row to the respective open varieties
and the third row is the $\widehat \Omega$ version of the first row; the vertical maps between the first two rows are natural inclusions and
the vertical maps between the last two rows are the map (\ref{mu}) and its variants. 

\begin{lem}
\label{I-comp}
The diagram (\ref{I-comp-1}) is commutative. 
Moreover, the top squares  and the bottom right square are cartesian. 
\end{lem}

\begin{proof}
The commutativity is by definition. 
The fact that the top squares are cartesian is due to the same reason of the cartesian property of the top squares in (\ref{R-comp-1}).
The cartesian property of the  bottom right square can be seen as follows. For a pair 
$( x, [\hat g, \hat x]) \in \E^{\heartsuit, \F}_{\V,\Omega}\times \widehat \E''^{\widehat \F}$ such that $\mu_\nu (x) = \widehat p_3 ([\hat g, \hat x])$, 
we see that it defines a unique element $g_{\mbf i_-}\in \GL(\V_{\mbf i_-})$ defined by  
$x_{h_1} \hat g_{\mbf i_+} x_{h_1}^{-1}$ if $h'_1=\mbf i_+$ and $h''_1=\mbf i_-$. 
This implies that the bottom right square is cartesian. 
 This finishes the proof of  the lemma. 
\end{proof}

We need the following information for the middle square in the bottom of the diagram (\ref{I-comp-1}). 
Let $\C_{\widehat p_2, \mu''_\nu}$ be the fiber product of $\widehat p_2$ and $\mu''_\nu$ in diagram (\ref{I-comp-1}).
Let $\hat p: \C_{\widehat p_2, \mu''_\nu} \to \E''^{\heartsuit,\F}$ and $\hat \mu: \C_{\widehat p_2, \mu''_\nu} \to \widehat \E'^{\widehat \F}$ be the projections.
Then the middle square in the bottom of (\ref{I-comp-1}) factors through $\C_{\widehat p_2, \mu''_\nu}$, i.e., there is a map 
$p': \E'^{\heartsuit,\F}\to \C_{\widehat p_2, \mu''_\nu}$ such that the following diagram commutes.
\begin{align}
\label{I-comp-2}
\begin{split}
\xymatrix{
\E'^{\heartsuit,\F} \ar[r]^{p^0_2} \ar[d]_{\mu'_\nu}  \ar[dr]^{p'} &  \E''^{\heartsuit,\F}\\
\widehat \E'^\F &   \C_{\widehat p_2, \mu''_\nu} \ar[u]_{\hat p} \ar[l]^{\hat \mu}
}
\end{split}
\end{align}

\begin{lem}
\label{p-fiber}
The map $p':  \E'^{\heartsuit,\F}\to \C_{\widehat p_2, \mu''_\nu}$  in (\ref{I-comp-2}) is a $\G_{\T}^{[\mbf i_-], \F}\times \G_{\W}^{[\mbf i_-], \F} $-bundle. 
\end{lem}

\begin{proof}
This is due to the fact that $p^0_2$ is a  $Q^\F/U^\F\cong  \G_\T^{\F}\times \G_\W^\F$-bundle and the map $\hat p$ is a  
$\widehat Q^{\widehat \F}/\widehat U^{\widehat \F}\cong  \G_{\widehat \T}^{\widehat \F}\times \G_{\widehat \W}^{\widehat \F}$-bundle. 
\end{proof}

\subsection{The embeddings $\psi_\Omega$}
\label{Hall-thm}
Fix a square root $q^{1/2}$ of $q$ in the field $\mbb C$ of complex numbers and let $\mbb Q(q^{1/2})$ be the subfield of $\mbb C$
generated by $\mbb Q$ and $q^{1/2}$.

Let $\phi: X\to Y$ be a map between two finite sets. If $f: Y\to \mbb Q(q^{1/2})$ is a function, we write
$\phi^* (f) = f \phi$. If $f' : X\to \mbb Q(q^{1/2})$ is a function, we define 
a function $\phi_!(f') : Y\to \mbb Q(q^{1/2})$ by $\phi_!(f')(y) = \sum_{x\in \phi^{-1}(y)} \f'(x)$ for all $y\in Y$.

Fix $\V\in \mathcal V_{I,\nu}$.
Let $\H_{\nu,\Omega}$  be the vector space over $\mbb Q(q^{1/2})$ consisting of all 
$\mbb Q(q^{1/2})$-valued $\G^{ \F}_{\V}$-invariant functions on $\E^{\F}_{\V,\Omega}$.
We set \index{$\H_\Omega$}
\begin{align}
\label{Hall-1}
\H_{\Omega} = \oplus_{\nu\in \mbb N[I]} \H_{\nu,\Omega}.
\end{align}
For any $\tau, \omega,\nu\in \mbb N[I]$ such that $\tau+\omega=\nu$, 
we fix a triple $(\T,\W,\V)\in \mathcal V_{I,\tau}\times \mathcal V_{I, \omega}\times \mathcal V_{I,\nu}$. 
Define two  maps   \index{$\Ind^\nu_{\tau,\omega}$} \index{$\Res^{\nu}_{\tau,\omega}$}
\begin{align*}
\Ind^\nu_{\tau,\omega}: \H_{\tau,\Omega}\times \H_{\omega,\Omega} \to \H_{\nu,\Omega},\
\Res^{\nu}_{\tau,\omega}: \H_{\nu,\Omega} \to \H_{\tau,\Omega}\times \H_{\omega,\Omega},
\end{align*} 
by
\begin{align*}
\Ind^\nu_{\tau,\omega}  = (q^{1/2})^{-m_\Omega(\tau,\omega)} \frac{1}{\# \G^{\F}_{\T}\times \G^{\F}_{\W} }  (p_3)_! (p_2)_! (p_1)^* \ \mbox{and}\
\Res^{\nu}_{\tau,\omega} = (q^{1/2})^{-m^*_\Omega(\tau,\omega)} \kappa_!\iota^* 
\end{align*}
where $p_i$'s are from the diagram (\ref{Ind}), $\kappa$ and $\iota$ from (\ref{Res}) and 
\index{$m_\Omega(\tau,\omega) $} \index{$m^*_\Omega(\tau,\omega) $}
\begin{align*}
m_\Omega(\tau,\omega) = \sum_{\mbf i \in \I} \tau_\mbf i\omega_\mbf i + \sum_{h\in \Omega} \tau_{h'} \omega_{h''},\
m^*_\Omega(\tau,\omega) = -  \sum_{\mbf i \in \I} \tau_\mbf i\omega_\mbf i + \sum_{h\in \Omega} \tau_{h'} \omega_{h''}.
\end{align*} 
If $\nu_1, \cdots, \nu_m\in \mbb N[I]$ for $m\geq 3$ such that $\nu_1+\cdots+\nu_m=\nu$, we define inductively $\Ind^\nu_{\nu_1,\cdots, \nu_m}$ and $\Res^\nu_{\nu_1,\cdots,\nu_m}$ by
\[
\Ind^\nu_{\nu_1,\cdots, \nu_m} =\Ind^{\nu}_{\nu_1+\nu_2,\nu_3,\cdots, \nu_m}  (\Ind^{\nu_1+\nu_2}_{\nu_1,\nu_2} \otimes 1) ,
\Res^\nu_{\nu_1,\cdots,\nu_m} = (\Res^\nu_{\nu_1,\nu_2}\otimes 1) \Res^\nu_{\nu_1+\nu_2,\nu_3,\cdots,\nu_m}.
\]
It is known that  the pair $(\H_{\Omega}, (\Ind^{\nu}_{\tau,\omega})_{\tau,\omega\in\mbb N[I]}))$ is a unital associative algebra over $\mbb Q(q^{1/2})$; see~\cite{L98}. It is the Hall algebra in Lusztig's formulation.  
For simplicity, we write \index{$r_\Omega$}
\[
r_\Omega(f) = \oplus_{\tau,\omega, \nu: \tau+\omega=\nu} \Res^{\nu}_{\tau,\omega}(f). 
\]
It is known that $r_\Omega$ defines an algebra homomorphism $r_\Omega: \H_{\Omega} \to \H_{\Omega}\otimes \H_{\Omega}$ where the multiplication on $\H_{\Omega}\otimes \H_{\Omega}$ is twisted as  
$(f_1\otimes f_2) (f'_1\otimes f'_2) = (q^{1/2})^{\nu_2\cdot \nu'_1} f_1\circ f'_1 \otimes f_2 \circ f'_2$ for all $f_1, f_2, f'_1 , f'_2 \in \H_{\Omega}$ such 
that$f_2\in \H_{\nu_2,\Omega}$ and $f'_1\in \H_{\nu'_1,\Omega}$.

Similar to $\H_{\nu,\Omega}$, we define $\H_{\nu,\widehat \Omega }$ and $\H^{\heartsuit}_{\nu,\Omega}$ for $\nu\in \mbb N[\bI]$ where $\E^{\widehat \F}_{\widehat \V, \widehat \Omega}$ and
$\E^{\heartsuit,\F}_{\V,\Omega}$  are used, respectively.
We set \index{$\H^{\heartsuit}_{\Omega} $}
\[
\H^{\heartsuit}_{\Omega} = \oplus_{\nu \in \mbb N[ \bI] } \H^{\heartsuit}_{\nu,\Omega},\quad 
\H_{\widehat \Omega} = \oplus_{\nu \in \mbb N[\bI]} \H_{\nu, \widehat \Omega}.
\]
One can define bilinear maps $\Ind^{\nu, \heartsuit}_{\tau,\omega}$ and $ \Res^{\nu,\heartsuit}_{\tau,\omega}$ with respect to $\H^{\heartsuit}_{\Omega}$, and
bilinear maps $\widehat \Ind^{\nu}_{\tau,\omega}$ and $\widehat \Res^{ \nu}_{\tau,\omega}$ on 
$\H_{\widehat \Omega}$ by using the respective diagrams in (\ref{I-comp-1}) and (\ref{R-comp-1}).
For simplicity, we write
\begin{align*}
r^{\heartsuit}_\Omega (f) =\oplus \Res^{\nu, \heartsuit}_{\tau,\omega}(f),  \ \mbox{if} \  f\in \H^{\heartsuit}_{\nu,\Omega};
\
%
r_{\widehat \Omega} (f) = \oplus \widehat \Res^{ \nu}_{ \tau, \omega}(f),  \ \mbox{if} \ f\in \H_{\nu,\widehat \Omega}
\end{align*}
where both sums run over $\tau, \omega,\nu\in \mbb N[\bI]$ such that $\tau+\omega=\nu$.

The inclusion $j_\nu : \E^{\heartsuit}_{\V,\Omega} \to \E_{\V, \Omega}$ defines an injective map 
\begin{align*}
j_{\nu!}: \H^{\heartsuit}_{\nu,\Omega} \to \H_{\nu,\Omega}
\end{align*}
by extension by zero. 
The contraction map $\mu$ in (\ref{mu}) defines an isomorphism $\mu^*_\nu: \H_{\nu,\widehat \Omega} \to \H^{\heartsuit}_{\nu,\Omega}$ by pulling back
$f\mapsto \mu\circ f$. We are interested in its twisted version:
\begin{align}
\label{mu-star}
\mu^{\star}_{\nu} = (q^{1/2})^{-\nu^2_{i_-} i_-\cdot i_-} \mu^*_\nu. 
\end{align}

By summing $j_{\nu!}$ and $\mu^\star_{\nu}$ we have the linear maps
\begin{align*}
\H_{\widehat \Omega} \overset{\mu^\star}{\longrightarrow} \H^{\heartsuit}_{\Omega} \overset{j_!}{\longrightarrow} \H_{\Omega},
\end{align*}
where $\mu^\star$ is isomorphic and $j_!$ is injective.

\begin{thm}
\label{Hall-emb}
The maps $\mu^\star$ and $j_!$ are algebra homomorphisms, and so the composition 
\begin{align}
\label{Hall-emb-a}
\psi_{\Omega}\equiv j_!\mu^\star: \H_{\widehat \Omega}\to \H_{\Omega}
 \end{align}
defines an embedding of Hall algebras.  Moreover, they are compatible with the restriction maps $r_{\Omega}, r^{\heartsuit}_\Omega$ and $r_{\widehat \Omega}$.
\end{thm}

\index{$\psi_{\Omega}$}

\begin{proof}
Since the top squares in (\ref{I-comp-1}) and (\ref{R-comp-1}) are cartesian, it is clear that $j_!$ is an algebra homomorphism and compatible with the restriction maps $r_\Omega$ and $r^{\heartsuit}_\Omega$.

Next, we show that $\mu^\star$ is compatible with the inductions $\widehat \Ind^{ \nu}_{\tau, \omega}$ and
$\Ind^{\nu,\heartsuit}_{\tau,\omega}$. 
In light of the commutative diagram (\ref{I-comp-1}), we have
\begin{align}
\label{Hall-emb-1}
\begin{split}
\mu^*_\nu (\widehat p_3)_! (\widehat p_2)_! \widehat p_1^* 
& =  (p^0_3)_! \mu''^*_\nu (\widehat p_2)_! \widehat p^*_1\\
& = (p^0_3)_! \widehat p_! \widehat \mu^* \widehat p^*_1\\
& = \frac{1}{\# \G^{i_-, \F}_{\T} \times \G^{i_-,\F}_{\W} } (p^0_3)_! (p^0_2)_! \mu'^*_{\nu} \widehat p^*_1\\
&=  \frac{1}{\# \G^{i_-, \F}_{\T} \times \G^{i_-,\F}_{\W} } (p^0_3)_! (p^0_2)_!  (p^0_1)^* (\mu^*_\tau \otimes \mu^*_{\omega}) ,
\end{split}
\end{align}
where the third equality is due to the commutative diagram (\ref{I-comp-2}) and  Lemma~\ref{p-fiber}:
\begin{align*}
(p^0_2)_! \mu'^*_\nu = 
\widehat p_! p'_! p'^* \widehat mu^* = \# (\G^{i_-, \F}_{\T} \times \G^{i_-,\F}_{\W}) \widehat p_! \widehat \mu^*.
\end{align*}
By (\ref{Hall-emb-1}), we have
\begin{align*}
\mu^\star_{\nu} \widehat \Ind^{ \nu}_{\tau,\omega}
& = (q^{1/2})^{-\nu^2_{i_-} i_-\cdot i_- - m_{\widehat \Omega}( \tau, \omega)} 
\frac{1}{\# \G^{\widehat \F}_{\widehat \T}\times \G^{\widehat \F}_{\widehat \W} } 
\mu^*_\nu (\widehat p_3)_!(\widehat p_2)_! (\widehat p_1)^*\\
&=(q^{1/2})^{-\nu^2_{i_-} i_-\cdot i_- - m_{\widehat \Omega}(\tau,\omega)} 
\frac{1}{\# \G^{\F}_\T\times \G^{\F}_\W} (p^0_3)_! (p^0_2)_!  (p^0_1)^* (\mu^*_\tau \otimes \mu^*_{\omega}) \\
& = (q^{1/2})^{-N} \Ind^{\nu,\heartsuit}_{\tau,\omega} (\mu^\star_{\tau} \otimes \mu^{\star}_{\omega})
\end{align*}
where 
\begin{align*}
N & = (\nu^2_{i_-} -\tau^2_{i_-} - \omega^2_{i_-})i_-\cdot i_- +m_{\widehat \Omega}(\tau,\omega) - m_\Omega(\tau,\omega)\\
&=(\nu^2_{i_-} -\tau^2_{i_-} - \omega^2_{i_-})i_-\cdot i_- - 2 \tau_{i_-}\omega_{i_-} i_-\cdot i_-\\
&=0.
\end{align*}
Here in the calculation of $m_{\widehat \Omega}(\tau,\omega)$, the $i_0$-entries of $\tau$ and $\omega$ is regarded as the $[\mbf i_+]$-entry under the identification
$\widehat{\mbf I}/\widehat a= \bI$. 
This proves the compatibility of $\mu^\star_\nu$ with the inductions $\widehat \Ind^{\nu}_{\tau,\omega}$ and
$\Ind^{\nu,\heartsuit}_{\tau,\omega}$.

Finally, we show the compatibility of $\mu^\star$ with the restrictions $\widehat \Res^{\nu}_{\tau, \omega}$ 
and $\Res^{\nu,\heartsuit}_{\tau,\omega}$.
By the diagram (\ref{R-comp-1}), we have
\begin{align}
\label{Hall-emb-2}
\begin{split}
\kappa^0_! (\iota^{0})^* \mu^*_\nu 
& = \kappa^0_! (\mu'_\nu)^* \widehat \iota^* \\
& = \kappa^1_! \tilde \kappa_! \tilde \kappa^* (\mu''_\nu)^* \widehat \iota^*\\
&=(q^{1/2})^{2 \tau_{i_+} \omega_{i_-} i_-\cdot i_-}  \kappa^1_! (\mu''_\nu)^* \widehat \iota^*\\
&=(q^{1/2})^{2 \tau_{i_+} \omega_{i_-} i_-\cdot i_-}( \mu^*_\tau \otimes \mu^*_\omega) \widehat \kappa_! \widehat \iota^*,
\end{split}
\end{align}
where the third equality is due to Lemma~\ref{fiber-k}. By using (\ref{Hall-emb-2}), we have
\begin{align*}
\begin{split}
\Res^{\nu,\heartsuit}_{\tau,\omega} \mu^\star_\nu 
&= (q^{1/2})^{-\nu^2_{i_-} i_-\cdot i_- - m^*(\tau, \omega)} \kappa^0_! (\iota^0)^* \mu^*_\nu\\
&= (q^{1/2})^{ -( \nu^2_{i_-}   -2 \tau_{i_+} \omega_{i_-}) i_-\cdot i_-- m^*(\tau, \omega) } ( \mu^*_\tau \otimes \mu^*_\omega) \widehat \kappa_! \widehat \iota^*\\
& =(q^{1/2})^{- N'} (\mu^\star_\tau\otimes \mu^\star_\omega) \widehat \Res^{\nu}_{ \tau,\omega},
\end{split}
\end{align*}
where
\begin{align*}
N' & = (\nu^2_{i_-}   -2 \tau_{i_+} \omega_{i_-} -\tau^2_{i_-} - \omega^2_{i_-} ) i_-\cdot i_-+ m^*_\Omega (\tau, \omega)  - m^*_{\widehat \Omega}( \tau,  \omega)\\
&=m^*_\Omega (\tau, \omega)  - m^*_{\widehat \Omega}( \tau, \omega)\\
&=0.
\end{align*}
Here in the calculation of $m^*_{\widehat \Omega}( \tau, \omega)$ the $i_0$-entries of $\tau$ and $\omega$ 
are regarded as the $[\mbf i_+]$-entries under the identification $\widehat {\mbf I}/\widehat a=\bI$. 
This shows that $\mu^\star_\nu$ is compatible with the restrictions. Therefore we have the proof. 
\end{proof}

Let $\mathscr O$ be a $\G^\F_\V$-orbit in $\E^\F_{\V, \Omega}$. \index{$\mathscr O$}
Let $p_{\mathcal O}$ be the characteristic function of the orbit $\mathscr O$ whose value at $\mathscr O$ is $1$ and $0$ otherwise. 
The collection of $p_{\mathcal O}$ forms a PBW basis of $\H_\Omega$. By the definitions, we see that the  embedding $j_!\mu^{\star}$ is compatible with the PBW bases.

\begin{prop}
There is $j_! \mu^{\star}_\nu (p_{\mathscr O'}) = p_{\mathscr O}$, if $\mu_\nu (\mathscr O) = \mathscr O'$.  
\end{prop}

Now we address the compatibility of $\mu^{\star}$ and $j_!$ with the inner products.
Following Lusztig~\cite[1.16(a)]{L98}, we define an inner product on $\H_{\nu,\Omega}$ by
\begin{align}
\label{inner}
(f_1, f_2)_{\Omega} = q^{\dim \G_\V} \frac{1}{\# \G^{\F}_\V} \sum_{x\in \E^{\F}_{\V,\Omega}} f_1(x) f_2(x).
\end{align}
The inner product induces one on $\H^{\heartsuit}_{\nu,\Omega}$ via the rule
\[
(f_1,f_2)_{\heartsuit}= (j_{\nu!} f_1, j_{\nu!}f_2)_{\Omega}, \quad \forall f_1, f_2\in \H^{\heartsuit}_{\nu,\Omega}.
\]
Further we define an inner product on $\H_{ \nu,\widehat \Omega}$ by
\begin{align*}
(( \tilde f_1, \tilde f_2))_{\widehat \Omega} = (\mu^\star_\nu (\tilde f_1), \mu^{\star}_\nu (\tilde f_2))_{\heartsuit},\quad \forall \tilde f_1, \tilde f_2\in \H_{ \nu,\widehat \Omega}. 
\end{align*}
Let  $(-,-)_{\widehat \Omega}$ be the $\widehat \Omega$-version of $(-,-)_\Omega$  on $\H_{\nu,\Omega}$. 
The following proposition shows that the two inner products on $\H_{\nu,\widehat \Omega}$ coincide. 

\begin{prop}
\label{Inner-com}
We have $((\tilde f_1,\tilde f_2))_{\widehat \Omega}= (\tilde f_1,\tilde f_2)_{\widehat \Omega}$ for 
all $\tilde f_1,\tilde f_2\in \H_{\nu,\widehat \Omega}$ and $\nu\in \mbb N[\bI]$.
\end{prop}

\begin{proof}
By definition, we have
\begin{align*}
\begin{split}
(( \tilde f_1,\tilde f_2))_{\widehat \Omega}
&= (\mu^\star_\nu \tilde f_1, \mu^\star_\nu f_2)_{\heartsuit}\\
& = q^{\dim \G_\V} \frac{1}{\# \G^\F_\V} \sum_{x\in \E^{\heartsuit,\F}_{\V,\Omega}} (\mu^\star_\nu \tilde f_1)(x) (\mu^\star_\nu \tilde f_2)(x)\\
& =  q^{\dim \G_\V} \frac{\#\G^{[\mbf i_-],\F}_\V}{\# \G^\F_\V}
\sum_{[x] \in\G^{[\mbf i_-],\F}_{\V}\backslash \E^{\heartsuit,\F}_{\V,\Omega}} (\mu^\star_\nu \tilde f_1)(x) (\mu^\star_\nu \tilde f_2)(x)\\
&= q^{\dim \G_\V} \frac{1}{\# \G^{\widehat \F}_{\widehat \V}}
\sum_{\widehat x \in \E^{\widehat \F}_{\widehat \V, \widehat \Omega}} (q^{1/2})^{-2\nu^2_{i_-} i_-\cdot i_-}  \tilde f_1(\widehat x) \tilde f_2(\widehat x)\\
&= q^{\dim \G_{\widehat \V}} \frac{1}{\# \G^{\widehat \F}_{\widehat \V}}
\sum_{\widehat x \in \E^{\widehat \F}_{\widehat \V, \widehat \Omega}} \tilde f_1(\widehat x) \tilde f_2(\widehat x)
=(\tilde f_1, \tilde f_2)_{\widehat \Omega}, \forall \tilde f_1,\tilde f_2\in \H_{\nu,\widehat \Omega}.
\end{split}
\end{align*}
Therefore the proposition is proved. 
\end{proof}

Proposition~\ref{Inner-com} says that the embedding $\psi_{\Omega}$ is compatible with the inner products on $\H_{\Omega}$ and $\H_{\widehat \Omega}$.

When $\nu=i$ and $\V(i) \in \mathcal V_{I,\nu}$, the space $\E^{\F}_{\V(i), \Omega}$ consists of a point. Let $\theta_{i,\Omega}$ be the characteristic function, denoted by  $1_{\E^{\F}_{\V(i) ,\Omega}}$, 
of the space $\E^{\F}_{\V(i) ,\Omega}$. 
\index{$\theta_{i, \Omega}$}
Let $\f_\Omega$ be the subalgebra of $\H_\Omega$ generated by $\theta_{i,\Omega}$ for all $i\in I$.
Let 
\begin{align}
\theta_{i_0,\Omega} = \theta_{i_+,\Omega} \theta_{i_-,\Omega} - q^{-1/2} \theta_{i_-,\Omega} \theta_{i_+,\Omega}. 
\end{align}
Fix $\V(i_++i_-) \in \mathcal V_{I, i_++i_-}$, By a direct computation, we see that 
\begin{align}
\theta_{i_0, \Omega} = (q^{1/2} )^{i_+\cdot i_-} 1_{\E^{\heartsuit,\F}_{\V(i_++i_-) ,\Omega}}.
\end{align}
This is a special case of Lemma~\cite[Lemma 8.10]{L98}.
By (\ref{mu-star}), we have 
\begin{align}
\label{mu-generator}
\mu^\star_{i_++i_-} ( \theta_{i_0,\widehat \Omega}) = \theta_{i_0,\Omega} \ \mbox{and}\ 
\mu^\star_{i} (\theta_{i,\widehat\Omega}) =\theta_{i, \Omega}, \forall i\in \bI -\{ i_0\}. 
\end{align}
Let $\f^{\heartsuit}_\Omega$ be the subalgebra generated by $\theta_{i,\Omega} $ for all $i\in \bI$.
Since all generators are supported on the space $\E^{\heartsuit,\F}_{\V,\Omega}$ for various $\V$. We see immediately

\begin{lem}
\label{f-comp}
The algebra $\f^{\heartsuit}_\Omega$ is a subalgebra of $\H^{\heartsuit}_{\Omega}$. 
Moreover, the map $\psi_\Omega$ restricts to an isomorphism
\begin{align}
\psi_{\Omega} : \f_{\widehat \Omega} \to \f^{\heartsuit}_\Omega, \theta_{i,\widehat \Omega} \mapsto \theta_{i,\Omega},\forall i\in \bI.
\end{align}
\end{lem}

\subsection{Hall algebras as a split subquotient}

In this section, we further show that $\H_{\widehat \Omega}$ is a split subquotient of $\H_{\Omega}$. 


Recall the inclusion $j_{\nu}: \E^{\heartsuit, \F}_{\V, \Omega} \to \E^{\F}_{\V, \Omega}$ from the diagram (\ref{R-comp-1}). 
Let $\E^{c, \F}_{\V,\Omega}$ be its complement.  This is $\G^\F_{\V}$-stable. 
Let $\H^c_{\nu,\Omega}$ be the space of $\mbb Q(q^{1/2})$-valued, $\G^\F_{\V}$-invariant, functions on $\E^{c,\F}_{\V, \Omega}$ for $\V\in \mathcal V_{\nu, I}$. 
Let 
$$
\H^c_{\Omega} = \oplus_{\nu\in \mbb N[\bI]} \H^c_{\nu, \Omega};
\H^{\bI}_{\Omega} = \oplus_{\nu\in \mbb N[\bI]} \H_{\nu,\Omega}.
$$
The space $\H^{\bI}_{\Omega}$ is a subalgebra of $\H_{\Omega}$. 
By definition, we have 
\begin{align}
\label{H-sum}
\H^{\bI}_{\Omega} = \H^{\heartsuit}_{\Omega} \oplus \H^c_{\Omega}. 
\end{align}

Note that there is a surjective map $j_{\nu}^*: \H_{\nu, \Omega} \to \H^{\heartsuit}_{\nu, \Omega}$. 
Let $j^*=\oplus j^*_\nu: \H^{\bI}_\Omega \to \H^{\heartsuit}_\Omega$ be the sum of $j^*_\nu$ over all $\nu \in \mbb N[\bI]$. 

\begin{prop}
\label{H-subq}
\begin{enumerate}
\item The space $\H^c_{\Omega}$ is a two-sided ideal of $\H^{\bI}_{\Omega}$. Moreover we have the following split short exact sequence of algebras preserving the $\bI$-grading
\[
0 \to \H^c_{\Omega} \to \H^{\bI}_{\Omega} \overset{j^*}{\to} \H^{\heartsuit}_{\Omega} \to 0. 
\]

\item $\H_{\widehat \Omega}$ is a split subquotient of $\H_\Omega$, isomorphic to $\H^{\bI}_{\Omega}/\H^c_{\Omega}$. 
\end{enumerate}
\end{prop}

\begin{proof}
Consider the diagram (\ref{Ind}). Suppose that one of the components in the pair $(x', x'') \in \E^\F_{\T,\Omega} \times \E^\F_{\W, \Omega}$ 
is in either $\E^{c, \F}_{\T,\Omega}$ or $\E^{c, \F}_{\W,\Omega}$, then the set  
$p_3 p_2 p^{-1}_1 \{ (x', x'')\} $ is in $\E^{c,\F}_{\V, \Omega}$. 
So $\H^c_\Omega$ is a two-sided ideal of $\H^{\bI}_\Omega$. 
By (\ref{H-sum}), the sequence in (1) is exact. It is split because $j^* j_!=1$.
The statement (2) is due to the combination of (1) and  Theorem~\ref{Hall-emb}. The proposition is proved. 
\end{proof}

Due to Lemma~\ref{R-comp}, the morphism $j^*$ also compatible with the restrictions.

Recall the composition algebra $\f_{\Omega}$ and $\f^{\heartsuit}_\Omega$ from Section~\ref{Hall-thm}. We define 
\[
\f^{\bI}_\Omega = \f_{\Omega} \cap \H^{\bI}_\Omega, \f^c_\Omega = \f_\Omega \cap \H^c_{\Omega}. 
\]
We have the following analogue of Proposition~\ref{H-subq}.

\begin{prop}
\label{f-subq}
\begin{enumerate}
\item The space $\f^c_{\Omega}$ is a two-sided ideal of $\f^{\bI}_{\Omega}$.

\item $\f_{\widehat \Omega}$ is a subalgebra in  $\f^{\bI}_{\Omega}/\f^c_{\Omega}$. 
Moreover, when $\widehat \Omega$ is of finite type, $\f_{\widehat \Omega} \cong \f^{\bI}_{\Omega}/\f^c_{\Omega}$.
\end{enumerate}
\end{prop}

\begin{proof}
By Proposition~\ref{H-subq}, $\f^c_\Omega$ is a two-sided ideal of $\f^{\bI}_\Omega$. 
Note that we have $\f^{\bI}_{\Omega}/\f^c_{\Omega} \hookrightarrow \H^{\bI}_{\Omega}/\H^c_{\Omega}$ thanks to the short exact sequence in
Proposition~\ref{H-subq}.
So the algebra $\f_{\widehat \Omega}$ is included in $\f^{\bI}_{\Omega}/\f^c_{\Omega}$. 
When $\widehat \Omega$ is of finite type, then $\H_{\widehat \Omega} = \f_{\widehat \Omega}$ and hence we have the last statement in (2).
The proposition is proved. 
\end{proof}

\begin{rem}
In general,  $\f^{\widehat I}_{\Omega} /\f^c_{\Omega}$ is bigger than $\f_{\widehat \Omega}$. See Section~\ref{cyclic} for an example.  
\end{rem}

Let $\f_{\Omega, \bI}$ be the subalgebra of $\f^{\bI}_{\Omega}$ generated by the following elements
$\theta_{i, \Omega} $ for $i\in I - \{ i_+, i_-\}$,  $\theta_{i_+,\Omega} \theta_{i_-,\Omega}$ and $\theta_{i_-,\Omega} \theta_{i_+,\Omega}$. 
Then we have

\begin{prop}
\label{f-quot}
The assignment $\theta_{i, \Omega} \mapsto \theta_{i, \widehat \Omega}$ for $i\in \bI-\{i_0\}$, $\theta_{i_+,\Omega}\theta_{i_-,\Omega}\mapsto \theta_{i_0,\widehat \Omega}$ and  $\theta_{i_-,\Omega} \theta_{i_+,\Omega} \mapsto 0$
defines a surjective algebra homomorphism $j^*: \f_{\Omega, \bI} \to \f_{\widehat \Omega}$. 
Moreover, the kernel of $j^*$ is the ideal of $\f_{\Omega, \bI} $ generated by $\theta_{i_-,\Omega} \theta_{i_+,\Omega}$.
The map $j^*$ splits via $\psi_{\Omega}$ in Theorem~\ref{Hall-emb}. 
\end{prop}

\begin{proof}
Since the function  $\theta_{i_-,\Omega}\theta_{i_+, \Omega}$ is supported on $\E^{c, \F}_{\V,\Omega}$, we have
$j^* (\theta_{i_-,\Omega}\theta_{i_+, \Omega})=0$, and hence
$j^* ( \theta_{i_0,\Omega}) = j^* (\theta_{i_+,\Omega}\theta_{i_-,\Omega})=\theta_{i_0,\widehat \Omega}$. 
Since any monomials in $\f_{\Omega, \bI}$ can be written as a sum of monomials in $\theta_i$ for all $ i\in \bI -\{i_0\}$ and $\theta_{i, \ve}$ plus a sum of monomials
having $\theta_{i_-} \theta_{i_+}$, we see that the kernel of $j^*$ must be the two sided ideal generated by $\theta_{i_-} \theta_{i_+}$. 
Now the proposition follows from Proposition~\ref{H-subq} and Lemma~\ref{f-comp}.
\end{proof}

\subsection{Drinfeld doubles}
\label{D-double}
We recall the construction of the Drinfeld double of Hall algebra $\H_\Omega$ from~\cite[Section 5]{X97}. 
Let $D\H^{\geq 0}_\Omega$ be the vector space over $\mbb Q(q^{1/2})$ spanned by the elements $K_\mu f^+$ where $\mu \in Y$ and $f\in \H_\Omega$. 
It carries a Hopf algebra structure. In particular, the multiplication is defined by
\begin{align}
\begin{split}
f^+ f'^+ &= f f', \quad \forall f, f'\in \H_\Omega,\\
K_\mu K_{\mu'} &= K_{\mu+\mu'}, \quad \forall \mu, \mu'\in Y,\\
K_\mu f^+ &= (q^{1/2})^{\langle \mu, \nu\rangle} f^+ K_\mu, \forall \mu\in Y, f\in \H_{\nu, \Omega}. 
\end{split}
\end{align}
Write $\tilde K_\mu= K_{\sum \mu_i \frac{i\cdot i}{2}i} $ if $\mu =\sum \mu_i i$. 
The comultiplication is defined by 
\begin{align}
\begin{split}
\Delta (f^+) &= \sum_{\tau, \omega: \tau+\omega=\nu} g_\nu^{\tau, \omega} f^+_1 \tilde K_{\omega} \otimes f^+_2, \quad \forall f\in \H_{\nu, \Omega}, \\
\Delta(K_\mu) &= K_\mu\otimes K_\mu, 
\end{split}
\end{align}
where $\Res^{\nu}_{\tau, \omega}(f) = g_\nu^{\tau, \omega} f_1\otimes f_2$. 
The counit $\ve$ is defined by $\ve (f^+) =0$ unless $f\in \H_{0, \Omega}$ and $\ve (K_\mu) =1$ for all $\mu \in Y$. 
Recall the notations $\Ind^\nu_{\nu_1,\cdots,\nu_m} $ and $\Res^\nu_{\nu_1,\cdots,\nu_m}$ from Section~\ref{Hall-thm}. 
The antipode is given by  
\begin{align}
\sigma(f^+) = \tilde K_{-\nu} \sum_{r\geq 1} (-1)^r \sum_{\nu_1,\cdots, \nu_r: \nu_i\neq 0, \forall i} \Ind^\nu_{\nu_1,\cdots, \nu_r} \Res^{\nu}_{\nu_1,\cdots,\nu_r} (f)^+, \forall f\in \H_{\nu,\Omega}.
\end{align}
And $\sigma(K_\mu)=K_{-\mu}$. 
Let $D\f^{\geq 0}_{\Omega}$ be the Hopf subalgebra spanned by the elements $K_\mu f^+$ where $\mu \in Y$ and $f\in \f_\Omega$. 

Let $D\H^{\leq 0}_{\Omega}$ be the vector space over $\mbb Q(q^{1/2})$ spanned by  by the elements $K_\mu f^-$ where $\mu \in Y$ and $f\in \H_\Omega$. 
It carries a Hopf algebra structure. In particular, the multiplication is defined by
\begin{align}
\begin{split}
f^- f'^- &= f f', \quad \forall f, f'\in \H_\Omega,\\
K_\mu K_{\mu'} &= K_{\mu+\mu'}, \quad \forall \mu, \mu'\in Y,\\
K_\mu f^- &= (q^{1/2})^{-\langle \mu, \nu\rangle} f^- K_\mu, \forall \mu\in Y, f\in \H_{\nu, \Omega}. 
\end{split}
\end{align}
The comultiplication is defined by 
\begin{align}
\begin{split}
\Delta (f^-) &= \sum_{\tau, \omega: \tau+\omega=\nu} g_\nu^{\tau, \omega} f^-_1 \otimes \tilde K_{-\omega} f^-_2, \quad \forall f\in \H_{\nu, \Omega}, \\
\Delta(K_\mu) &= K_\mu\otimes K_\mu, 
\end{split}
\end{align}
The counit $\ve$ is defined by $\ve (f^-) =0$ unless $f\in \H_{0, \Omega}$ and $\ve (K_\mu) =1$ for all $\mu \in Y$. 
The antipode is given by  
\begin{align}
\sigma(f^-) = \sum_{r\geq 1} (-1)^r \sum_{\nu_1,\cdots, \nu_r: \nu_i\neq 0,\forall i} \Ind^\nu_{\nu_1,\cdots, \nu_r} \Res^{\nu}_{\nu_1,\cdots,\nu_r} (f)^- \tilde K_{\nu}, \forall f\in \H_{\nu,\Omega}.
\end{align}
Further, $\sigma(K_\mu)=K_{-\mu}$. 
Let $D\f^{\leq 0}_{\Omega}$ be the Hopf subalgebra spanned by the elements $K_\mu f^-$ where $\mu \in Y$ and $f\in \f_\Omega$. 

Define a bilinear pairing $\varphi: D\H^{\geq 0}_\Omega \times D\H^{\leq 0}_\Omega \to \mbb Q(q^{1/2})$ by
\[
\varphi ( K_\mu f^+, K_{\mu'} f'^-) = (q^{1/2})^{- \langle \mu, \mu'\rangle- \langle \nu, \mu'\rangle + \langle \mu, \nu'\rangle} (f, f') , \forall f\in \H_{\nu, \Omega}, f'\in \H_{\nu', \Omega}.  
\]
The triple $(D\H^{\geq 0}_\Omega, D\H^{\leq 0}_\Omega, \varphi)$ is a skew-Hopf pairing. 
Applying the general machinery, we see that the tensor product $D\H_\Omega = D\H^{\geq 0}_\Omega \otimes D\H^{\leq 0}_{\Omega}$ over $\mbb Q(q^{1/2})$ carries a Hopf algebra structure.
\index{$D \H_\Omega$}
In particular, the multiplication is given by the following rules
\begin{align*}
(a\otimes 1)  (a'\otimes 1) &= aa'\otimes 1, 
(1\otimes b) (1\otimes b') = 1 \otimes bb', (a\otimes 1) (1\otimes b) = a \otimes b,\\
(1\otimes b) (a\otimes 1)& = \sum \varphi(a_1, \sigma (b_1))  \varphi (a_3,b_3) a_2\otimes b_2 ,
\end{align*}
where $a\in D\H^{\geq 0}_\Omega$, $b\in D\H^{\leq 0}_\Omega$, 
$(\Delta\otimes 1) \Delta(a) = \sum a_1\otimes a_2\otimes a_3$ and $(\Delta\otimes 1) \Delta(b)=\sum b_1\otimes b_2 \otimes b_3$. 
Let $D\f_\Omega = D\f^{\geq 0}_\Omega \otimes D\f^{\leq 0}_\Omega$ be the Hopf subalgebra of $D\H_\Omega$.

The embedding $\psi_\Omega: \H_{\widehat \Omega} \to \H_\Omega$ induces embeddings 
\begin{align*}
\psi^{\geq 0}_\Omega: D\H^{\geq 0}_{\widehat \Omega} \to D\H^{\geq 0}_{\Omega},  K_\mu f^+ \mapsto K_\mu \psi_\Omega (f)^+,\\
\psi^{\leq 0}_\Omega: D\H^{\leq 0}_{\widehat \Omega} \to D\H^{\leq 0}_{\Omega},  K_\mu f^- \mapsto K_\mu \psi_\Omega (f)^-.
\end{align*}
So by  universality it induces  injective linear maps
\begin{align}
D\psi_\Omega: D\H_{\widehat \Omega} \to D\H_\Omega,
D\psi'_\Omega: D\f_{\widehat \Omega} \to D\f_\Omega. 
\end{align}

\index{$D\psi_\Omega$}
\index{$D\psi'_\Omega$}

\begin{prop}
\label{double}
The injective map $D\psi'_\Omega: D\f_{\widehat \Omega} \to D\f_\Omega$ is an  algebra homomorphism. 
\end{prop}

\begin{proof}
By definition, it suffices to show that 
\begin{align}
\label{double-1}
D\psi'_\Omega ((1\otimes b) ( a\otimes 1) ) = D\psi'_\Omega(1\otimes b) D\psi'_\Omega (a\otimes 1).
\end{align}
Further it is enough to show that $b= \theta^-_{j, \widehat \Omega}$ and $a= \theta^+_{i, \widehat \Omega}$ for $i, j\in \bI$. 
The equality (\ref{double-1}) holds for all cases except that $i=j=i_0$. So we only need to check the remaining case. 
Write $\Delta^2=(\Delta\otimes 1) \Delta$ and $\underline v_{i} = (q^{1/2})^{i\cdot i/2}$. 
Fix $\ve\in \{\pm 1\}$. Define $\theta_{i_0, \ve, \Omega} = \theta_{i_+,\Omega} \theta_{i_-,\Omega} - \underline v^{-\ve}_{i_0} \theta_{i_-,\Omega} \theta_{i_+,\Omega}$ so that $\theta_{i_0, \Omega} =
\theta_{i_0, -1, \Omega} $.
Define also $\theta^{\dagger}_{i_0, \ve, \Omega} = \theta_{i_-,\Omega} \theta_{i_+,\Omega} - \underline v^{\ve}_{i_0} \theta_{i_+,\Omega} \theta_{i_-,\Omega}$ so that 
$\theta_{i_0,\Omega} = - \underline v^{-\ve}_{i_0} \theta_{i_0, 1, \Omega}$. 
For simplicity, we shall drop the subscript $\Omega$ in the following computations. By definition, we have 
\begin{align*}
\Delta^2 ( \theta_{i_0,\ve}^+) & =\theta^+_{i_0,\ve} \otimes 1\otimes 1 + \tilde K_{i_0} \otimes \theta^+_{i_0,\ve} \otimes 1 + \tilde K_{i_0} \otimes \tilde K_{i_0} \otimes \theta^+_{i_0,\ve} \\
&+ ( 1 - \underline v^{1-\ve}_{i_0}) \tilde K_{i_0} \otimes \tilde K_{i_+} \theta^+_{i_-} \otimes \theta^+_{i_+} 
+ (\underline  v_{i_0} - \underline  v^{-\ve}_{i_0} ) \tilde K_{i_0} \otimes  \tilde K_{i_-} \theta^+_{i_+} \otimes \theta^+_{i_-}\\
&+ ( 1- \underline v^{1-\ve}_{i_0} )( \tilde K_{i_+} \theta^+_{i_-} \otimes \theta^+_{i_+} \otimes 1 + \tilde K_{i_+} \theta^+_{i_-} \otimes \tilde K_{i_+} \otimes 
\theta^+_{i_+})\\
&+ ( \underline v_{i_0} - \underline v^{-\ve}_{i_0} ) ( \tilde K_{i_-} \theta^+_{i_+} \otimes \theta^+_{i_-} \otimes 1 + \tilde K_{i_-} \theta^+_{i_+} \otimes \tilde K_{i_-} \otimes \theta^+_{i_-} ) .
\end{align*}
And we have 
\begin{align*}
\Delta^2(\theta^{\dagger - }_{i_0,\ve} ) & = \theta^{\dagger -}_{i_0,\ve} \otimes \tilde K_{-i_0} \otimes \tilde K_{- i_0} + 1\otimes \theta^{\dagger -}_{i_0, \ve} \otimes \tilde K_{-i_0} + 1\otimes 1\otimes \theta^{\dagger -}_{i_0,\ve} \\
& +(\underline v_{i_0} - \underline v^{\ve}_{i_0} )\theta^-_{i_+} \otimes \tilde K_{- i_+} \theta^-_{i_-} \otimes \tilde K_{- i_0} +  (1 - \underline v^{1+\ve}_{i_0} ) \theta^-_{i_-} \otimes K_{- i_-} \theta^-_{i_+} \otimes \tilde K_{- i_0} \\
& + ( \underline v_{i_0} - \underline  v^{\ve}_{i_0} ) ( \theta^{-}_{i_+} \otimes \tilde K_{-i_+} \otimes \tilde K_{-i_+} \theta^-_{i_-} + 1\otimes \theta^{-}_{i_+} \otimes \tilde K_{-i_+} \theta^{-}_{i_-} ) \\
& + ( 1- \underline  v^{1+\ve}_{i_0} ) (\theta^{-}_{i_-} \otimes \tilde K_{-i_-} \otimes \tilde K_{-i_-} \theta^-_{i_+} + 1\otimes \theta^{-}_{i_-} \otimes \tilde K_{-i_-} \theta^{-}_{i_+}) .
\end{align*}
With these formulas, we have 
\begin{align}
\label{Double-commutator}
(1\otimes \theta^{\dagger-}_{i_0,\ve} ) ( \theta^+_{i_0, \ve}\otimes 1) = 
\theta^+_{i_0,\ve} \otimes \theta^{\dagger -}_{i_0,\ve} + 
\varphi(\theta^+_{i_0,\ve}, \theta^{\dagger -}_{i_0,\ve})  \tilde K_{i_0} \otimes 1 + 
\varphi (\theta^+_{i_0,\ve} , \sigma (\theta^{\dagger -}_{i_0,\ve})) 1\otimes \tilde K_{-i_0}. 
\end{align}
Therefore the equality (\ref{double-1}) holds for $i=j=i_0$.  Proposition is thus proved. 
\end{proof}

If $\widehat \Omega$ is of finite type, the algebra $\f_{\widehat \Omega}$ coincides with $\H_{\widehat \Omega}$. Therefore $D\psi_{\Omega}$ is an algebra homomorphism if $\widehat \Omega$ is  of finite type. In general, the map  $D\psi_{\Omega}$  is not an algebra homomorphism.

The reduced Drinfeld double $D_1\H_{\Omega}$ is the quotient of $D\H_\Omega$ by the two-sided ideal generated by $K_\mu\otimes 1-1\otimes K_\mu$ for all $\mu \in Y$. 
Let $D_1\f_\Omega$ be the image of $D\f_{\Omega}$ under the associated quotient map. 
Clearly, the linear map $D\psi_\Omega$ descends to a linear map $D_1\psi_\Omega: D_1 \H_{\widehat \Omega} \to D_1\H_\Omega$, which is an algebra homomorphism if
$\widehat \Omega$ is of finite type.  Moreover the restriction map
\begin{align}
\label{red-D}
D_1\psi'_\Omega: D_1\f_{\widehat \Omega} \to D_1\f_{\Omega}
\end{align} 
is an injective algebra homomorphism. 

\index{$D_1 \H_\Omega$}
\index{$D_1\f_\Omega$}
\index{$D_1\psi_\Omega$}
\index{$D_1\psi'_\Omega$}

\section{Embeddings among Lusztig algebras}

In this section, we show that there is an embedding of Lusztig algebras induced by edge contractions. 
We further study the behaviors of  canonical bases of Lusztig's algebras under such an embedding. 
In addition, we show that they are compatible with respect to multiplications, comultiplications, bar involutions, and inner products. 
Finally, we show that Lusztig algebra is a split subquotient of its higher rank. 

\subsection{Lusztig algebra $\f$}
\label{f}
\index{$\A$}
\index{$[a]$}
\index{$\f$}
\index{$\theta_i$}
\index{$\begin{bmatrix} b\\a\end{bmatrix} $}

Let $\A=\mbb Z[v, v^{-1}]$ be the ring of Laurent polynomials in the variable $v$.
For each integer $a\in \mbb N$, we define
$[a]= \frac{v^a-v^{-a}}{v-v^{-1}}$ and $  [a]^! = [a] [a-1] \cdots [1].$
We set $[0]^!=1$. 
For $0\leq a \leq b$, we define 
$
\begin{bmatrix} b\\a\end{bmatrix} = \frac{[b]^!}{[a]^! [b-a]^!}.
$
For each $i\in I$, we write $v_i=v^{\frac{i\cdot i}{2}}$. Let
$[a]_i$, $[a]^!_i$ and $\begin{bmatrix} b\\ a\end{bmatrix}_i$ be
the polynomials obtained from the respective ones without subscript $i$  by substituting $v$ by $v_i$. 
Let $\mbb Q(v)$ be its field of fractions.

To a Cartan datum $(I, \cdot)$, one can associate a Lusztig algebra $\f\equiv \f_{I}$, which turns out to be isomorphic to the negative/positive half of the Drinfeld-Jimbo quantum group associated to $(I,\cdot)$.
This is a unital associative algebra over $\mbb Q(v)$ generated by $\theta_i$ for all $i\in I$ and 
subject to the quantum Serre relations
\[
\sum_{r+s=1- 2 i\cdot j/i\cdot i} (-1)^r \theta_i^{(s)} \theta_j \theta_i^{(r)} =0, \quad \forall i \neq j\in I,
\]
where 
$
\theta^{(n)}_i= \frac{\theta_i^n}{[n]^!_i},\quad  \forall n\in \mbb N.
$
The algebra $\f$ is $\mbb N[I]$-graded: 
$
\f = \oplus_{\nu \in \mbb N[I]} \f_{\nu},
$
where $\f_{\nu}$ is spanned by monomials in $\theta_i$ of degree $\nu$.
Elements in $\f_{\nu}$ will be called homogeneous of degree $\nu$. If $x$ is homogeneous, we  write $|x|$ for its degree.
On $\f\otimes \f\equiv \f\otimes_{\mbb Q(v)}\f$, there is an associative algebra structure via the following twisted multiplication
\begin{align*}
(x\otimes y ) (x'\otimes y') = v^{|y|\cdot |x'|} xx' \otimes yy' , \quad \forall x,y,x',y' \ \mbox{homogeneous}.
\end{align*}
When regarded as an algebra, we use the above multiplication for $\f\otimes \f$. 
There exists a unique algebra homomorphism $r: \f \to \f\otimes \f$ defined by $\theta_i \mapsto \theta_i \otimes 1+ 1\otimes \theta_i$ for all $i\in I$. 
There exists a unique $\mbb Q(v)$-valued symmetric bilinear form $(-,-)$ on $\f$ such that 
\begin{itemize}
\item[] $(1,1) =1$,  $(\theta_i, \theta_j) =\delta_{i,j} (1-v^{-2}_i)^{-1}$, and $(xx', y) = (x\otimes x' , r (y)) $ for all $x, x', y\in \f$, 
\end{itemize}
where the form on $\f\otimes \f$ is given by 
$(x\otimes y, x'\otimes y') = (x,x') (y,y')$. 
There exists a unique $\mbb Q$-algebra involution  on $\f$ defined by $v\mapsto v^{-1}$ and $\theta_i\mapsto \theta_i$ for all $i\in I$.
Write $\bar x$ for the application of the involution to $x$.
Let $_{\A} \f$ be the $\A$-subalgebra of $\f$ generated by the elements $\theta^{(n)}_i$ for all $i\in I, n\in \mbb N$.
It's the integral form of $\f$.

Let $\B$ be the canonical basis of $\f$. Up to a sign, this basis can be characterized as follows.
\index{$\B$}
\[
\pm \B=\{ x\in \f| x\in {}_\A \f, \bar x=x, (x, x) \in 1+ v^{-1}\mbf A \}. 
\]
where $\mbf A= \mbb Q[[v^{-1}]]\cap \mbb Q(v)$. 
To remove the sign, we need more information. 
For any $i\in I$ and $n\in \mbb N$, set  $\pm\B_{i, \geq n} = \pm \B \cap \theta^{(n)}_i \f$ and
$\pm \B_{i, n} = \pm \B_{i,\geq n} - \pm\B_{i, \geq n+1}$. 
It is known that for any $b\in \pm\B_{i, n}$  for $n>0$, there is a unique $b_{i, n} \in \pm \B_{i, 0}$
such that $\theta^{(n)}_ib_{i,n} = b+ \mrm{span}_\A\{ b'' | b''\in \pm \B_{i, \geq n+1}\}$. 
We define a function
\[
\sgn: \pm \B \to \{ \pm 1\}
\]
inductively by $\sgn (1) =1$ and $\sgn (b) = \sgn (b_{i,n})$ if $b\in \pm \B_{i, n}$ and $n>0$. 
Then we have 
\begin{align}
\label{B}
\B = \sgn^{-1}(\{1\})
=\{ x\in \f| \sgn(x)=1, x\in {}_\A \f, \bar x=x, (x, x) \in 1+ v^{-1} \mbf A\}. 
\end{align}

\subsection{The algebras $\f^i$ and ${}^i\f$}
\label{fi}

\index{$\f^i,$ $^i\f$}
We recall the definition of $\f^i$ and some results from~\cite[Chapter 38] {L10} and~\cite{L96}.
In this section, we fix an element $i\in I$.
There exists a $\mbb Q(v)$-linear map $r_i: \f\to \f$ defined by
$r(x)= r_i(x) \otimes \theta_i  \ \mbox{plus other bi-homogeneous terms}$, for any homogeneous element $x\in \f$.
It can be characterized by the conditions
$r_i(1)=0$, $r_i (\theta_j) =\delta_{ij}$ and $r_i(xy) = v^{i\cdot |y|} r_i(x) y + x r_i(y)$ for any homogeneous elements $x$ and $y$ in $\f$.
We define 
\begin{align}
\f^i =\{ x\in \f| r_i(x) =0\}. 
\end{align}
Then $\f^i$ is a subalgebra of $\f$. 
Let
\[
f'(i,j;m) =\sum_{r+s=m} (-1)^r v^{r i\cdot j } v^{r(m-1)}_i \theta_i^{(s)} \theta_j \theta^{(r)}_i, \forall j\in I-\{i\}, m\in \mbb N. 
\]
The algebra $\f^i$ is generated by the elements $f'(i,j, m)$ for various $j\in I-\{i\}$ and $m\in \mbb N$. 
Let $_{\A} \f^i= \f^i\cap {}_{\A} \f$. 
We know that  $\f=\f^i\oplus \f \theta_i$ as $\mbb Q(v)$-vector space.
Let
\begin{align}
\pi^i: \f \to \f^i
\end{align}
be the canonical projection whose kernel is $\f \theta_i$.
The bar involution leaves $\f\theta_i$ stable and hence it induces a bar involution on $\f^i$, denoted by ``$\longrightarrow$''.
We have 
$\pi^i(\bar x)= \overset{\longrightarrow}{\pi^i(x)}$ for all $x\in \f$.
Observe that $\f^i$ and $\f \theta_i$ are orthogonal with each other with respect to the bilinear form $(-,-)$ on $\f$. 
Hence the bilinear form on $\f$ induces a non-degenerate symmetric bilinear form on $\f^i$, denoted by $(-, -)_i$.
We know that  $r(\f) \subseteq \f \otimes \f^i$. We define a $\mbb Q(v)$-linear map $r^i: \f^i \to \f^i\otimes \f^i$ by
$r(x) - r^i(x) \in \f\theta_i \otimes \f^i$ for all $x\in \f^i$. We have
\[
(x, yz)_i= (r^i(x), y\otimes z)_i,\forall x, y, z\in \f.
\]
Let $\B^i = \pi^i ( \B- \B\cap \f\theta_i)$ be the canonical  basis of $\f^i$. 
The element in $\B^i$ can be characterized up to sign as follows:
$\pm \beta \in \B^i$ if and only if 
$\beta\in {}_{\A}\f^i$,
$\overset{\longrightarrow} \beta = \beta$, and
$(\beta,\beta)_i \in 1+ v^{-1} \mbf A .$ 

Similarly, there exists a linear map $_ir: \f\to \f$ defined by
$_i r (1)=0$, $_i r(\theta_j) =\delta_{ij}$ and $_i r(xy) = {}_ir(x) y + v^{i\cdot |x|} x {}_i r(y)$ for any homogeneous elements $x$ and $y$ in $\f$.
We set
\[
{}^i \f =\{ x\in \f | {}_ir(x)=0\}.
\]
We have a decomposition $\f = \theta_i \f \oplus  {}^i\f$. 
Let ${}^i\pi: \f\to {}^i\f$ be the canonical projection.
Let ``$\longleftarrow$'' be the bar involution on ${}^i\f$ induced by the one on $\f$.
Let $_i(-,-)$ be the bilinear form on ${}^i\f$ induced from the bilinear form on $\f$. 
Let $_\A {}^i\f={}^i\f \cap {}_\A\f$. 
Let ${}^i r: {}^i\f\to {}^i \f\otimes {}^i\f$ be the linear map defined by
$r(x) -{}^i r(x) \in  {}^i\f\otimes \theta_i\f$ for all $x\in \f$. 
Let ${}^i \B ={}^i\pi (\B- \B\cap \theta_i\f)$ be the canonical basis of ${}^i\f$.
The element in ${}^i\B$ can be characterized up to sign as follows:
$\pm \beta \in {}^i\B$ if and only if 
$\beta\in {}_{\A}{}^i\f$,
$\overset{\longleftarrow} \beta = \beta$, and
${}_i(\beta,\beta) \in 1+ v^{-1}\mbf A .$ 

\subsection{The embeddings $\psi_\ve$ and $\psi^{\dagger}_\ve$}

Recall from Section~\ref{Cartan} that $(\bI,\cdot)$ is the edge contraction of $(I,\cdot)$ along the pair $\{ i_+,i_-\}$.
Let $\ve \in\{ \pm 1\}$.
Consider the following elements in $\f_I$. 
\begin{align*}
\theta_{i_0,\ve} & = \theta_{i_+} \theta_{i_-} - v^{-\ve}_{i_0} \theta_{i_-} \theta_{i_+}, &&
\theta^{\dagger}_{i_0,\ve}  = \theta_{i_-} \theta_{i_+} - v^{\ve}_{i_0}  \theta_{i_+} \theta_{i_-}.
%
\end{align*}
Now that we have $\theta^{\dagger}_{i_0,\ve} = - v^{\ve}_{i_0} \theta_{i_0,\ve}$. 

To avoid ambiguities, we write $\widehat \theta_j$, $\forall j\in \bI$, for the generators in $\f_{\bI}$.
We have the following embeddings.

\begin{thm}
\label{f-emb}
Let $\ve \in \{ \pm 1\}$.
The assignments $\widehat \theta_i \mapsto \theta_i$ if $i\in \bI - \{i_0\}$ and $\widehat \theta_{i_0} \mapsto \theta_{i_0,\ve}$ 
(respectively, $\widehat \theta_{i_0} \mapsto \theta^{\dagger}_{i_0, \ve}$)
define an algebra embedding 
\begin{align}
\label{f-emb-a}
\psi_{\ve}: \f_{\bI} \to \f_I \quad (\mbox{respectively,}\ \psi^{\dagger}_\ve : \f_{\bI} \to \f_I)  . 
\end{align} 
Moreover we have
$\psi_{\ve} ({}_{\A}\f_{\bI}) \subseteq {}_{\A}\f_I$ and $\psi^{\dagger}_\ve ({}_\A \f_{\bI}) \subseteq {}_\A\f_I$.
\end{thm}

\index{$\psi_\ve$}
\index{$\psi_\ve^{\dagger}$}

\begin{proof}
Let $\ve=1$. It is sufficient to show that $\psi_{\ve}$ is an algebra homomorphism for $v=q^{1/2}$, 
where $q=p^e$  and $p$ a prime for infinitely many $e$.  
Recall $\f_{\Omega}$ from Section~\ref{Hall-thm}. 
There is an isomorphism $\f_I|_{v=q^{1/2}} \to \f_{\Omega}$ defined by 
$\theta_i\mapsto \theta_{i,\Omega}$ for all $i\in I$, where 
 $\f_I|_{v=q^{1/2}}$ is the specialization of $\f_I$ to $v=q^{1/2}$.
So it is enough  to show that the map  
$\psi_{\ve, q} : \f_{\widehat \Omega} \to \f_{\Omega}$ defined by 
$\theta_{i,\widehat \Omega} \mapsto \theta_{i,\Omega}$ for all $i\in \bI$ is an algebra embedding, which is guaranteed by Lemma~\ref{f-comp}. 
Therefore $\psi_{\ve}$ is  an algebra embedding for $\ve=1$. 

Note that  the composition ${}^- \circ \psi_1\circ {}^-$ is an algebra homomorphism and the evaluation on the generators 
of $\f_{\bI}$ is exactly the rule defined by $\psi_{-1}$. Hence we have that $\psi_{-1}$  is an algebra homomorphism. 

By applying a similar argument we obtain that $\psi^{\dagger}_{\ve}$ is an algebra homomorphism, by switching the role of $i_+$ and $i_-$. 

Note that  for any $n\in \mbb N$,  we have 
\begin{align}
\label{theta-0}
\begin{split}
\theta_{i_0,\ve}^{(n)}   =\sum_{\ell=0}^n (-1)^\ell v^{-\ve \ell }_{i_0} \theta^{(\ell)}_{i_-} \theta^{(n)}_{i_+} \theta^{(n-\ell)}_{i_-}, \quad 
\theta^{\dagger (n)}_{i_0,\ve} = \sum_{\ell=0}^n (-1)^\ell v^{\ve \ell}_{i_0} \theta^{(\ell)}_{i_+} \theta^{(n)}_{i_-} \theta^{(n-\ell)}_{i_+}.
%
\end{split}
\end{align}
This implies that $\psi_{\ve}( {}_{\A}\f_{\bI}) \subseteq {}_{\A}\f_I$ and $\psi^{\dagger}_{\ve} ({}_\A \f_{\bI} ) \subseteq {}_\A\f_I$.
The theorem is thus proved.
\end{proof}

Note that we have 
\begin{align}
\label{p-bar}
{}^-\circ \psi_{\ve} = \psi_{-\ve} \circ {}^-. 
\end{align}
Now we compare the restrictions. 
For a triple $\nu, \tau, \omega \in \mbb N[I]$ such that $\tau+\omega=\nu$, we denote 
$r^\nu_{\tau,\omega}: \f_{\nu, I} \to \f_{\tau, I}\otimes \f_{ \omega, I}$ be the restriction of $r:\f_I \to \f_I \otimes \f_I$ to the respective homogeneous components.
We write $\widehat r^{\nu}_{\tau,\omega}$ for the analogous one in $\f_{\bI}$.
Let $\psi_{\ve,\nu}: \f_{\nu,\bI} \to \f_{\nu, I}$ be the restriction of $\psi_{\ve}$ to the homogeneous component $\f_{\nu, \bI}$ of $\f_{\bI}$ to the homogeneous component $\f_{\nu, I}$ of $\f_I$. 
We have the following compatibility.

\begin{prop}
\label{f-restriction-comp}
For any triple $ \nu,\tau , \omega \in \mbb N[\bI]\subseteq \mbb N[I]$ such that 
$\tau+ \omega=\nu$. We have 
\begin{align*}
r^\nu_{\tau, \omega} \circ  \psi_{\ve,\nu} &= (\psi_{\ve,\tau} \otimes \psi_{\ve, \omega}) \circ \widehat r^{\nu}_{\tau,\omega}, && \mbox{if}\ \ve =1,\\
({}^- \circ r^{\nu}_{\tau,\omega} \circ {}^-) \circ \psi_{\ve,\nu} & = (\psi_{\ve,\tau} \otimes \psi_{\ve,\omega}) \circ ({}^- \circ \widehat r^{\nu}_{\tau,\omega} \circ {}^-) , &&  \mbox{if} \ \ve=-1,\\
r^\nu_{\tau, \omega} \circ  \psi^{\dagger}_{\ve,\nu} &= (\psi^{\dagger}_{\ve,\tau} \otimes \psi^{\dagger}_{\ve, \omega}) \circ \widehat r^{\nu}_{\tau,\omega}, && \mbox{if}\ \ve =-1,\\
({}^- \circ r^{\nu}_{\tau,\omega} \circ {}^-) \circ \psi^{\dagger}_{\ve,\nu} & = (\psi^{\dagger}_{\ve,\tau} \otimes \psi^{\dagger}_{\ve,\omega}) \circ ({}^- \circ \widehat r^{\nu}_{\tau,\omega} \circ {}^-) , &&  \mbox{if} \ \ve=1.
\end{align*}
%
%
%
\end{prop}

\begin{proof}
Assume that $\ve=+1$. 
As in the proof of Theorem~\ref{f-emb}, we only need to check that the diagram commutes in the case when $v=q^{1/2}$. In this case, 
the map $r$ (resp. $\widehat r$) becomes
$r_\Omega$ (resp. $r_{\widehat \Omega}$) and the commutativity is due to Thereom~\ref{Hall-emb} and (\ref{mu-generator}). This finishes the proof for $\ve=+1$. For $\ve=-1$, it is a consequence of the statement of $\ve=+1$ and the fact (\ref{p-bar}). 
The proposition is proved. 
\end{proof}

The operator  ${}^- \circ r^{\nu}_{\tau,\omega}\circ {}^-$ is a summand of the operator $\bar r$ in~\cite[1.2.10]{L10}. 
Note  that 
\[
\theta_{i_0,\ve}=
\begin{cases}
f'(i_+, i_-; 1)&\mbox{if}\ \ve =1,\\
\overline{ f'(i_+,i_-; 1)}& \mbox{if}\ \ve =-1.
\end{cases}
\theta^{\dagger}_{i_0,\ve}=
\begin{cases}
f'(i_-, i_+; 1)&\mbox{if}\ \ve =-1,\\
\overline{ f'(i_-,i_+; 1)}& \mbox{if}\ \ve =1.
\end{cases}
\]
So $\psi_{\ve} (\f_{\bI}) \subseteq \f^{i_+}_I$ if $\ve =1$. 
Note that ${}_{i_-} r( \theta_{i_0,\ve}) =0$ for $\ve =1$, and thus we have $\psi_{\ve} (\f_{\bI}) \subseteq {}^{i_-}\f_I$.  Summing up the above analysis,  we have
\begin{align*}
\psi_{\ve} (\f_{\bI}) & \subseteq {}^{i_-}\f_I \cap \f^{i_+}_I \ \mbox{and} \
\psi_{\ve} (\f_{\bI}) \subseteq \overline{{}^{i_+}\f_I\cap \f^{i_-}_I}, && \mbox{if}\ \ve =1,\\
\psi_{\ve} (\f_{\bI}) & \subseteq {}^{i_+}\f_I \cap \f^{i_-}_I \ \mbox{and} \
\psi_{\ve} (\f_{\bI}) \subseteq \overline{{}^{i_-}\f_I\cap \f^{i_+}_I}, &&  \mbox{if}\ \ve =-1.
\end{align*}

Now we address the compatibility of bar involutions.
Observe that the embedding $\psi_{\ve}$ is not compatible with the bar involutions on $\f_I$ and $\f_{\bI}$.
However by composing with the projections $\pi^i$ for $i=i_+, i_-$, we will be able to restore the compatibility. 

\begin{lem}
\label{bar-comp}
Let $\ve \in \{\pm 1\}$. The composition $\pi^{i_+} \psi_{\ve}$ (resp. ${}^{i_-} \pi \psi_{\ve}$) is compatible with the bar involutions 
on $\f_{\bI}$ and $\f^{i_+}_I$ (resp. ${}^{i_-}\f_I$).
The composition $\pi^{i_-} \psi^{\dagger}_{\ve}$ (resp. ${}^{i_+} \pi \psi^{\dagger}_{\ve}$) is compatible with the bar involutions 
on $\f_{\bI}$ and $\f^{i_-}_I$ (resp. ${}^{i_+}\f_I$).
\end{lem}

\begin{proof}
We only need to show that the generator $\theta_{i_0,\ve}$ is bar invariant in the algebras $\f^{i_+}_I$ and ${}^{i_-}\f_I$ with respect to 
$\longrightarrow$ and $\longleftarrow$, respectively.
For $\f^{i_+}_I$, we have
\[
\overset{\longrightarrow}{\theta_{i_0,\ve}} = 
\pi^{i_+} (\overline{\theta_{i_0,\ve}}) =
\pi^{i_+} ( \theta_{i_+} \theta_{i_-} - v^{\ve}_{i_0}\theta_{i_-}\theta_{i_+} )
=\pi^{i_+} (\theta_{i_0,\ve} + (v^{-\ve}_{i_0} - v^{\ve}_{i_0}) \theta_{i_-} \theta_{i_+})
=\theta_{i_0,\ve}.
\]
In a similar manner, for ${}^{i_-}\f_I$, we have 
\[
\overset{\longleftarrow}{\theta_{i_0,\ve}}
={}^{i_-}\pi (\theta_{i_0,\ve} + (v^{-\ve}_{i_0} - v^{\ve}_{i_0}) \theta_{i_-} \theta_{i_+})
=\theta_{i_0,\ve}.
\]
We are done for the $\psi_\ve$ case.

The case for $\psi^{\dagger}_\ve$ can be proved similarly by switching the role of $i_+$ and $i_-$. 
The lemma is thus proved. 
\end{proof}

Moreover, the bilinear forms are compatible.
There is a bilinear form $\{-,-\}$ on $\f_I$ defined by $\{ x,y\} = \overline{(\bar x,\bar y)}$ for all $x,y\in \f_I$. 
Similarly, there is a bilinear form $\{-,-\}$ on $\f_{\bI}$. 

\begin{prop}
\label{Inner-com-2}
The bilinear forms on $\f_{\bI}$ and $\f_I$ are compatible in the following ways. For all $x, y \in \f_{\bI}$, we have
\begin{align*}
(x, y) &= (\psi_\ve(x), \psi_{-\ve}(y)) = ( \psi^{\dagger}_\ve (x), \psi^{\dagger}_{-\ve} (y))    &&  \forall \ve\in \{\pm 1\},\\
(x,y) & = (\psi_\ve (x) ,\psi_\ve (y) ) = (\psi^{\dagger}_{-\ve}(x),\psi^{\dagger}_{-\ve}(y))  && \mbox{if} \ \ve =1,\\
\{x,y\} & = \{ \psi_\ve (x),\psi_\ve (y) \} = \{\psi^{\dagger}_{-\ve}(x), \psi^{\dagger}_{-\ve}(y)\} && \mbox{if} \ \ve =-1. 
\end{align*} 
Moreover for any homogeneous $x, y\in \f_{\nu,\bI}$, we have 
\begin{align*}
(x, y) & = 
v^{2 \nu_{i_0}}_{i_0}   (\psi_{\ve} (x), \psi_{\ve} (y)) = v^{2\nu_{i_0}}_{i_0} (\psi^{\dagger}_{-\ve}(x), \psi^{\dagger}_{-\ve}(y))  &&  \mbox{if} \ \ve =-1,\\
\{x,y\} & = v^{- 2\nu_{i_0}}_{i_0} \{ \psi_\ve(x), \psi_\ve (y) \} = v^{-2\nu_{i_0}}_{i_0} \{ \psi^{\dagger}_{-\ve}(x), \psi^{\dagger}_{-\ve}(y)\} && \mbox{if} \ \ve=1.
\end{align*}
\end{prop}

\begin{proof}
We only show the statements related to the morphism $\psi_\ve$. The statement related to $\psi^{\dagger}_\ve$ can be shown similarly. 
First we show the second equality in the proposition. We observe that we can prove the result when $v$ is specialized to $q^{1/2}$ for infinitely many $q$ of prime powers. In the case, 
it is a consequence of Proposition~\ref{Inner-com}. 
When $\ve=-1$, we have
\begin{align*}
\{x,y\} & = \overline{(\bar x,\bar y)} \\
& = \overline{ ( \psi_{-\ve} (\bar x), \psi_{-\ve}(\bar y)) }  && (\mbox{Second equality}) \\
& = 
\overline{  ( \overline{ \psi_\ve (x)} , \overline{\psi_\ve (y)} )}  && (\ref{p-bar}) \\
& = \{ \psi_\ve(x),\psi_\ve (y)\}, && \forall x,y\in \f_{\bI}. 
\end{align*}
So the third equality holds. 
The fourth (resp. fifth) one is a consequence of the second (resp. third) one and the fact that $\theta_{i_0,-1} =- v_{i_0} \theta_{i_0, +1}$.  

Finally, we show the first equality. 
It is enough to show that the equality holds when $x$ and $y$ are monomials in $\theta_i$ for $i\in \bI$.
We shall prove the equality by induction on the degree of $x$.
The equality clearly holds if $x=y= \theta_i$ for  all $i\in \bI-\{i_0\}$. If $x=y =\theta_{i_0}$, then
\begin{align*}
(\psi_\ve (\theta_{i_0}), \psi_{-\ve} (\theta_{i_0}) )& = (\theta_{i_0,\ve},\theta_{i_0,-\ve})\\
&= (\theta_{i_+} \theta_{i_-}, \theta_{i_+}\theta_{i_-}) + (\theta_{i_-}\theta_{i_+},\theta_{i_-}\theta_{i_+})  - (v_{i_0} + v^{-1}_{i_0}) (\theta_{i_+}\theta_{i_-},\theta_{i_-}\theta_{i_+}) \\
&=\frac{2}{(1-v^{-2}_{i_0})^2} - \frac{(v_{i_0} + v^{-1}_{i_0} ) v^{-1}_{i_0}}{(1-v^{-2}_{i_0})^2} \\
& = \frac{1}{1-v^{-2}_{i_0}} = (\theta_{i_0}, \theta_{i_0}).
\end{align*}
So we have
\begin{align}
\label{Inner-com-2a}
(\theta_i,\theta_i) = (\psi_\ve (\theta_i),\psi_{-\ve}(\theta_i) ,\quad \forall i\in \bI.
\end{align}
Assume now that $x= \theta_{i_1} \cdots \theta_{i_n}\in \f_{\nu,\bI}$. If $\ve =-1$, then we have
\begin{align*}
(\psi_\ve(x), \psi_{-\ve}(y)) & = (\psi_\ve(\theta_{i_1}) \psi_\ve (\theta_{i_2}\cdots \theta_{i_n}), \psi_{-\ve}(y))\\
& = (\psi_\ve(\theta_{i_1}) \otimes  \psi_\ve (\theta_{i_2}\cdots \theta_{i_n}), \hat r \psi_{-\ve}(y)) \\
&= ( \psi_\ve(\theta_{i_1}) \otimes  \psi_\ve (\theta_{i_2}\cdots \theta_{i_n}),  \hat r^{\nu}_{i_1,\nu-i_1} \psi_{-\ve}(y))\\
&=   ( \psi_\ve(\theta_{i_1}) \otimes  \psi_\ve (\theta_{i_2}\cdots \theta_{i_n}), (\psi_{-\ve}\otimes \psi_{-\ve}) r^{\nu}_{i_1,\nu-i_1} (y))
&& \mbox{(Proposition~\ref{f-restriction-comp})}\\
&= (\theta_{i_1} \otimes \theta_{i_2}\cdots \theta_{i_n}, r^{\nu}_{i_1,\nu-i_1}(y)) &&  \mbox{(Induction and (\ref{Inner-com-2a}))}\\
& = (x, y).
\end{align*}
If $\ve=1$, we have 
\begin{align*}
(x,y) = (y,x) \overset{(\star)}{=} (\psi_{-\ve}(y),\psi_{-(-\ve)}(x))= (\psi_\ve(x) ,\psi_{-\ve}(y)),
\end{align*}
where we apply the result for $\ve=-1$ in the step ($\star$). 
By induction we see that the second equality in the proposition holds.
The proposition is thus proved. 
\end{proof}

Now we can compare the canonical bases $\B_{\bI}$ (\ref{B}) and $\B^{i_\pm}_I$ in Section~\ref{fi} of $\f_{\bI}$ with $\f^{i_\pm}_I$, respectively.

\begin{thm}
\label{B-emb}
There is 
$\pi^{i_+} \psi_{\ve} (\B_{\bI}) \subseteq  \B^{i_+}_I$ and ${}^{i_-} \pi \psi_{\ve} (\B_{\bI}) \subseteq  {}^{i_-} \B_I$. 
Similarly, there is 
$\pi^{i_-} \psi^{\dagger}_{\ve} (\B_{\bI}) \subseteq   \B^{i_-}_I$ and ${}^{i_+} \pi \psi^{\dagger}_{\ve} (\B_{\bI}) \subseteq   {}^{i_+} \B_I$.
\end{thm}

\begin{proof}
By Theorem~\ref{f-emb}, Lemma~\ref{bar-comp} and the second equality in Proposition~\ref{Inner-com-2}, we have
\[
\pi^{i_+} \psi_{\ve} (\B_{\bI}) \subseteq  \pm \B^{i_+}_I.
\]
Now we remove the sign, which we shall prove by induction with respect to the homogenous degree of a canonical basis element. 
Clearly we have $\pi^{i_+} \psi_{\ve} (1)=1\in \B^{i_+}_I$. 
For any $\nu\in \mbb N[\bI]$ such that $\nu\neq 0$, we assume that 
$ \pi^{i_+} \psi_{\ve} (b') \in \B^{i_+}_I $ for any $b'\in \B_{\bI}$ of homogeneous degree  $\omega < \nu$.
Now assume that $b \in \B_{\bI}$ is of homogeneous degree $\nu$. 
Let $c= \pi^{i_+} \psi_{\ve} (b)$. 
Then we have $c \in \pm \B^{i_+}_I$. We want to show that $c\in \B^{i_+}_I $. 
Since $ \nu \neq 0$, we see that there exists $i\in \bI$ such that $b\in \B_{\bI, i, n}$ for $n>0$. 
So there exists a unique $b'\in \B_{\bI, i, 0}$ such that $\theta^{(n)}_i b'= b+ \mrm{span}_\A \{ b'' | b'' \in \B_{\bI, i, >n}\}$.
Write $c' =  \pi^{i_+} \psi_{\ve} (b) $, then we have
\[
\theta^{(n)}_i c' = c + \mrm{span}_\A \{ c'' | c'' \in \B_{I, i, > n}\} , \mrm{if} \ i\neq i_0. 
\]
By induction assumption, we see that $c'\in \B^{i_+}_I$ and so by the definition of the canonical basis $\B_I$ we have $c\in \B^{i_+}_I$ if $i\neq i_0$. 

Now if $i=i_0$, then 
by using (\ref{theta-0}) and that $\theta^{\ell}_{i_-} \theta^{(n)}_{i_+}\theta^{(n-\ell)}_{i_-} = \theta^{(n-\ell)}_{i_+} \theta^{(n)}_{i_-}\theta^{(\ell)}_{i_+}$, we see that 
$$\theta^{(n)}_{i_0,\ve} = \theta^{(n)}_{i_+} \theta^{(n)}_{i_-} \ (\mrm{mod} \ \f_I \theta_{i_+}).$$
Since $b'\in \B_{\bI, i_0, 0}$, we must have $c'\in \B_{I, i_-, 0}$ in light of the fact that 
the support of its geometric interpretation  has an open dense subset in $\E^{\heartsuit,\F}_{\W,  \Omega'}$, by Lemma~\ref{f-comp}, 
where the dimension vector of $\W$ is the degree of $c'$ and $\Omega'$ is an orientation obtained from $\Omega$ by making any 
vertex in $i_-=[\mbf i_- ]$ a sink. 
So we have 
\begin{align}
\label{sign-0}
\theta^{(n)}_{i_+} \theta^{(n)}_{i_-} c' = c + \mrm{span}_\A \{ c'' | c'' \in \B_{I, i_-, > 0}\} \ (\mrm{mod} \ \f_I\theta_{i_+}), \mrm{if} \ i =i_0. 
\end{align}
Since $c'\in \B_{I, i_-, 0}$, there exists $d' \in \B_{I, i_-, n}$ such that 
\[
\theta^{(n)}_{i_-} c' = d' + \mrm{span}_\A \{ d'' | d'' \in \B_{I, i_-, >n}\} . 
\]
If $d'' \in \B_{I, i_-,>n}$, then $\theta^{(n)}_{i_+} d'' \not \in  \B_{I, i_-, 0}$. 
And so thanks to (\ref{sign-0}) we must have 
\[
\theta^{(n)}_{i_+} d' = c+  \mrm{span}_\A \{ c'' | c'' \in \B_{I, i_+, > n}\} .
\] 
Moreover  the element $d'$ must be in $\B_{I, i_+, 0}$, 
otherwise $c\in \B_{I, i_+, >n}$ and in turn $b\in \B_{\bI, i_0, >n}$ a contradiction to $b\in \B_{\bI, i_0, n}$.
Therefore, we must have $c\in \B^{i_+}_I$. By induction, we have shown that $\pi^{i_+} \psi_{\ve} (\B_{\bI}) \subseteq  \B^{i_+}_I$. 

As a consequence of the above proof, we have that for any $b\in \B_{\bI}$, 
$$
\psi_{\ve}{b} = c \ \mrm{mod}\ \f_I \theta_{i_+} \cap \theta_{i_-}\f_I \ \mbox{with} \ c\in \B_I.
$$ 
Therefore we have
${}^{i_-} \pi \psi_{\ve} (\B_{\bI}) \subseteq {}^{i_-}\B_I$. This finishes the proof of the statement in the theorem for $\psi_\ve $. 

The proof in the $\psi^{\dagger}_{\ve}$ case is the same as that of $\psi_{\ve} $ by switching the role of $i_+$ and $i_-$.
This finishes the proof. 
\end{proof}

\subsection{The algebra $\f_{\bI}$ as a split subquotient}

Let $\f_{I, \bI}$ be the subalgebra of $\f_{I}$ generated by 
the elements $\theta_i$, $\forall i \in \bI-\{ i_0\}$, $\theta_{i_+}\theta_{i_-}$ and $\theta_{i_-}\theta_{i_+}$. 
\index{$\f_{I, \bI}$}
We have

\begin{prop}
\label{f-subquot}
The assignments $\theta_i\mapsto\widehat  \theta_i$  for all $i\in \bI -\{ i_0\}$,  $\theta_{i_+}\theta_{i_-}\mapsto \widehat \theta_{i_0}$ and $\theta_{i_-}\theta_{i_+}\to 0$ define a surjective split algebra homomorphism $j^*: \f_{I,\bI} \to \f_{\bI}$. 
Moreover the kernel of $j^*$ is the two-sided ideal generated by $\theta_{i_-} \theta_{i_+}$. 
\end{prop}

\begin{proof}
The existence of $j^*$ is due to Proposition~\ref{f-quot}. 
\end{proof}

The bar operator on $\f_I$ induces an operator on $\f_{I, \bI}$, still denoted by the same notation. We note that the algebra homomorphism $j^*$ respects the bar operators. 

Let ${}_{\mathcal A} \f_{I, \bI} = {}_{\mathcal A}\f_I \cap \f_{I, \bI}$. We have $j^* ({}_{\mathcal A} \f_{I, \bI} ) ={}_{\mathcal A} \f_{\bI}$.

Let $\B_{I,\bI}$ be the subset of $\B_I$ consisting of all elements, appeared as a summand, in the monomials in $\f_{I, \bI}$. 
Let $\k_{I, \bI}$ be the subspace in $\f_I$ spanned by elements in $\B_{I, \bI}$. 
By definition, we see that $\k_{I, \bI}$ is a subalgebra of $\f_I$. 
Clearly, we have $\f_{I, \bI}\subseteq \k_{I, \bI}$. 
Let $\j_{I, \bI}$ be the two-sided ideal of $\f_{I, \bI}$ generated by $\theta_{i_-}\theta_{i_+}$.
Let $\j'_{I, \bI}$ be the spanned of canonical basis elements appeared in the monomials in $\j_{I, \bI}$. 
It is clear that $\j'_{I, \bI}$ is a two-sided ideal of $\k_{I, \bI}$. 
We have the following commutative diagram.
\begin{align}
\label{iota}
\begin{CD}
\j_{I, \bI} @>>> \f_{I, \bI} @>>> \f_{I, \bI} /\j_{I, \bI}\\
@VVV @VVV @VV \iota V \\
\j'_{I, \bI} @>>> \k_{I, \bI} @>>> \k_{I, \bI}/\j'_{I, \bI} 
\end{CD}
\end{align}
We have 

\begin{prop}
The morphism $\iota$ in (\ref{iota}) is an isomorphism.  In other words, we have an isomorphism $\f_{\bI} \cong \k_{I, \bI} / \j'_{I, \bI}$. 
\end{prop}

\begin{proof}
To show that $\iota$ is injective, it is enough to show that the square on the left in the diagram (\ref{iota}) is cartesian. 
This amounts to show that if $x\in \j'_{I, \bI} \cap \f_{I, \bI}$, then $x\in \j_{I, \bI}$. Now that $x\in \f_{I, \bI}$ means that $x$ can be written as a linear sum
of monomials in $\theta_{i}$, $i\in \bI-\{i_0\}$ and $\theta_{i_0, \ve}$  plus a linear sum of monomials  having $\theta_{i_-}\theta_{i_+}$. 
The first sum must be zero by mirroring them to functions in $\f_{\Omega}$ which are supported on $\E^{\heartsuit}_{\V,\Omega}$. 
So $x$ must be in $\j_{I, \bI}$. 

To show that $\iota$ is surjective, we recall that $\f_{\bI}\cong \f_{I, \bI} / \j_{I, \bI}$. If $b \in \B_{I, \bI}$ such that $b\not\in \j'_{I, \bI}$, then, $b$ is mapped to a canonical basis element in $\f_{\bI}$, which is a linear sum, say $S$,  of $\theta_{i}$ for $i\in \bI$. 
The corresponding element $S'$ in $\f_{I, \bI}$ is equal to $b$ plus an element in $\j'_{I, \bI}$, and so  gets sent to $b$ via $\iota$.  Thus $\iota$ is surjective. This finishes the proof. 
\end{proof}

For the remaining part of this section, we address when $\f_{I, \bI}$ is equal to $\k_{I, \bI}$. 

\begin{prop}
If $\Gamma$ is a Dynkin graph and $a=1$, then we have $\k_{I, \bI} = \f_{I, \bI}$. 
Furthermore,  $\f_{I, \bI} = \oplus_{\nu \in \mbb N[\bI]} \f_{I, \nu}$ and 
$\B_{I, \bI} = \sqcup_{\nu\in \mbb N[\bI]} \B_{I, \nu}$, where $\B_{I, \nu} = \f_{I, \nu} \cap \B_I$. 
 In this case, $j^* (\B_{I, \bI})= \B_{\bI} \sqcup \{0\}$.
 \end{prop}

\begin{proof}
Let $\nu\in \mbb N[\bI]$. Let $m$ be any monomial in $\f_{I, \nu}$. We want to show that $m\in \f_{I, \bI}$. 
It suffices to show that if 
$m'=\theta_{i_+} \theta_{i_1} \cdots \theta_{i_n} \theta_{i_-}$ such that $ i_1,\cdots,i_n \neq i_+, i_-$, then we have $m'\in \f_{I, \bI}$. 
We shall prove this statement by induction on $n$. If $n=1$, then, by assumption, either $\{i_1, i_+\}$ or $\{i_1, i_-\}$ is disjoint, and so we have
$m' = \theta_{i_1} \theta_{i_+}\theta_{i_-}$ or $\theta_{i_+}\theta_{i_-}\theta_{i_1}$.  Both are in $\f_{I, \bI}$. 
In general, assume that $k$ is the smallest integer such that $i_k$ is either disjoint from $i_+, i_1, \cdots, i_{k-1}$ or 
disjoint from $i_{k+1}, \cdots,   i_n, i_-$. Such a $k$ exists because of the assumption. 
Then we have $m' = \theta_{i_k} \theta_{i_+} \theta_{i_1} \cdots \theta_{i_{k-1}} \theta_{i_{k+1}} \cdots \theta_{i_n} \theta_{i_-}$ or 
$\theta_{i_+} \theta_{i_1} \cdots \theta_{i_{k-1}} \theta_{i_{k+1}} \cdots \theta_{i_n} \theta_{i_-}\theta_{i_k}$. 
So by induction, $m'\in \f_{I, \bI}$.  This implies that $\f_{I, \bI} = \oplus_{\nu\in \mbb N[\bI]} \f_{I, \nu}$. 
Thus we have $\B_{I, \bI} = \cup_{\nu\in \mbb N[\bI]} \B_{I, \nu}$, which in turn implies that $\k_{I, \bI} =\f_{I, \bI}$. 

In the setting of Proposition~\ref{f-quot} we know that $j^*$ sends canonical basis elements to canonical basis elements or zero. 
This finishes the proof. 
\end{proof}

\begin{rem}
In general, we have $\f_{I, \bI} \subsetneq \k_{I, \bI}$. 
\end{rem}

\section{Embeddings among quantum groups}
\label{QG}
In this section, we shall show that there exists an embedding among Drinfeld-Jimbo quantum groups under an edge contraction of a Cartan datum. 
It is conjectured further that there is a split subquotient as in the negative half case. 
We also show that the embedding naturally induces an embedding on the associated modified forms. 
We further show that the embedding is compatible with the inner products and comultiplications. 
Furthermore, we show that the embedding is compatible with the irreducible integrable highest weight modules of dominant highest weight and the canonical bases therein.  Finally, we study the compatibility of tensor products of modules.

\subsection{The embedding $\Psi_\ve$}
\label{Emb-U}

\index{$\Psi_\ve$}

Recall the Cartan datum $(I,\cdot)$ satisfying (\ref{i-comp})  from Section~\ref{Cartan}. 
Let $(Y,X)_I$ be a root datum of $(I, \cdot)$ in Section (\ref{Edge-root}).
Let $\U$ be the Drinfeld-Jimbo quantum group associated with the root datum $(Y, X)_I$ in~\cite[3.1.1]{L10}.
Precisely, $\U$ is a unital associative algebra over $\mbb Q(v)$ defined by  a generator-relation presentation: 
the generators are 
$
E_i, F_i \ \mbox{and} \ K_\mu, \ \forall i\in I, \mu \in Y,
$
and   the defining relations are the following relations (\ref{U1})--(\ref{U6}). 
\begin{align}
\label{U1}
\tag{U1}
&K_\mu K_{\mu'} = K_{\mu+ \mu'},  && \forall \mu, \mu'\in Y.\\
\label{U2}
\tag{U2}
& K_\mu E_i = v^{\langle \mu, i'\rangle} E_i K_\mu, && \forall i\in I, \mu\in Y.\\
\label{U3}
\tag{U3}
& K_\mu F_i = v^{-\langle \mu, i'\rangle} F_i K_\mu,  && \forall i\in I, \mu\in Y.\\
\label{U4}
\tag{U4}
& E_i F_j - F_j E_i = \delta_{ij} \frac{\tilde K_i - \tilde K^{-1}_{i} }{v_i-v^{-1}_i}, &&    \forall i, j\in I.\\
\label{U5}
\tag{U5} 
& \sum_{r+s=1- \langle i, j'\rangle } (-1)^r E_i^{(s)} E_j E_i^{(r)}  =0, &&  \forall i\neq j\in I.\\
\label{U6}
\tag{U6}
& \sum_{r+s=1- \langle i,j'\rangle } (-1)^r F_i^{(s)} F_j F_i^{(r)}=0, && \ \forall i \neq j\in I.
\end{align}
Recall that $ \tilde K_i = K_{\frac{i\cdot i}{2} i}$.

The algebra $\U$ can be equipped with a Hopf algebra structure whose comultiplication is defined by
\[
\Delta ( E_i) = E_i \otimes 1+ \tilde K_i \otimes E_i,
\Delta(F_i) = F_i \otimes \tilde K^{-1}_i + 1\otimes F_i,
\Delta(K_\mu) = K_\mu \otimes K_\mu,  
\]
for all $i\in I, \mu \in Y$.
The counit $\U\to \mbb Q(v)$ is defined by $E_i\mapsto 0$, $F_i \mapsto 0$ and $K_{\mu} \mapsto 1$ for all $i\in I, \mu \in Y$. 
The antipode $S: \U \to \U$ is given by the rule
$S( E_i) = - \tilde K_i^{-1} E_i$, $S(F_i) = - F_i \tilde K_i$ and $S(K_\mu) = K_{-\mu}$ for all $i\in I$ and $\mu \in Y$. 

Let $\U^+$ (resp. $\U^-$) be the subalgebra of $\U$ generated by $E_i$ (resp. $F_i$) for all $i\in I$. 
Let $\U^0$ be the subalgebra generated by $K_\mu$ for all $\mu\in Y$. 
Then there is an isomorphism of vector spaces 
\begin{align}
\label{triangular}
\U^+\otimes \U^0\otimes \U^-\to \U, x\otimes y\otimes z\mapsto xyz.
\end{align}

We have isomorphisms $\f\cong \U^+$, $\theta_i\mapsto E_i$ for all $i\in I$ and $\f\cong \U^-$, $\theta_i\mapsto F_i$ for all $i\in I$.
For any $x\in \f$, we write $x^+$ and $x^-$ for the image of $x$ under the above isomorphisms respectively. 

We shall write $\U_{I}$ for $\U$ when we need to emphasize the dependence of $I$.

Recall from Section~\ref{Cartan} that $(\bI,\cdot)$ be the edge contraction of $(I,\cdot)$ along the pair $\{ i_+,i_-\}$.
Let $(Y, X)_{\bI}$ be the edge contraction of $(Y, X)_{I}$ along $\{ i_+, i_-\}$ in Section~\ref{Edge-root}. 
Let $\U_{\bI}$ be the quantum group associated with the root datum $(Y, X)_{\bI}$.
To avoid ambiguities, we put a hat on the generators in $\U_{\bI}$, i.e.,  $\widehat E_i, \widehat F_i , \widehat K_\mu$. Notice that 
we have $\U^0_{\bI} = \U^0_I$.

We define the following elements in $\U_I$.  \index{$E_{i_0,\ve}$} \index{$F_{i_0,\ve}$}
\begin{align*}
E_{i_0,\ve}  =
E_{i_+} E_{i_-} - v^{-\ve}_{i_0} E_{i_-} E_{i_+} \ \mbox{and}\
 F_{i_0,\ve}  = F_{i_-} F_{i_+} - v^{\ve}_{i_0} F_{i_+} F_{i_-}, \ \mbox{where} \ \ve\in \{\pm 1\}.
 \end{align*}
 Note that we already have 
$ K_{i_0}  = K_{i_+} K_{i_-}$ by definition.
We have the following embedding among quantum groups. 
We shall provide a more direct proof. 

\begin{thm}
\label{Psi-U}
Let us fix $\ve\in \{\pm 1\}$.
Let $(\bI,\cdot)$ be the edge contraction of the Cartan datum $(I,\cdot)$.
Then there exists an algebra embedding 
\begin{align*}
\Psi_\ve: \U_{\bI} \to \U_I, \  \widehat E_i \mapsto E_i, \widehat E_{i_0}\mapsto E_{i_0,\ve},  \
\widehat F_i \mapsto F_i, \widehat F_{i_0} \mapsto F_{i_0,\ve}, \ \widehat K_\mu \mapsto K_{\mu}, 
\end{align*}
for all $i\in \bI-\{i_0\} , \mu\in Y$.
\end{thm}

\begin{proof}
We first show that $\Psi_\ve$ is an algebra homomorphism, i.e., 
the elements  $E_{i_0,\ve}$, $F_{i_0,\ve}$, $E_i, F_i, K_\mu$ for all $i\in \bI-\{i_0\}$ and $\mu \in Y$ satisfy the defining relations of $\U_{\bI}$, denoted by 
$(\mrm Ua)_{\bI}$ for $a=1,\cdots, 6$.
Note that there exists an isomorphism $\f_I\to \U^+_I$ (resp. $\f_I\to \U^-_I$) defined by  $\theta_i\mapsto E_i$ 
(resp. $\theta_i\to F_i$), for all $i\in I$. 
So  the  relations $(\ref{U5})_{\bI}$ and $(\ref{U6})_{\bI}$  are due to the fact that $\psi_{\ve}$ is an embedding in Theorem~\ref{f-emb}.
Since $\U^0_{\bI}=\U^{0}_I$, the $K_\mu$s satisfy the defining relation $(\ref{U1})_{\bI}$ automatically.

For $(\ref{U2})_{\bI}$, the relation is satisfied by definition except the case $i=i_0$. 
Note that 
$$
K_\mu E_{i_+}E_{i_-} = v^{\langle \mu, i'_+\rangle + \langle \mu, i'_-\rangle} E_{i_+} E_{i_-} K_\mu= v^{\langle \mu, i'_0\rangle} E_{i_+} E_{i_-} K_{\mu}.
$$ 
Similarly, $K_\mu E_{i_-} E_{i_+}= v^{\langle \mu, i'_0\rangle} E_{i_-} E_{i_+} K_{\mu}$. 
So we have $K_\mu E_{i_0,\ve} = v^{\langle \mu, i'_0\rangle} E_{i_0,\ve} K_{\mu}$ as desired. Thus the condition $(\ref{U2})_{\bI}$ is verified.
For $(\ref{U3})_{\bI}$, the proof is similar to that of $(\ref{U2})_{\bI}$ once we have
\[
K_\mu F_{i_+} F_{i_-} = v^{-\langle \mu, i'_0\rangle} F_{i_+} F_{i_-} K_\mu \ \mbox{and}\
K_\mu F_{i_-} F_{i_+} = v^{-\langle \mu, i'_0\rangle} F_{i_-} F_{i_+} K_\mu.
\]

It remains to verify $(\ref{U4})_{\bI}$. The relation holds automatically if $i, j\neq i_0$. 
Assume that $i=i_0$ and $j\neq i_0$, then we have 
\begin{align*}
E_{i_0,\ve} F_j-F_jE_{i_0,\ve}  
&= (E_{i_+} E_{i_-} - v^{-\ve}_{i_0} E_{i_-} E_{i_+} ) F_j - F_j ( E_{i_+} E_{i_-} - v^{-\ve}_{i_0} E_{i_-} E_{i_+}) \\
& = E_{i_+} E_{i_-} F_j - F_j E_{i_+} E_{i_-} - v^{-\ve}_{i_0} ( E_{i_-} E_{i_+} F_j - F_j E_{i_-} E_{i_+})\\
&=(E_{i_+}F_j  - F_j E_{i_+}) E_{i_-}  - v^{-\ve}_{i_0}  E_{i_-} (E_{i_+} F_j -F_j  E_{i_+})=
0. 
\end{align*} 
The same argument can be applied to the case that $i\neq i_0$ and $j=i_0$.
The last case to verify is that of $i=j=i_0$. 
We observe that
\begin{align*}
E_{i_0,\ve} = T''_{i_+, \ve} ( E_{i_-}) \ \mbox{and}\
F_{i_0,\ve} = T''_{i_+, \ve} (F_{i_-}) ,
\end{align*} 
where $T''_{i_+, \ve}$ is an automorphism on $\U_I$ defined in~\cite[37.1.3]{L10}. 
So we have
\begin{align*}
E_{i_0,\ve} F_{i_0,\ve} - F_{i_0,\ve} E_{i_0,\ve} 
& = T''_{i_+, \ve} ( E_{i_-} F_{i_-} - F_{i_-} E_{i_-}) \\
& = T''_{i_+, \ve} \left ( \frac{\tilde K_{i_-} - \tilde K^{-1}_{i_-}}{v_{i_-} - v^{-1}_{i_-}} \right ) \\
& =  \frac{\tilde K_{i_0} - \tilde K^{-1}_{i_0}}{v_{i_0} - v^{-1}_{i_0}}. 
\end{align*}
This finishes the proof that the elements $E_i, F_i, K_{\mu}$ for all $i\in \bI-\{i_0\}$ and $\mu \in Y$, together with $E_{i_0,\ve}$ and $F_{i_0,\ve}$, satisfy
the defining relations of $\U_{\bI}$. Therefore, the map $\Psi_\ve$ is an algebra homomorphism. 

Let $\Psi^+_\ve: \U^+_{\bI}\to \U^+_I$ and $\Psi^-_\ve: \U^-_{\bI} \to \U^-_I$ and $\Psi^0_\ve: \U_{\bI} \to \U_{I}$ be the restriction to 
the positive, negative and Cartan parts of $\U_{\bI}$.  By definition $\Psi^0_\ve$ is the identity map. 
Thanks to Theorem~\ref{f-emb}, we see that $\Psi^{+}_\ve$ and $\Psi^{-}_\ve$ are injective. 
Therefore the map $\Psi^+_\ve \otimes \Psi^0_\ve\otimes \Psi^-_\ve: \U^+_{\bI} \otimes \U^0_{\bI} \otimes \U^{-}_{\bI} \to
\U^+_{I} \otimes \U^0_{I} \otimes \U^{-}_{I}$ is injective. Clearly, we have the following commutative diagram
\[
\xymatrix{
\U^+_{\bI} \otimes \U^0_{\bI} \otimes \U^{-}_{\bI}  \ar[d]_{\Psi^+_\ve \otimes \Psi^0_\ve \otimes \Psi^-_\ve} \ar[r] & \U_{\bI} \ar[d]^{\Psi_\ve}\\
\U^+_{I} \otimes \U^0_{I} \otimes \U^{-}_{I} \ar[r]& \U_I
}
\]
where the horizontal maps are multiplication maps from (\ref{triangular}). Thanks to the above diagram, the morphism $\Psi_\ve$ must be an embedding. This finishes the proof. 
\end{proof}

In what follows, we shall show that the embedding $\Psi_\ve$ can be thought of as  a generic version of the algebra homomorphism $D_1\psi'_\Omega$ in (\ref{red-D}). 
Let $\U_I^q$ be the algebra defined in the same way as $\U_I$ but with the ground field replaced by $\mbb Q(q^{1/2})$ and $v$ by $q^{1/2}$. (Recall that $\underline v_i = (q^{1/2})^{i\cdot i/2}$.)
Let $\Phi^q_\ve: \U_{\bI}^q \to \U_I^q$ be the counterpart of $\Phi_\ve$.  
It is known from~\cite{X97} that there is an Hopf algebra embedding $\varkappa_I :\U^q_I \to D_1 \H_\Omega$ defined by 
\[
E_i \mapsto \theta_{i,\Omega}^+ , F_i\mapsto -\underline v^{-1}_i \theta_{i,\Omega}^-, K_i\mapsto K_i, \quad \forall i\in I. 
\]
We have 

\begin{prop}
\label{U-Hall}
Let $\ve = 1$. 
Then we have the following commutative diagram.
\[
\begin{CD}
\U^q_{\bI} @>\varkappa_{\bI}>> D_1\H_{\widehat \Omega}\\
@V\Psi_\ve^qVV @VV  D_1\psi_\Omega V \\
\U^q_I  @> \varkappa_I >> D_1 \H_\Omega. 
\end{CD}
\]
\end{prop}

\begin{proof}
It is enough to show that the diagram is commutative with respect to the Chevalley generators of $\U^q_{\bI}$. This can be checked by the definition. 
\end{proof}

\begin{rem}
Since Proposition~\ref{U-Hall} holds for an arbitrary prime power $q$. By the commutative diagram in Proposition~\ref{U-Hall}, we can deduce that 
$\Psi_\ve $ is an algebra embedding by using the fact that $D_1\psi'_\Omega$ is an algebra embedding when restricting to $D_1\f_{\widehat \Omega}$ (\ref{red-D}).  This provides a second proof of Theorem~\ref{Psi-U}.
\end{rem}

\subsection{Subquotient}
\label{subq}

Let $\U_{I,\bI}$ be the subalgebra of $\U_I$ generated by the elements $E_i, F_i, K_{\mu}$ for $i\in \bI-\{i_0\}$ and $\mu\in Y$, and $E_{i_+} E_{i_-}$, $E_{i_-}E_{i_+}$,
$F_{i_-}F_{i_+}$ and $F_{i_+}F_{i_-}$.  
Let $\mathcal J_{I, \bI}$ be the two-sided ideal  of $\U_{I, \bI}$ generated by $E_{i_-}E_{i_+}$ and $F_{i_+}F_{i_-}$. 
We then have a surjective algebra homomorphism 
\begin{align}
\label{phi}
\Phi: \U_{\bI} \to \U_{I,\bI}/\mathcal J_{I,\bI}
\end{align}
by composing the map $\Psi_\ve$, for a fixed $\ve$, and the canonical quotient map $\U_{I, \bI} \to \U_{I,\bI}/\mathcal J_{I, \bI}$. 
Note that the map $\Phi$ is independent of the choice of $\ve$. 
We have 
\index{$\U_{I, \bI}$} \index{$\mathcal J_{I, \bI}$} \index{$\Phi$}

\begin{conj}
\label{subqa}
The map $\Phi$ is an isomorphism, hence $\U_I$ is a split subquotient of $\U_{\bI}$. 
\end{conj}

In type A, this conjecture can be verified.  For finite types, it is likely that the conjecture can be shown by making use of a variant of the commutative diagram in the following Proposition~\ref{prop:Psi-Schur}. 

\subsection{Further structures}

In the following we understand the compatibility of $\Psi_\ve$ with the bar involutions, the integral forms and the involutions \textomega.

Let $^-$ be the bar involution on $\U_I$ defined by $\bar v=v^{-1}$,  $\bar E_i= E_i$, $\bar F_i=F_i$ and $\bar K_\mu = K_{-\mu}$ for all $i\in I$ and $\mu \in Y$.
Clearly we have $\bar E_{i_0,\ve} = E_{i_0, -\ve}$ and $\bar F_{i_0, \ve}= F_{i_0, -\ve}$. So we have
\[
\Psi_\ve \circ {}^- = {}^- \circ \Psi_{-\ve}. 
\]

Let ${}_\A \U_I$ be the $\A$-subalgebra of $\U_I$ generated by $E^{(n)}_i, F^{(n)}_i$ and $K_{\mu}$ for all $i\in I, n\in \mbb N$ and $\mu\in Y$.
Similarly, we define ${}_\A \U_{\bI}$. Then by Theorem~\ref{f-emb}, we have
\begin{align}
\Psi_\ve ({}_\A \U_{\bI}) \subseteq {}_\A \U_I. 
\end{align}

Let $\text{\textomega}_I$  be an involution on $\U_I$ defined by $E_i\mapsto F_i$, $F_i \mapsto E_i$ and $K_\mu \mapsto K_{-\mu}$ 
for all $i\in I$ and $\mu \in Y$. \index{$\text{\textomega}_I$}  We simply write \textomega \ for $\text{\textomega}_I$ if there is no ambiguity.
The  morphism $\Psi^{\dagger}_\ve= \text{\textomega}_I \Psi_\ve \text{\textomega}_{\bI}: \U_{\bI}\to \U_I$ is an embedding defined by
$\widehat E_i \mapsto E^{\dagger}_{i,\ve}, \widehat F_i \mapsto F^{\dagger}_{i,\ve}, K_{\mu}\to K_{\mu}$ for all $i\in \bI$ and $\mu\in Y$
where $E^{\dagger}_{i,\ve}=E_i$ and $F^{\dagger}_{i,\ve}=F_i$ if $i\neq i_0$ and 
$E^{\dagger}_{i_0,\ve} = E_{i_-}E_{i_+}-v^{\ve}_{i_0} E_{i_+}E_{i_-}$ and $F^{\dagger}_{i_0,\ve} = F_{i_+}F_{i_-} - v^{-\ve}_{i_0} F_{i_-} F_{i_+}$.
Note that $E^{\dagger}_{i_0,\ve} = - v^{-\ve}_{i_0} E_{i_0,\ve}$ and $F^{\dagger}_{i_0,\ve} = - v^{-\ve}_{i_0} F_{i_0,\ve}$. 

\index{$\Psi^{\dagger}_\ve$}

Let $\rho: \U_I \to \U^{opp}_I$ be the algebra isomorphism  defined by
$\rho( E_i) = v_i \tilde K_i F_i$, $\rho(F_i) = v_i \tilde K^{-1}_i E_i$ and $\rho (K_{\mu} ) = K_\mu$ for all $i\in I, \mu \in Y$. 
We write $\rho_I$ for $\rho$ to avoid ambiguities whenever needed. 
By a direct computation, we have
\begin{align}
\label{rho-Psi}
\rho_I \Psi_\ve = \Psi_{-\ve} \rho_{\bI}, \quad \forall \ve\in \{\pm 1\}.
\end{align}

\index{$\rho$}

\subsection{Comultiplications}

If $s\in\U_I$ is a monomial in $E_i, F_i, K_\mu$, we define $||s||$ to be the degree of $s$ where  $||s||\in \mbb Z[I]$ and  
the $i$-th component of $||s||$ is the difference of the number of $E_i$ and the number of $F_i$ in $s$. 
In particular, $||E_i||=i$, $||F_i||=-i$ and $||K_\mu||=0$ for all $i\in I$ and $\mu\in Y$. 
For each $\nu\in \mbb Z[I]$, let $\U_I(\nu)$ be the subspace of $\U_I$ spanned by all monomials of degree $\nu$.
Clearly, in light of the definition, we have $\U_I=\oplus_{\nu\in \mbb Z[I]} \U_I(\nu)$. Further we have
$\Delta_I (\U_I(\nu)) \subseteq \oplus_{\tau, \omega\in \mbb Z[i]: \tau + \omega =\nu} \U_I(\tau) \otimes \U_I(\omega)$. 
Let $$\Delta^{\nu}_{\tau,\omega, I}: \U_I(\nu) \to \U_I(\tau) \otimes \U_I(\omega)$$ be the linear map induced by $\Delta_I$ by restricting to 
$\U_I(\nu)$ and then projection to the component $\U_I(\tau) \otimes \U_I(\omega)$. 

By replacing $\U_I$ by $\U_{\bI}$, there is a well-defined linear map $\Delta^{\nu}_{\tau,\omega, \bI}: \U_{\bI}(\nu) \to \U_{\bI}(\tau)\otimes \U_{\bI}(\omega)$ for any $\nu, \tau, \omega \in \mbb Z[\bI]$ such that $\nu =\tau+\omega$. 
Note that the algebra embedding $\Psi_\ve$ respects the grading on $\U_{\bI}$, i.e., $\Psi_{\ve}(\U_{\bI} (\nu)) \subseteq \U_I(\nu)$ for all $\nu\in \mbb Z[I]$. Let $\Psi_{\ve}|_{\nu}: \U_{\bI}(\nu)\to \U_I(\nu)$ be the restriction of $\Psi_\ve $ to $\U_{\bI}(\nu)$. 

\begin{prop}
\label{Emb-co}
Let $\nu, \tau, \omega \in \mbb Z[\bI]$ such that $\tau +\omega=\nu$. We have
\begin{align}
\Delta^{\nu}_{\tau, \omega, I} \Psi_\ve|_{\nu} = (\Psi_\ve|_{\tau} \otimes \Psi_\ve|_{\omega}) \Delta^{\nu}_{\tau, \omega, \bI} 
\end{align}
\end{prop}

\begin{proof}
By a direct computation, we have 
\begin{align*}
\Delta_I(E_{i_0,\ve}) = E_{i_0,\ve} \otimes 1 + \tilde K_{i_0} \otimes E_{i_0,\ve} + ( 1- v^{1-\ve}_{i_0}) \tilde K_{i_+} E_{i_-} \otimes E_{i_+} +
(v_{i_0} - v^{-\ve}_{i_0}) \tilde K_{i_-} E_{i_+} \otimes E_{i_-}. \\
\Delta_I(F_{i_0,\ve}) = F_{i_0,\ve} \otimes \tilde K^{-1}_{i_0} + 1\otimes F_{i_0,\ve} + (v_{i_0} - v^{\ve}_{i_0} ) F_{i_+} \otimes \tilde K^{-1}_{i_+} F_{i_-} + (1-v^{1+\ve}_{i_0} ) F_{i_-}\otimes \tilde K^{-1}_{i_-} F_{i_+}. 
\end{align*}
If $\ve=1$, the above formula becomes 
\begin{align*}
\Delta_I(E_{i_0,\ve}) & = E_{i_0,\ve} \otimes 1+ \tilde K_{i_0} \otimes E_{i_0,\ve} +
(v_{i_0} - v^{-1}_{i_0}) \tilde K_{i_-} E_{i_+} \otimes E_{i_-}. \\
\Delta_I(F_{i_0,\ve}) & = F_{i_0,\ve} \otimes \tilde K^{-1}_{i_0} + 1\otimes F_{i_0,\ve} + (1-v^{2}_{i_0} ) F_{i_-}\otimes \tilde K^{-1}_{i_-} F_{i_+}. 
\end{align*}
Without the third terms in the above formula, the  statement in the proposition holds automatically. 
Moreover, the third terms does not contribute to the component $\U_I(\tau)\otimes \U_I(\omega)$ for $\tau, \omega \in \mbb Z[\bI]$ because
there is no term $E_{i_-}\otimes -$ showing up in the above formula and as such the statement in the proposition holds for $\ve=1$. The case for $\ve=-1$ can be shown similarly. The proposition is thus proved.
\end{proof}

\subsection{Modified quantum groups}

Recall from~\cite[23.1]{L10} the modified quantum group associated with $\U_I$ is defined to be
\begin{align*}
\dot \U_I&= \bigoplus_{\lambda', \lambda''\in X} {}_{\lambda'} \U^I_{\lambda''},\
 {}_{\lambda'} \U^I_{\lambda''}  = \U_I/ \sum_{\mu\in Y} (K_{\mu} - v^{\langle \mu, \lambda'\rangle}) \U_I + \sum_{\mu\in Y} \U_I 
 ( K_\mu - v^{\langle \mu, \lambda''\rangle} ) .
\end{align*}
Let $\pi_{\lambda',\lambda''}: \U_I\to {}_{\lambda'} \U^I_{\lambda''}$ be the projection map. 
For all monomials $s$ in $\U_I$,
$\pi_{\lambda', \lambda''} (s)=0$ if $||s|| \neq \lambda'-\lambda''$. 
The space $\dot \U_I$ inherits an associative algebra without unit structure from $\U_I$ by defining 
$
\pi_{\lambda'_1, \lambda''_1} (s) \pi_{\lambda'_2,\lambda''_2} (t) = \delta_{\lambda''_1,\lambda'_2} \pi_{\lambda'_1,\lambda''_2} (st)
$
for all monomials $s$ and $t$ such that $||s|| = \lambda'_1-\lambda''_1$ and $||t||= \lambda'_2-\lambda''_2$, for all $\lambda'_1,\lambda''_1,\lambda'_2,\lambda''_2\in X$.
The space $\dot \U_I$ admits a $\U_I$-bimodule structure by setting
$t'. \pi_{\lambda',\lambda''}(s). t'' = \pi_{\lambda' + ||t'||, \lambda''- ||t''||} ( t'st'')$ for all monomials $s, t', t'' $ in $\U_I$. 
Let $1_{\lambda} = \pi_{\lambda, \lambda} (1) $ for all $\lambda\in X$. Then we have elements 
$E_i 1_{\lambda}$ and $F_i1_{\lambda}$ in $\dot \U_I$ for all $i\in I$ and $\lambda\in X$. 
These elements are multiplicative generators of $\dot \U_I$.

By the definition of $\Psi_\ve$, there is an induced linear map
\[
\dot \Psi_\ve: \dot \U_{\bI} \to \dot \U_{I}
\]
such that $\dot \Psi_\ve ( {}_{\lambda'} \U^I_{\lambda''}) \subseteq {}_{\lambda'} \U^{\bI}_{\lambda''}$ for all $\lambda',\lambda''\in X$.
Since $\Psi_{\ve}$ is degree preserving, the linear map $\dot \Psi_\ve$ is an algebra homomorphism. 
It is also compatible with the bimodule structures, i.e., 
\begin{align}
\label{Psi-dot}
\dot \Psi_\ve ( t' . s. t'') = \Psi_\ve (t') . \dot \Psi_\ve (s) . \Psi_\ve(t''), \forall t', t''\in \U_{\bI}, s\in \dot \U_{\bI}.
\end{align}
Moreover we have \index{$\dot \Psi_\ve$}

\begin{prop}
\label{Psi-Udot}
Let $\ve\in \{\pm 1\}$. 
The embedding $\Psi_\ve : \U_{\bI} \to \U_I$ in Theorem~\ref{Psi-U} induces an emebedding
$\dot \Psi_\ve : \dot \U_{\bI} \to \dot \U_I$ such that 
$1_{\lambda} \mapsto 1_{\lambda}$, $\widehat E_i 1_{\lambda} \mapsto E_{i} 1_{\lambda}$, 
$\widehat E_{i_0}1_{\lambda}\mapsto E_{i_0,\ve} 1_{\lambda}$,
$\widehat F_i 1_{\lambda} \mapsto F_{i} 1_{\lambda}$ and 
$\widehat F_{i_0}1_{\lambda}\mapsto F_{i_0,\ve} 1_{\lambda}$
for all $i\in \bI-\{i_0\}$ and $\lambda \in X$. 
\end{prop}

\begin{proof}
The assignments on the multiplicative generators are due to (\ref{Psi-dot}).
It remains to show that $\dot \Psi_\ve$ is an embedding. 
Recall from~\cite[23.2.1]{L10}, the set 
$\{ b^+ 1_{\lambda} b'^-| b, b'\in \mbf B_{\bI}, \lambda \in X\}$ is a basis of $\dot \U_{\bI}$.
By Theorem~\ref{B-emb}, the image of the above set under $\dot \Psi_\ve$ is linearly independent. Hence we have
the embedding claim. The proposition is thus proved. 
\end{proof}

Let ${}_\A \dot{\U}_I$ be the $\A$-subalgebra of $\dot \U_I$ generated by the elements $E^{(n)}_i1_{\lambda}$, $F^{(n)}_i 1_{\lambda}$ for various $n\in \mbb N$, $\lambda \in X$ and $i\in I$. Clearly we have

\begin{align}
\label{Udot-integral}
\dot \Psi_\ve ({}_\A \dot{\U}_{\bI}) \subseteq {}_\A \dot{\U}_I.
\end{align}

The comultiplication $\Delta$ on $\U_I$ induces a linear map 
\[
\Delta^I_{\lambda'_1,\lambda''_1, \lambda'_2,\lambda''_2} : {}_{\lambda'} \U^I_{\lambda''} \to {}_{\lambda'_1}\U^I_{\lambda''_1}\otimes {}_{\lambda'_2} \U^I_{\lambda''_2}, \ \mbox{where} \ \lambda'=\lambda'_1+\lambda'_2, \lambda''=\lambda''_1+\lambda''_2
\]
for all $\lambda'_1,\lambda''_1, \lambda'_2,\lambda''_2\in X$. Note that we have
$\dot \Psi_\ve ({}_{\lambda'} \U^{\bI}_{\lambda''}) \subseteq {}_{\lambda'} \U^I_{\lambda''}$.
Let $\dot \Psi_\ve |_{\lambda', \lambda''} : {}_{\lambda'} \U^{\bI}_{\lambda''} \to {}_{\lambda'} \U^I_{\lambda''}$ be the restriction of $\dot \Psi_\ve$ to the piece ${}_{\lambda'} \U^{\bI}_{\lambda''}$.
By Proposition~\ref{Emb-co}, we have 

\begin{prop}
\label{Emb-dot-co}
Let $\lambda'_1,\lambda''_1, \lambda'_2, \lambda''_2\in X$ such that $\lambda'_1-\lambda''_1$, $\lambda'_2-\lambda''_2 \in \mbb Z[\bI]$.
We have
\[
\Delta^I_{\lambda'_1,\lambda''_1, \lambda'_2, \lambda''_2} \dot \Psi_\ve|_{\lambda'_1+\lambda'_2,\lambda''_1+\lambda''_2}=
(\dot \Psi_\ve|_{\lambda'_1,\lambda''_1}\otimes \dot \Psi_\ve |_{\lambda'_2,\lambda''_2})
\Delta^{\bI}_{\lambda'_1,\lambda''_1, \lambda'_2, \lambda''_2} 
\]
\end{prop}

\begin{proof}
We observe that ${}_{\lambda'} \U^I _{\lambda''}=1_{\lambda'} \U_I(\lambda'-\lambda'') 1_{\lambda''} $. 
The proposition is a consequence of Proposition~\ref{Emb-co}.
Let $\lambda'=\lambda'_1+\lambda'_2$, $\lambda''=\lambda''_1+\lambda''_2$,  
$\tau = \lambda'_1-\lambda''_1$ and $\omega=\lambda'_2-\lambda''_2$ and $\nu =\tau+\omega$. 
Consider the following diagram 
\[ 
\xymatrix{
& \U_I(\nu) \ar[rr]^{\Delta^{\nu}_{\tau,\omega, I}}  \ar@{.>}[dd] & & \U_I(\tau) \otimes \U_I(\omega) \ar[dd] \\
\U_{\bI}(\nu) \ar[rr]^{\Delta^{\nu}_{\tau,\omega, \bI}}  \ar[dd] \ar[ur]^{\Psi_\ve|_{\nu}} && \U_{\bI}(\tau)\otimes \U_{\bI}(\omega) \ar[ur]_{\Psi_{\ve}|_{\tau}\otimes\Psi_\ve|_{\omega}} \ar[dd] \\
& {}_{\lambda'} \U^I_{\lambda''}  \ar@{.>}[rr] && {}_{\lambda'_1} \U^I_{\lambda''_1} \otimes {}_{\lambda'_2}\U^I_{\lambda''_2}\\
 {}_{\lambda'} \U^{\bI}_{\lambda''}  \ar@{.>}[ur] \ar[rr] && {}_{\lambda'_1} \U^{\bI}_{\lambda''_1} \otimes {}_{\lambda'_2}\U^{\bI}_{\lambda''_2} \ar[ur]
}
\]
where the vertical maps are quotient maps and the maps in the bottom square are induced from the corresponding maps in the top square. 
By Proposition~\ref{Emb-co}, the top square is commutative. By the definitions, the side squares are commutative. Therefore the bottom square is commutative, which is exactly the statement in the proposition as desired. 
\end{proof}

\subsection{Bilinear forms}
Recall  $\rho: \U_I \to \U^{opp}_I$ from Section~\ref{Emb-U}.
By~\cite[Theorem 26.1.2]{L10}, there exists a unique  bilinear form $(-,-): \dot \U_I\times \dot \U_I \to \mbb Q(v)$ satisfying the following conditions (Ba)-(Bc). 
\begin{itemize}
\item[(Ba)]
$( 1_{\lambda_1} x1_{\lambda_2},  1_{\lambda'_1} x' 1_{\lambda'_2}) =0$ for all $x, x'\in \U_I$  if either $\lambda_1\neq \lambda'_1$ or $\lambda_2\neq \lambda'_2$.

\item[(Bb)]  $(ux, y) = (x, \rho(u)y)$ for all $x, y\in \dot \U_I$ and $u\in \U_I$.

\item[(Bc)]  $(x^- 1_{\lambda}, x'^- 1_{\lambda} ) = (x, x')$ for all $x, x'\in \f_I$ and $\lambda \in X$. 
\end{itemize}
Here $x^-$ is the image of $x$ under the isomorphism $\f_I\to \U^-_I$.
We write $(-, -)_I$ for the above linear form to emphasize the dependence on $\dot \U_I$.  
We have the following compatibility of bilinear forms.

\begin{prop}
\label{Psi-inner}
For any $x, y\in \dot \U_{\bI}$, we have 
\begin{align}
(x, y)_{\bI} = (\dot \Psi_\ve (x), \dot \Psi_{-\ve} (y))_{I}, \quad \forall \ve \in \{ \pm 1\}.
\end{align}
\end{prop}

\begin{proof}
For any $x, y \in \dot \U_{\bI}$, we define a bilinear form 
$(x, y)' = (\dot \Psi_\ve (x), \dot \Psi_{-\ve} (y))_{I}$. 
It suffices to show that the form $(-, -)' $ satisfies the conditions (Ba)-(Bc). 
The condition (a) is evidently satisfied by $(-, -)'$. The condition (Bc) is due to the first equality in Proposition~\ref{Inner-com-2}. 
It remains to show that the form $(-, -)'$ satisfies the condition (Bb). 
For any $u\in \U_{\bI}, x, y\in \dot \U_{\bI}$, we have
\begin{align*}
(ux, y)' & =( \dot \Psi_\ve ( ux), \dot \Psi_{-\ve}(y))_{I} \\
& = (\Psi_\ve(u) \dot \Psi_\ve (x), \dot \Psi_{-\ve}(y))_I && (\ref{Psi-dot}) \\
& = ( \dot \Psi_{\ve}(x) , \rho_I (\Psi_\ve(u)) \dot \Psi_{-\ve}(y))_I && (\mbox{(b) for $(-,-)_I$}) \\
& = ( \dot \Psi_{\ve} (x) , \Psi_{-\ve} ( \rho_{\bI}(u) )  \dot \Psi_{-\ve}(y))_I && (\ref{rho-Psi}) \\
& = (   \dot \Psi_{\ve} (x) , \dot \Psi_{-\ve} (\rho_{\bI}(u) y))_I && (\ref{Psi-dot}) \\
& = (x, \rho_{\bI}(u)y)'. 
\end{align*}
Therefore the condition (Bb) holds for $(-,-)'$ and so the two forms $(-,-)_{\bI}$ and $(-,-)'$ all satisfy the conditions (Ba)-(Bc) and hence are the same. The proof is complete. 
\end{proof}

\subsection{Subquotient}
\label{subq-mod}

Recall the subalgebra $\U_{I, \bI}$ and its two sided ideal $\mathcal J_{I, \bI}$ from Section~\ref{subq}. Let
$\dot \U_{I, \bI}$ be the subalgebra of $\dot \U_I$ consists of elements $x 1_\lambda$ for all $x \in \U_{I, \bI}$ and $\lambda\in X$. 
Let $\dot{\mathcal J}_{I, \bI}$ be the two-sided ideal of $\dot \U_{I, \bI}$ consists of elements $x1_\lambda$ for $x\in \mathcal J_{I, \bI}$ and $\lambda \in X$. 
Clearly, we have $\dot \Psi_\ve (\dot \U_{\bI}) \subseteq \dot \U_{I, \bI}$ and hence we have a surjective algebra homomorphism
$\dot \Phi: \dot \U_{\bI} \to \dot \U_{I, \bI} / \dot{\mathcal J}_{I, \bI}$ by composing $\dot \Psi_\ve$ with the canonical quotient map. 
Just like Conjecture~\ref{subqa}, we expect that the morphism $\dot \Phi$ is an isomorphism, and hence $\dot \U_{\bI}$ is expected to be 
a split subquotient of $\dot \U_I$.  Note that in type A, this is known to be true.

\subsection{Simple modules}
\label{simple}
Recall the set $X^+_I$ of dominant integral weights from Section~\ref{Edge-root} and Lusztig algebra $\f$ from Section~\ref{f}.
Let us fix a $\lambda\in X^+_I$ in this section. 
On $\f_I$, there is a unique $\U_I$-module structure such that
$$E_i .1=0, F_i. u = \theta_i x, K_\mu. x = v^{\langle \mu, \lambda - |x|\rangle} x$$
for all $i\in I$ and homogeneous $x\in \f_I$.  
We set $\mT_I(\lambda)= \sum_{i\in I} \f_I \theta^{\langle i,\lambda\rangle +1}$. 
Then $\mT_I(\lambda)$ is the maximal $\U_I$-module in $\f_I$.
Let $\Lambda_{\lambda, I} = \f_I/\mT_I(\lambda)$ be the simple quotient. 
Let $\eta_\lambda$ denote the image of $1$ in $\Lambda_{\lambda}$. 
We shall write $\eta_{\lambda, I}$ for $\eta_\lambda$ to avoid ambiguity whenever deemed necessary.
Let $\B_I(\lambda)$ be the set of nonzero elements in the image of $\B_I$ under the canonical projection $\f_I \to \Lambda_{\lambda, I}$.
The set $\B_I(\lambda)$ is the canonical basis of $\Lambda_{\lambda, I}$.  \index{$\Lambda_{\lambda, I}$}

Recall the involution \textomega \ from Section~\ref{Emb-U}.
Let ${}^{\mbox{\textomega}} \Lambda_{\lambda, I}$ be the $\U_I$-module induced from $\Lambda_{\lambda, I}$ by redefining the $\U_I$-action by
$u.x = \mbox{\textomega} (u) . x$ for all $u\in \U_I$ and $x\in \Lambda_{\lambda, I}$.  The image of $1$ is denoted by $\xi_{-\lambda}$ or 
$\xi_{-\lambda, I}$ to avoid ambiguities.

Since $X^+_I\subseteq X^+_{\bI}$, we define a $\U_{\bI}$ structure on $\f_{\bI}$ with highest weight $\lambda$.
The $\U_{\bI}$-modules $\mT_{\bI}(\lambda)$, $\Lambda_{\lambda, \bI}$  and ${}^{ \mbox{\textomega} } \Lambda_{\lambda,\bI}$ are defined similarly as above.

Let $\ve\in \{\pm 1\}$. 
The $\U_I$-module $\f_I$ of highest weight $\lambda$ can be regarded as a $\U_{\bI}$-module via the embedding $\Psi_\ve: \U_{\bI} \to \U_I $ in Theorem~\ref{Psi-U}.
Then by the definition of the module structure, the map $\psi^{\dagger}_\ve : \f_{\bI} \to \f_{I}$ in (\ref{f-emb-a}) is a $\U_{\bI}$-module homomorphism. 
Moreover we have
\begin{align}
\label{Theta-mod}
\begin{split}
& \theta^{\dagger (\langle i_0,\lambda\rangle +1)}_{i_0,\ve}   = 
\sum_{r+s =\langle i_0,\lambda\rangle +1 } (-1)^r v^{\ve r}_{i_0} \theta^{(r)}_{i_+} \theta^{( \langle i_0,\lambda\rangle +1)}_{i_-} 
\theta^{(s)}_{i_+}\\
& =\sum_{\substack{r+s =\langle i_0,\lambda\rangle +1 \\ s\geq \langle i_+ , \lambda\rangle +1}} (-1)^r v^{\ve r}_{i_0} \theta^{(r)}_{i_+} \theta^{( \langle i_0,\lambda\rangle +1)}_{i_-} 
\theta^{(s)}_{i_+} 
+
\sum_{\substack{r+s =\langle i_0,\lambda\rangle +1 \\ r \geq \langle i_- , \lambda\rangle +1}} (-1)^r v^{\ve r}_{i_0} \theta^{(r)}_{i_+} \theta^{( \langle i_0,\lambda\rangle +1)}_{i_-} 
\theta^{(s)}_{i_+} \\
& = \sum_{\substack{r+s =\langle i_0,\lambda\rangle +1 \\ s\geq \langle i_+ , \lambda\rangle +1}} (-1)^r v^{\ve r}_{i_0} \theta^{(r)}_{i_+} \theta^{( \langle i_0,\lambda\rangle +1)}_{i_-} 
\theta^{(s)}_{i_+} 
+
\sum_{\substack{r+s =\langle i_0,\lambda\rangle +1 \\ r \geq \langle i_- , \lambda\rangle +1}} (-1)^r v^{\ve r}_{i_0} \theta^{(s)}_{i_-} \theta^{( \langle i_0,\lambda\rangle +1)}_{i_+} 
\theta^{(r)}_{i_-} 
\end{split}
\end{align}
where the first equality is due to (\ref{theta-0}) and the last one is due to $\theta^{(r)}_{i_+} \theta^{(r+s)}_{i_-} \theta^{(s)}_{i_+} = \theta^{(s)}_{i_-} \theta^{(r+s)}_{i_+} \theta^{(r)}_{i_-}$. 
This calculation shows that $\psi^{\dagger}_{\ve} (\theta^{(\langle i_0, \lambda\rangle +1)}_{i_0} ) =\theta^{\dagger (\langle i_0,\lambda\rangle +1)}_{i_0,\ve} \in \mT_I(\lambda)$, and hence we have 
\begin{align}
\label{mod-a}
\psi^{\dagger}_\ve ( \mT_{\bI}(\lambda)) \subseteq \mT_{I}(\lambda).
\end{align}
Furthermore, since $\mT_{\bI}(\lambda)$ is a maximal $\U_{\bI}$-submodule in $\f_{\bI}$, we see that the following diagram is cartesian.
\begin{align}
\label{mod-b}
\begin{CD}
\mT_{\bI}(\lambda) @>\psi^{\dagger}_\ve>> \mT_I(\lambda)\\
@VVV @VVV\\
\f_{\bI} @>\psi^{\dagger}_\ve>> \f_I
\end{CD}
\end{align}
where the vertical maps are inclusions.
The property (\ref{mod-a}) induces a $\U_{\bI}$-module homomorphism $\psi^{\dagger}_{\ve, \lambda}: \Lambda_{\lambda, \bI} \to \Lambda_{\lambda,I}$.
The cartesian property of the diagram (\ref{mod-b}) implies that the map $\psi^{\dagger}_{\ve, \lambda}$ is indeed injective. 

In a similar manner, we have an injective $\U_{\bI}$-module homomorphism $\psi_{\ve,\lambda}: {}^{ \mbox{\textomega} } \Lambda_{\lambda,\bI}
\to {}^{ \mbox{\textomega} } \Lambda_{\lambda, I}$ induced from the embedding $\psi_\ve: \f_{\bI} \to \f_I$.

\begin{prop}
\label{prop:simple}
Assume that $\langle i_-,\lambda\rangle =0$, then we have 
$\psi^{\dagger}_{\ve,\lambda} (\B_{\bI}(\lambda)) \subseteq \B_{I}(\lambda)$. 
Assume that $\langle i_+,\lambda \rangle=0$, then we have 
$\psi_{\ve,\lambda} (\B_{\bI}(\lambda)) \subseteq \B_{I}(\lambda)$. 
\end{prop}

\begin{proof}
Since $\langle i_-,\lambda\rangle =0$, we see that $\f_{I} \theta_{i_-} \subseteq \mT_{I}(\lambda)$. This implies that
the canonical projection $\f_I\to \Lambda_{\lambda,I}$ factors through the projection $\f_{I} \to \f^{i_-}_I$. In particular, we see that
the canonical basis $\B^{i_-}_I$ of $\f^{i_-}_I$ gets sent to $\B_I(\lambda)$ or zero under the induced map $\f^{i_-}_I \to \Lambda_{\lambda,I}$.
Now apply Theorem~\ref{B-emb} to conclude the proof of the property on $\psi^{\dagger}_{\ve}$.
The proof of the statement on $\psi_{\ve}$ is entirely similar. This finishes the proof.   
\end{proof}

Let ${}_{\A}\Lambda_{\lambda, I}$ be the integral form of $\Lambda_{\lambda, I}$, i.e., the $\A$-submodule of $\Lambda_{\lambda, I}$ spanned by elements
in $\B_I(\lambda)$. 
Similarly we define ${}_\A {}^{\mbox{\textomega}} \Lambda_{\lambda, I}$  the integral form of  ${}^{ \mbox{\textomega} } \Lambda_{\lambda,\bI}$.
Thanks to Theorem~\ref{f-emb}, we have
\begin{align}
\label{simple-integral}
\psi^{\dagger}_{\ve,\lambda}({}_{\A} \Lambda_{\lambda,\bI} ) \subseteq {}_{\A} \Lambda_{\lambda, I}, 
\psi_{\ve,\lambda} ( {}_\A {}^{\mbox{\textomega}} \Lambda_{\lambda, \bI}) \subseteq {}_\A {}^{\mbox{\textomega}} \Lambda_{\lambda, I}. 
\end{align}
We present the following example to show that the assumptions in Proposition~\ref{prop:simple} are needed.

\begin{ex}
\label{sl2}
Let $(I, \cdot)$ be the Cartan datum of type $A_2$, i.e.,  $I=\{ i_+, i_-\}$ such that $i_+\cdot i_+=i_-\cdot i_-= -2 i_+\cdot i_-=2$.
Let us fix a root datum. Let $\varpi_{i_-}$ be the fundamental weight associated to $i_-$, i.e., $\langle i_+, \varpi_{i_-}\rangle =0$ and $\langle i_-, \varpi_{i_-}\rangle=1$.
Then $\Lambda_{\varpi_{i_-},I}$ is three dimensional and its canonical basis consists of  the images of the  elements $1$, $\theta_{i_-}$, and  $\theta_{i_+}\theta_{i_-}$ under the canonical projection. 
Let $(\bI, \cdot)$ be the edge contraction of $(I, \cdot)$. 
Then $\Lambda_{\varpi_{i_-}, \bI}$ is the natural representation of quantum $\mathfrak{sl}_2$ with 
the canonical basis $\{ 1, \widehat \theta_{i_0}\}$. 
Then we have
$$
\psi^{\dagger}_{\ve, \varpi_{i_-}} ( \widehat \theta_{i_0})= \theta^{\dagger}_{i_0}= - v^{\ve}_{i_0} \theta_{i_+}\theta_{i_-}. 
$$
Thus $\psi^{\dagger}_{\ve,\lambda}$ does not send the canonical basis elements in $\Lambda_{\varpi_{i_-},\bI}$ to those in $\Lambda_{\varpi_{i_-},I}$.
Similarly, $\psi_{\ve,\lambda}$ does not send the canonical basis elements in 
${}^{ \mbox{\textomega} }\Lambda_{\varpi_{i_+}, \bI}$ to those in ${}^{ \mbox{\textomega} }\Lambda_{\varpi_{i_+}, I}$. 
\end{ex}

We end this section with the following problem. 

\begin{prob}
\label{branching} 
Determine the branching rule of $\Lambda_{\lambda, I}$ as a $\U_{\bI}$-module. 
\end{prob}

We refer to Remark~\ref{Naive-5} for a solution when the Cartan datum $(I, \cdot)$ is of classical type. 

\subsection{Tensor products}

Let $\lambda, \lambda' \in  X^+_I$. Set $\zeta= \lambda'-\lambda$. Consider the following left ideal of $\dot \U_I$.
\[
P_I(\lambda, \lambda') = \sum_{i\in I} \dot \U_I E^{(\langle i, \lambda\rangle +1)}_i 1_{\zeta} +
\sum_{i\in I} \dot \U_I F^{(\langle i, \lambda'\rangle+1)}_i 1_{\zeta}.
\]
We have the following short exact sequence in the category of left $\U_I$-modules.
\begin{align}
\label{ses-I}
0\to P_I(\lambda,\lambda') \to \dot \U_I 1_{\zeta} \overset{ \pi}{\to} {}^{ \mbox{\textomega} } \Lambda_{\lambda, I} \otimes \Lambda_{\lambda', I} \to 0
\end{align}
where $\pi$ is given by $u\mapsto u. \xi_{-\lambda} \otimes \eta_{\lambda'}$ for all $u\in \dot \U_I1_{\zeta}$.

Since $X^+_I\subseteq X^+_{\bI}$, we can define a similar left ideal $P_{\bI}(\lambda,\lambda')$ in $\dot \U_{\bI}$ and the 
short exact sequence in the category of left $\U_{\bI}$-modules similar to (\ref{ses-I}) with $I$ replaced by $\bI$.

By (\ref{Theta-mod}), we have $\dot \Psi_\ve(P_{\bI}(\lambda,\lambda')) \subseteq P_I(\lambda, \lambda')$. 
Hence, thanks to the short exact sequence (\ref{ses-I}), the morphism  $\dot \Psi_\ve$ induces a linear map 
\begin{align}
\label{Psi-tensor}
\dot \Psi^{\lambda,\lambda'}_{\ve}: {}^{ \mbox{\textomega} } \Lambda_{\lambda, \bI} \otimes \Lambda_{\lambda', \bI} \to 
{}^{ \mbox{\textomega} } \Lambda_{\lambda, I} \otimes \Lambda_{\lambda', I},
\end{align}
such that $u. \xi_{-\lambda, \bI} \otimes \eta_{\lambda', \bI} \to \dot \Psi_{\ve}(u) . \xi_{-\lambda, I} \otimes \eta_{\lambda',I}$ for all
$u\in \dot \U_{\bI}1_{\zeta}$. 
Clearly the map $\dot \Psi^{\lambda,\lambda'}_{\ve}$ is a $\U_{\bI}$-module homomorphism if ${}^{ \mbox{\textomega} } \Lambda_{\lambda, I} \otimes \Lambda_{\lambda', I}$ is regarded as a $\U_{\bI}$-module via the homomorphism $\Psi_\ve$. 
Moreover, due to (\ref{Udot-integral}), the map is compatible with the integral forms, i.e.,
\begin{align}
\label{Psi-tensor-int}
\dot \Psi^{\lambda,\lambda'}_{\ve}( {}_{\A} {}^{ \mbox{\textomega} }\Lambda_{\lambda, \bI} \otimes_{\A} {}_{\A} \Lambda_{\lambda', \bI}) \subseteq 
{}_{\A} {}^{ \mbox{\textomega} } \Lambda_{\lambda, I} \otimes_{\A} {}_{\A} \Lambda_{\lambda', I}.
\end{align}

\begin{prop}
Let $\ve\in \{\pm 1\}$. 
The linear map $\dot \Psi^{\lambda, \lambda'}_\ve$ in (\ref{Psi-tensor})  is injective. 
\end{prop}

\begin{proof}
The $I$-grading on $\f_I$ induces an $I$-grading on $\Lambda_{\lambda', I}$ and ${}^{ \mbox{\textomega} }\Lambda_{\lambda, I}$. 
Write $\Lambda^{\lambda'- \nu}_{\lambda',I}$ (resp. ${}^{ \mbox{\textomega} }\Lambda^{-\lambda+\nu}_{\lambda,I}$) of the image of $\f_{\nu,I}$ under the canonical projection. 
Let ${}^{ \mbox{\textomega} } \Lambda_{\lambda, I} \otimes \Lambda_{\lambda', I}|_{\bI}$ be the subspace of ${}^{ \mbox{\textomega} } \Lambda_{\lambda, I} \otimes \Lambda_{\lambda', I}$ spanned by the element $x\otimes y\in {}^{ \mbox{\textomega} }\Lambda^{-\lambda+\tau}_{\lambda, I}\otimes \Lambda^{\lambda'- \tau'}_{\lambda', I}$ for various $\tau,\tau'\in \mbb Z[\bI]$.
Let $\dot \Psi^{\lambda,\lambda'}_{\ve}|_{\bI}$ be the composition of $\dot \Psi^{\lambda,\lambda'}_{\ve}$ with the projection 
from ${}^{ \mbox{\textomega} } \Lambda_{\lambda, I} \otimes \Lambda_{\lambda', I}$ to ${}^{ \mbox{\textomega} } \Lambda_{\lambda, I} \otimes \Lambda_{\lambda', I}|_{\bI}$.

In light of Proposition~\ref{Emb-dot-co}, we see that 
$\dot \Psi^{\lambda,\lambda'}_{\ve}|_{\bI} = \psi_{\ve,\lambda}\otimes \psi^{\dagger}_{\ve, \lambda'}$.
The latter is an injective map because the maps $\psi_{\ve,\lambda}$ and $\psi^{\dagger}_{\ve, \lambda'}$ are injective. 
This implies that the map $\dot \Psi^{\lambda,\lambda'}_{\ve}$  itself is injective. The proposition is proved. 
\end{proof}

\begin{rem}
In light of Proposition~\ref{prop:simple}, one expects that if $\langle i_+,\lambda\rangle =0$ and $\langle i_-, \lambda'\rangle=0$, then the map 
$\dot \Psi^{\lambda,\lambda'}_{\ve}|_{\bI}$ sends the canonical bases of
${}^{ \mbox{\textomega} } \Lambda_{\lambda, \bI} \otimes \Lambda_{\lambda', \bI}$ and the projection of the canonical bases of 
${}^{ \mbox{\textomega} } \Lambda_{\lambda, I} \otimes \Lambda_{\lambda', I}$ to 
${}^{ \mbox{\textomega} } \Lambda_{\lambda, I} \otimes \Lambda_{\lambda', I}|_{\bI}$.
This is true in the setting of Example~\ref{sl2} when $(\lambda, \lambda')=(\varpi_{i_-},\varpi_{i_+})$
by showing that the linear map $\dot \Psi^{\lambda,\lambda'}_{\ve}|_{\bI}$ is compatible with the quasi R-matrix 
on ${}^{ \mbox{\textomega} } \Lambda_{\lambda, \bI} \otimes \Lambda_{\lambda', \bI}$ and the (induced) quasi R-matrix on 
${}^{ \mbox{\textomega} } \Lambda_{\lambda, I} \otimes \Lambda_{\lambda', I}|_{\bI}$.
\end{rem}

\subsection{Quantum  Schur algebras}
\label{Schur}
In this section, we address the compatibility of $\dot \Psi_\ve$ with the quantum  Schur algebras.
This provides a generalization to the work~\cite{Li21}, where the case of affine type $A$ is treated. 

 Let $(I, \cdot)$ be a Cartan datum of finite type. 
Let $(Y, X)$ be the simply connected root datum of $(I, \cdot)$. 
Let $(\bI,\cdot)$ be the edge contraction of $(I,\cdot)$ along $\{i_+, i_-\}$. 
Let $(\widehat Y, \widehat X)$ be the simply connected root datum of $(\bI, \cdot)$.
Let $(Y, X)_{\bI}$ be the root datum of type $(\bI, \cdot)$ induced from $(Y, X)$. 
There is a morphism of root data $\phi=(f, g): (\widehat Y,\widehat X) \to (Y, X)_{\bI}$ with $f: \widehat Y\to Y$ and $g: X\to \widehat X$.

Let $\dot \U_{I}$ be the modified quantum group associated with the root datum $(Y, X)$.
Let  $\dot \U_{\bI, X}$ (resp. $\dot \U_{\bI, \widehat X}$) 
be the modified quantum group associated with the root datum  $(Y, X)_{\bI}$ (resp. $(\widehat Y, \widehat X)$) of type $(\bI, \cdot)$.

Now fix a $\lambda^0 \in X^+_I=\{ \lambda \in X| \langle i, \lambda\rangle \in \mbb N, \forall i\in I\}$.
Let $\Pi (\lambda^0)$ be the set of weights of the simple $\U_{I}$-module $\Lambda_{\lambda^0}$.
Let $\widehat{\lambda^0}=g(\lambda^0) \in \widehat X^+_{\bI}$.
Let $\Pi(\widehat {\lambda^0})$ be defined similar to $\Pi(\lambda^0)$.
In light of the injection $\psi_{\ve}$ in Section~\ref{simple}, we have an injection $\Pi(\widehat{\lambda^0}) \to \Pi(\lambda^0)$ defined by
sending $\widehat{\lambda^0} -\sum_{i\in \bI} a_i i' \mapsto \lambda^0 -\sum_{i\in \bI} a_i i'$.  Since $g$ is surjective, we extend the map to
a section $\tilde g: \widehat X \to X$ of $g$, i.e.,  $g. \tilde g=1: \widehat X\to \widehat X$.
We thus have an embedding 
\begin{align}
\phi_{\lambda^0}: \dot \U_{\bI, \widehat X} \to \dot \U_{\bI, X}
\end{align}
defined as the composition
$\oplus_{\lambda\in \widehat X}  \dot{\U}_{\bI, \widehat X} 1_{\lambda} \to \oplus_{\lambda\in \widehat X}  \dot{\U}_{\bI, X} 1_{\tilde g(\lambda)} \hookrightarrow \dot \U_{\bI, X}$.

Recall from~\cite[29.1]{L10}, we have a two-sided ideal $\dot \U_I[\geq \lambda^0]$ in $\dot \U_I$.  
Let $\S_{I, \lambda^0}$ be the quotient  algebra of $\dot \U_I$ by $\dot \U_I[\geq \lambda^0]$.
Let $1_{\lambda}$ denote the image of the element in the same notation in $\dot \U_I$ under the canonical projection 
$\pi_{\lambda^0} : \dot \U_I\to \S_{I, \lambda^0}$ for all $\lambda \in \Pi(\lambda^0)$. 
Let $\S_{\bI, \widehat{\lambda^0}}$ be the quotient of $\dot \U_{\bI, \widehat X}$ 
by the ideal $\dot \U_{\I, \widehat X} [\geq \widehat{\lambda^0}]$. 
There exists an algebra homomorphism, which does not respect the unit, $\sigma_{\lambda^0,\ve}: \S_{\bI, \widehat{\lambda^0}}\to \S_{I,\lambda^0}$ such that $1_{\lambda} \mapsto 1_{\tilde g(\lambda)}$ for all $\lambda\in \Pi(\widehat{\lambda^0})$ and that $\sigma_{\lambda^0}$ respects the $\U_{\bI}$ and $\U_{I}$-module structures via $\Psi_\ve$. 
The existence of $\pi_{\lambda^0}$ is due to the presentation of $\S_{I,\lambda^0}$ in~\cite{D03}.

Then we have

\begin{prop}
\label{prop:Psi-Schur}
The following diagram commutes.
\begin{align}
\label{Psi-Schur-b}
\begin{split}
\xymatrix{
\dot \U_{\bI, \widehat X}  \ar[r]^{\phi_{\lambda^0}} \ar[d]_{\pi_{\widehat{\lambda^0}}}  & \dot \U_{\bI, X} \ar[r]^{\dot \Psi_\ve}& \dot \U_{I} \ar[d]^{\pi_{\lambda^0}}  \\
\S_{\bI, \widehat \lambda^0} \ar[rr]_{\sigma_{\lambda^0,\ve}} && \S_{I,\lambda^0}
}
\end{split}
\end{align}
\end{prop}

It is interesting to see to what extent the canonical bases are compatible with the above diagram.


\subsection{Quantum coordinate algebras and Chevalley groups}

In this section, we assume that the Cartan datum $(I, \cdot)$  is of finite type, i.e., the associated Cartan matrix $(2 i\cdot j/i\cdot i)_{i, j\in I}$ is positive definite.
Fix a Cartan datum $(Y, X)$ of $(I, \cdot)$. 
Recall the map $\pi$ from (\ref{ses-I}). It induces an injective $\A$-linear map 
$\Hom_\A ( {}_{\A} ^{\mbox{\textomega}} \Lambda_{\lambda, I} \otimes_{\A} {}_{\A} \Lambda_{\lambda', I}, \A) 
\to \Hom_\A ( {}_{\A} \dot \U_{I}1_{\lambda'-\lambda}, \A)\hookrightarrow  \Hom_\A ( {}_{\A} \dot \U_{I}, \A) $. 
Let ${}_{\A} \dot \U_I^*(\lambda, \lambda')$ be the image of the above map. 
Let 
$${}_{\A} \bO_{I}=\sum_{\lambda, \lambda'\in X^+_I} {}_{\A} \dot \U_I^*(\lambda, \lambda').$$
It is known from~\cite{L09} that ${}_{\A}\bO_I$ can be equipped with a Hopf-algebra structure over $\A$, where the multiplication and the comultiplication 
are naturally induced from the comultiplication and multiplication of $_{\A}\dot \U_I$ respectively.
This is the quantum coordinate algebra over $\A$ of type $(I, \cdot)$.

Let $(\bI, \cdot)$ be the edge contraction of $(I, \cdot)$ along $\{ i_+, i_-\}$. 
Thanks to (\ref{Udot-integral}) and (\ref{Psi-tensor-int}), there is an $\A$-linear map induced by $\dot \Psi_{\ve}$ for each $\ve\in \{ \pm 1\}$, 
\index{$\dot \Psi^*_\ve$}
\begin{align}
\label{Psi-dual}
\dot \Psi^*_{\ve}: {}_{\A} \bO_I \to {}_{\A} \bO_{\bI}. 
\end{align}

\begin{prop}
Let $\ve\in \{\pm 1\}$. 
The $\A$-linear map $\dot \Psi^*_\ve: {}_{\A} \bO_I \to {}_{\A} \bO_{\bI}$ is a surjective co-algebra homomorphism.
\end{prop}

\begin{proof}
It is a coalgebra homomorphism is because $\dot \Psi_\ve$ is an algebra homomorphism. 
The surjective property is due to the fact that the map $\dot \Psi^{\lambda, \lambda'}_\ve$ is injective. 
The proof is finished.
\end{proof}

Let $R$ be a commutative ring with $1$. If there is a ring homomorphism $\phi: \A \to R$ respecting $1$, we can consider the Hopf algebras over $R$:
${}_R \bO_I =R\otimes_\A {}_{\A} \bO_{I}$ and ${}_R \bO_{\bI}$. By tensoring  $R$ with $\dot \Psi^*_\ve $ over $\A$ , we get a surjective coalgebra homomorphism 
$${}_R \dot \Psi_{\ve}^*: {}_{R} \bO_I \to {}_{R} \bO_{\bI}.$$

\begin{prop}
\label{Psi-coord}
Assume that $\phi(v)=1$. The coalgebra homomorphism ${}_R \dot \Psi^*_\ve= {}_R \dot \Psi^*_{-\ve}$ is a Hopf-algebra homomorphism. 
\end{prop}

\begin{proof}
In the case $\phi(v)=1$, the case reduces to the non quantum version. In this case, the $v=1$ version of the morphisms $\Psi_\ve$ and $\Psi_{-\ve}$ coincide and they are indeed a Hopf algebra homomorphism.
So is ${}_R \dot \Psi^*_\ve$. The proposition is proved. 
\end{proof}

Let $\G_{I, R} $ be the set of all algebra homomorphism from ${}_R \bO_I $ to $R$. 
Due to the fact that ${}_R \bO_I$ is a Hopf algebra, $\G_{I, R}$ is a group. Moreover, 
if $\phi(v) =1$, then $\G_{I, R}$ is a reductive group over $R$ of type $(I, \cdot)$. 
By Proposition~\ref{Psi-coord} and a standard argument, we have  \index{${}_R\Psi$}

\begin{thm}
\label{Psi-group}
Assume that $\phi(v)=1$. 
There is a group embedding
$
{}_R \Psi: \G_{\bI, R} \to \G_{I, R}
$
induced by ${}_R \dot \Psi^*_\ve$.
\end{thm}

Note that the algebra ${}_{\A} \bO_{\bI}$ is defined with respect to the root datum $(Y, X)$, regarded as a Cartan datum of $(\bI, \cdot)$ under the edge contraction along $\{ i_+, i_-\}$. 
We write ${}_{\A} \bO_{\bI, X}$ to emphasize the dependence of the root datum. 
On the other hand, the quantum coordinate algebra ${}_\A \bO_{\bI}$ is isomorphic to the algebra defined similar to ${}_\A\bO_{\bI}$ by replacing ${}_\A \dot \U_{\bI}$ by ${}_\A  \U_{\bI}$. The isomorphism can be established by exploring the quotient map ${}_\A \U_{\bI} \to {}_\A \dot \U_{\bI} 1_{\zeta}$. 
If we identify these two algebras, we see that there is a Hopf algebra homomorphism 
$${}_{\A}\bO_{\bI, X} \to {}_{\A}\bO_{\bI, X^{sc}_{\bI}},$$
where $(Y^{sc}_{\bI}, X^{sc}_{\bI})$ is the simply-connected root datum of $(\bI, \cdot)$.  
This leads to a group homomorphism ${}_R\pi: \G_{\bI, R}^{sc} \to \G_{\bI, R}$, where $\G^{sc}_{\bI, R}$ is defined with respect to the simply-connected root datum of $(\bI, \cdot)$.

\begin{ex}
Retaining the setting in Example~\ref{sl2}. Let $(Y, X)$ be the simply-connected root datum of $(I, \cdot)$.  
The composition ${}_R \pi \circ {}_R \Psi$ of the group homomorphism in Theorem~\ref{Psi-group} and ${}_R \pi$ is the embedding $\mrm{SL}(2, R) \to \mrm{SL}(3, R)$ defined by
\[
\begin{bmatrix}
a & b\\
c & d
\end{bmatrix}
\mapsto 
\begin{bmatrix}
a & 0 & b\\
0 & 1 & 0 \\
c & 0 & d
\end{bmatrix}
\]
\end{ex}

\section{Compatibility of Braid group actions}
\label{Braid}
In this section, we study the compatibility of braid group actions on $\U_I$ and $\U_{\bI}$. 
We show that when either $i_+$ or $i_-$ is an end vertex, they are compatible. 
The definition of the operators $T'_{i_0,e}$ and $T''_{i_0,e}$ on $\U_I$ is non trivial. 
We further show that the braid group actions are compatible under the subquotient map $\Phi$ in (\ref{phi}). 

\subsection{The  actions $\tT'_{i_0, e}$ and $\tT''_{i_0, e}$}

\index{$\tT'_{i_0, e}$} \index{$\tT''_{i_0, e}$}

Recall the braid group actions $T'_{i, e}$ and $T''_{i, e}$ on $\U$ from~\cite[Part VI]{L10}.
For any $e\in \{ \pm 1\}$, we set
\begin{align}
\label{tT}
\tT'_{i_0,e} = T'_{i_+, e} T'_{i_-,e} T'_{i_+, e},
\tT''_{i_0,e} = T''_{i_+,e} T''_{i_-,e} T''_{i_+, e}. 
\end{align}
Note that $\tT'_{i_0,e}$ and $\tT''_{i_0, e}$ are inverse to each other. 
We are interested in studying the behaviors of these operators in this section.
For simplicity, we study them under the following assumption.  
\begin{align}
\label{Braid}
\mbox{There is no $j\in I$ such that $j \cdot i_+\neq 0$ and $j \cdot i_-\neq 0$.}
\end{align}

\begin{prop}
\label{Braid-1}
Assume that the condition (\ref{Braid}) holds.
For any $e, \ve \in \{ \pm 1\}$, we have
\begin{align}
\label{T1}
\tag{T1}
\tT'_{i_0, e} (E_{i_0, \ve}) & = - \tilde K_{ei_0} ( - v^{-e}_{i_0} F_{i_0, - \ve}), \\
\label{T2}
\tag{T2}
\tT'_{i_0, e} (F_{i_0, \ve} ) & = - (-v^e_{i_0} E_{i_0, -\ve} ) \tilde K_{- ei_0},\\
\label{T3}
\tag{T3}
\tT'_{i_0,e} (E_j) & = \sum_{r+s= - \langle i_0, j'\rangle} (-1)^r v^{er}_{i_0} E^{(r)}_{i_0, -e} E_j E^{(s)}_{i_0, -e}, &&  \forall j\cdot i_-=0,\\
\label{T4}
\tag{T4}
\tT'_{i_0, e} (F_j) & = \sum_{r+s= - \langle i_0, j'\rangle} (-1)^r v^{-er}_{i_0}F^{(s)}_{i_0, -e} F_j F^{(r)}_{i_0, -e}, &&  \forall j\cdot i_- =0,\\
\label{T5}
\tag{T5}
\tT'_{i_0, e} (E_k) & = \sum_{r+s=-\langle i_0, k'\rangle} (-1)^r v^{er}_{i_0} ( - v^e_{i_0} E_{i_0, e})^{(r)} E_k (- v^e_{i_0} E_{i_0, e})^{(s)}, 
&&  \forall k \cdot i_+=0,\\
\label{T6}
\tag{T6}
\tT'_{i_0, e} (F_k) & = \sum_{r+s=-\langle i_0, k'\rangle} (-1)^r v^{-er}_{i_0} ( - v^{-e}_{i_0} F_{i_0, e})^{(s)} F_k (- v^{-e}_{i_0} F_{i_0, e})^{(r)}, 
&& \forall k\cdot i_+=0.
\end{align}
\end{prop}

\begin{proof}
We first observe from~\cite[Proposition 37.2.5]{L10} that
\[
T'_{i_+, 1} (E_{i_0, -1}) = E_{i_-}, 
T'_{i_-, 1} (E_{i_0, 1})= -v^{-1}_{i_0} E_{i_+}.
\]
So we have
\begin{align}
\label{T1a}
\begin{split}
\tT'_{i_0, 1} ( E_{i_0, -1}) & = T'_{i_+, 1} T'_{i_-, 1} T'_{i_+, 1} ( E_{i_0, -1})\\
& = T'_{i_+, 1} T'_{i_-, 1} (E_{i_-}) \\
&= T'_{i_+, 1} ( - \tilde K_{i_-} F_{i_-})\\
&=- \tilde K_{i_0} ( F_{i_+} F_{i_-} - v^{-1}_{i_0} F_{i_-} F_{i_+}) \\
& = - \tilde K_{ei_0} (- v^{-e}_{i_0} F_{i_0, -\ve}), \quad \mbox{if}\ (e,\ve)=(1, -1).
\end{split}
\end{align}
Moreover, we have
\begin{align}
\label{T1b}
\begin{split}
\tT'_{i_0, 1} (E_{i_0, 1}) & =T'_{i_+, 1} T'_{i_-, 1} T'_{i_+, 1} ( E_{i_0, 1})\\
& = T'_{i_-, 1} T'_{i_+, 1} T'_{i_-, 1} ( E_{i_0, 1})\\
& = T'_{i_-, 1} T'_{i_+, 1} ( - v^{-1}_{i_0} E_{i_+}) \\
&= T'_{i_-, 1} ( -v^{-1}_{i_0}  ( - \tilde K_{i_+} F_{i_+}))\\
& = - \tilde K_{i_0} ( - v^{-1}_{i_0} ( F_{i_-} F_{i_+} - v^{-1}_{i_0} F_{i_+} F_{i_-}) ) \\
& = - \tilde K_{ei_0} ( -v^{-e}_{i_0} F_{i_0,-\ve}) , \quad \mbox{if} \ (e, \ve) = (1, 1).
\end{split}
\end{align}
By (\ref{T1a})-(\ref{T1b}), we see that (\ref{T1}) holds for $e=1$. 
Note that $\overline{\tT'_{i_0, e}(u)} = \tT'_{i_0, -e} (\bar u)$ for all $u\in \U$. We have then
\begin{align*}
\begin{split}
\tT'_{i_0, -1} (E_{i_0, \ve}) &= \overline{\tT'_{i_0, 1} (\bar E_{i_0, \ve})}\\
&= \overline{\tT'_{i_0, 1} ( E_{i_0, -\ve})}\\
& = \overline{ - \tilde K_{i_0} ( - v^{-1}_{i_0} F_{i_0, \ve}) } \\
& =- \tilde K^{-1}_{i_0} ( - v_{i_0} F_{i_0, -\ve})\\
&= - \tilde K_{ei_0} ( -v^{-e}_{i_0} F_{i_0,-\ve}) , \quad \mbox{if} \ e =-1.
\end{split}
\end{align*}
This shows that (\ref{T1}) holds for $e=-1$, and thus the equality (\ref{T1}) is proved.

Note that 
\begin{align*}
T'_{i_+, 1} ( F_{i_0, -1} ) = F_{i_-},
T'_{i_-, 1} ( F_{i_0, 1}) = - v_{i_0} F_{i_+} .
\end{align*}
So we have 
\begin{align*}
\begin{split}
\tT'_{i_0,1} (F_{i_0, -1}) & = T'_{i_+, 1} T'_{i_-, 1} T'_{i_+, 1} (F_{i_0, -1}) \\
&= T'_{i_+, 1} T'_{i_-, 1}  ( F_{i_-})\\
& = T'_{i_+, 1} ( - E_{i_-} \tilde K^{-1}_{i_-} )\\
&=- ( E_{i_-} E_{i_+} - v_{i_0} E_{i_+} E_{i_-}) \tilde K^{-1}_{i_0}\\
& = - (-v^e_{i_0} E_{i_0, -\ve} ) \tilde K_{- ei_0}, \quad \mbox{if} \ (e,\ve) = ( 1, -1).
\end{split}
\end{align*}
\begin{align*}
\begin{split}
\tT'_{i_0, 1} (F_{i_0, 1}) & = T'_{i_-, 1} T'_{i_+, 1} T'_{i_-, 1} ( F_{i_0, 1}) \\
&=T'_{i_-, 1} T'_{i_+, 1} ( - v_{i_0} F_{i_+}) \\
& =T'_{i_-, 1} ( - v_{i_0}  ( - E_{i_+} \tilde K^{-1}_{i_+}) )\\
& = v_{i_0} ( E_{i_+} E_{i_-} - v_{i_0} E_{i_-} E_{i_+} ) \tilde K^{-1}_{i_0}\\
& = - (-v^e_{i_0} E_{i_0, -\ve} ) \tilde K_{- ei_0}, \quad \mbox{if} \ (e, \ve) = ( 1, 1).
\end{split}
\end{align*}
The above computation shows that (\ref{T2}) holds for $e=1$.  Further, we have
\begin{align*}
\begin{split}
\tT'_{i_0, -1} (F_{i_0, \ve}) & = \tT'_{i_0, -1}( \bar F_{i_0, -\ve}) \\
& =  \overline{\tT'_{i_0, 1} (F_{i_0, -\ve}) }\\
& = \overline{ - (- v_{i_0} E_{i_0, \ve}) \tilde K^{-1}_{i_0} } \\
& = - (-v^{-1}_{i_0} E_{i_0, -\ve}) \tilde K_{i_0} \\
& = - (-v^e_{i_0} E_{i_0, -\ve} ) \tilde K_{- ei_0}, \quad \mbox{if} \ e=-1.
\end{split}
\end{align*}
This shows that (\ref{T2}) holds for $e=-1$, and completing the proof of (\ref{T2}).

Assume now that $i_-\cdot j=0$. So we have $\langle i_0, j'\rangle = \langle i_+, j'\rangle$. 
\begin{align*}
\begin{split}
\tT'_{i_0, e} (E_j) & = T'_{i_-, e} T'_{i_+, e} T'_{i_-, e} ( E_j) \\
& =T'_{i_-, e} T'_{i_+, e} (E_j) \\
&= T'_{i_-, e} \left ( \sum_{r+s= - \langle i_+, j'\rangle } (-1)^r v^{er}_{i_0} E^{(r)}_{i_+} E_j E^{(s)}_{i_+} \right )\\
&=\sum_{r+s= - \langle i_0, j'\rangle } (-1)^r v^{er}_{i_0} E^{(r)}_{i_0, -e} E_j E^{(s)}_{i_0, -e},
\end{split}
\end{align*}
\begin{align*}
\begin{split}
\tT'_{i_0, e} (F_j) & = T'_{i_-, e} T'_{i_+, e} T'_{i_-, e} ( F_j) \\
& =T'_{i_-, e} T'_{i_+, e} (F_j) \\
&= T'_{i_-, e} \left ( \sum_{r+s= - \langle i_+, j'\rangle } (-1)^r v^{-er}_{i_0} F^{(s)}_{i_+} F_j F^{(r)}_{i_+} \right )\\
&=\sum_{r+s= - \langle i_0, j'\rangle } (-1)^r v^{-er}_{i_0} F^{(s)}_{i_0, -e} F_j F^{(r)}_{i_0, -e},
\end{split}
\end{align*}
where we use $T'_{i_-, e} (E_{i_+}) = E_{i_0, -e}$ and $T'_{i_-, e} (F_{i_+}) = F_{i_0, -e}$ 
in the last equality. This proves (\ref{T3}) and (\ref{T4}). 

Assume that $i_+ \cdot k=0$. Then we have $\langle i_0, k'\rangle = \langle i_-, k'\rangle$ and 
\begin{align*}
\begin{split}
\tT'_{i_0, e} (E_k) & = T'_{i_+, e} T'_{i_-, e} T'_{i_+, e} ( E_k) \\
& =T'_{i_+, e} T'_{i_-, e} (E_k) \\
&= T'_{i_+, e} \left ( 
\sum_{r+s= -\langle i_-, k'\rangle} (-1)^r v^{er}_{i_-} E^{(r)}_{i_-} E_k E^{(s)}_{i_-}
\right )\\
& = \sum_{r+s= - \langle i_0, k\rangle} (-1)^r v^{er}_{i_0} ( - v^e_{i_0} E_{i_0, e})^{(r)} E_k (- v^e_{i_0} E_{i_0, e})^{(s)},
\end{split}
\end{align*}
\begin{align*}
\begin{split}
\tilde T'_{i_0, e} (F_k) & = T'_{i_+, e} T'_{i_-, e} T'_{i_+, e} ( F_k) \\
& =T'_{i_+, e} T'_{i_-, e} (F_k) \\
&= T'_{i_+, e} \left ( 
\sum_{r+s= -\langle i_-, k'\rangle} (-1)^r v^{-er}_{i_-} F^{(s)}_{i_-} F_k F^{(r)}_{i_-}
\right ) \\
& = \sum_{r+s=-\langle i_0, k'\rangle} (-1)^r v^{-er}_{i_0} ( - v^{-e}_{i_0} F_{i_0, e})^{(s)} F_k (- v^{-e}_{i_0} F_{i_0, e})^{(r)}, 
\end{split}
\end{align*}
where we use $T'_{i_+, e} (E_{i_-}) =  - v^e_{i_0} E_{i_0, e}$ and 
$T'_{i_+, e} (F_{i_-}) = - v^{-e}_{i_0} F_{i_0, e} $ in the last equalities.  So we have (\ref{T5}) and (\ref{T6}). This finishes the proof of the proposition.
\end{proof}

Next we study the operator $\tT''_{i_0, e}$. 

\begin{prop}
\label{Braid-2}
Assume that the condition (\ref{Braid}) holds. 
For any $e, \ve\in \{ \pm 1\}$, we have
\begin{align*}
\tT''_{i_0, e} ( E_{i_0, \ve}) & = - ( - v^{e}_{i_0} F_{i_0, -\ve}) \tilde K_{ei_0},\\
\tT''_{i_0, e} ( F_{i_0, \ve}) & = - \tilde K_{- e i_0} ( - v^{-e}_{i_0}  E_{i_0, -\ve}), \\
\tT''_{i_0, e} ( E_{j} ) & = \sum_{r+s=- \langle i_0, j'\rangle} (-1)^r v^{-er}_{i_0} ( - v^{-e}_{i_0} E_{i_0, -e})^{(s)} E_j (-v^{-e}_{i_0} E_{i_0,-e})^{(r)},
&& \forall j\cdot i_-=0,\\
\tT''_{i_0, e} (F_{j}) & = \sum_{r+s=-\langle i_0,j'\rangle} (-1)^r v^{er}_{i_0} ( -v^e_{i_0} F_{i_0, -e})^{(r)} F_j ( -v^e_{i_0} F_{i_0, -e})^{(s)}, && \forall j\cdot i_-=0, \\
\tT''_{i_0, e} (E_k) & =  \sum_{r+s=-\langle i_0,k'\rangle} (-1)^r v^{-er}_{i_0}  E^{(r)}_{i_0, e} E_k E^{(s)}_{i_0, e}, &&\forall k\cdot i_+=0,\\
\tT''_{i_0,e}(F_k) & = \sum_{r+s=-\langle i_0,k'\rangle} (-1)^r v^{er}_{i_0}  F^{(s)}_{i_0, e} F_k F^{(r)}_{i_0, e}, && \forall k\cdot i_+=0.
\end{align*}
\end{prop}

\begin{proof}
Recall that $ \mbox{\textomega} $ is the involution on $\U$ defined by $E_i\mapsto F_i$, $F_i \mapsto E_i$ and $K_\mu \mapsto K_{-\mu}$ for all $i\in I$ and $\mu \in Y$.
Then we have $ \mbox{\textomega} ( E_{i_0, \ve}) = - v^{-\ve}_{i_0} F_{i_0,\ve}$
and
$
\tT''_{i_0,e}=  \mbox{\textomega}  \tT'_{i_0, e}  \mbox{\textomega} .
$
By using these facts, we get
\begin{align*}
\begin{split}
\tT''_{i_0, e} ( E_{i_0,\ve}) &=  \mbox{\textomega}  \tT'_{i_0, e}  \mbox{\textomega}  (E_{i_0, \ve})\\
&=  \mbox{\textomega}  \tT'_{i_0,e} (- v^{-\ve}_{i_0} F_{i_0,\ve}) \\
& =- v^{-\ve}_{i_0}   \mbox{\textomega}  ( - (-v^e_{i_0} E_{i_0, -\ve}) \tilde K_{-ei_0}) \\
& = - v^{-\ve}_{i_0} ( - (-v^e_{i_0} ) ( - v^{\ve}_{i_0} F_{i_0, -\ve} ) \tilde K_{ei_0} )\\
& = - (-v^e_{i_0} F_{i_0, -\ve} ) \tilde K_{ei_0}.
\end{split}
\end{align*}
\begin{align*}
\begin{split}
\tT''_{i_0, e}(F_{i_0,\ve}) & =  \mbox{\textomega}  \tT'_{i_0,e}  \mbox{\textomega}  ( F_{i_0,\ve})\\
& =  \mbox{\textomega}  \tT'_{i_0,e} ( - v^{\ve}_{i_0} E_{i_0,\ve}) \\
& =  (-v^{\ve}_{i_0})  \mbox{\textomega}  ( - \tilde K_{ei_0} ( - v^{-e}_{i_0} F_{i_0, -\ve}))\\
& =(-v^{\ve}_{i_0}) (-\tilde K_{- ei_0} ) ( - v^{-e}_{i_0} ( -v^{-\ve}_{i_0} E_{i_0,-\ve})) \\
&= - \tilde K_{-ei_0} (-v^{-e}_{i_0} E_{i_0,-\ve}) .
\end{split}
\end{align*}
The above computations verify the first two equalities in the proposition.

Assume that $j\cdot i_-=0$. We have
\begin{align*}
\begin{split}
\tT''_{i_0, e} (E_j) & =  \mbox{\textomega}  \tT'_{i_0, e}  \mbox{\textomega}  (E_j) \\
&=  \mbox{\textomega}  \tT'_{i_0,e}(  F_j) \\
& =  \mbox{\textomega}  ( \sum_{r+s= -\langle i_0, j'\rangle} (-1)^r v^{-er}_{i_0} F^{(s)}_{i_0, -e} F_j F^{(r)}_{i_0, -e}) \\
&= \sum_{r+s= -\langle i_0, j'\rangle} (-1)^r v^{-er}_{i_0}  ( -v^{-e}_{i_0} E_{i_0, -e})^{(s)} E_j ( -v^{-e}_{i_0} E_{i_0, -e})^{(r)}
\end{split}
\end{align*}
\begin{align*}
\begin{split}
\tT''_{i_0,e} (F_j) & =  \mbox{\textomega}  \tT'_{i_0,e}  \mbox{\textomega}  (F_j) \\
&=  \mbox{\textomega}  \tT'_{i_0,e} (E_j) \\
&=  \mbox{\textomega}  ( 
\sum_{r+s=-\langle i_0, j'\rangle} (-1)^r v^{er}_{i_0} E^{(r)}_{i_0, -e} E_j E^{(s)}_{i_0, -e}
)\\
& = \sum_{r+s=-\langle i_0, j'\rangle} (-1)^r v^{er}_{i_0}  ( -v^e_{i_0} F_{i_0, -e} )^{(r)} F_j 
(-v^e_{i_0} F_{i_0, -e})^{(s)}.
\end{split}
\end{align*}
The above computations verify the third and fourth equalities in the proposition.

Assume that $k\cdot i_+=0$. We have 
\begin{align*}
\begin{split}
\tT''_{i_0, e} (E_k) & =  \mbox{\textomega} \tT'_{i_0, e}  \mbox{\textomega}  (E_k) \\
&=  \mbox{\textomega} \tT'_{i_0,e} (F_k) \\
& =  \mbox{\textomega}  (
\sum_{r+s=-\langle i_0, k'\rangle } (-1)^r v^{-er}_{i_0} ( - v^{-e}_{i_0} F_{i_0, e})^{(r)} F_k ( -v^{-e}_{i_0} F_{i_0, e})^{(s)})\\
& =
\sum_{r+s=-\langle i_0, k'\rangle } (-1)^r v^{-er}_{i_0} ( - v^{-e}_{i_0} ( - v^e_{i_0} E_{i_0,e}))^{(r)} E_k ( - v^{-e}_{i_0} ( - v^e_{i_0} E_{i_0,e}))^{(s)}\\
&= \sum_{r+s=-\langle i_0, k'\rangle } (-1)^r v^{-er}_{i_0}  E^{(r)}_{i_0,e} E_k E^{(s)}_{i_0,e}
\end{split}
\end{align*}
\begin{align*}
\begin{split}
\tT''_{i_0, e}(F_k) & =  \mbox{\textomega}  \tT'_{i_0,e}  \mbox{\textomega}  (F_k)\\
& =  \mbox{\textomega}  \tT'_{i_0, e} (E_k) \\
& =  \mbox{\textomega}  ( 
\sum_{r+s=- \langle i_0,k'\rangle} (-1)^r v^{er}_{i_0} (-v^e_{i_0} E_{i_0,e})^{(r)} E_k ( -v^{e}_{i_0} E_{i_0,e})^{(s)})\\
& = \sum_{r+s=- \langle i_0,k'\rangle} (-1)^r v^{er}_{i_0} ( -v^e_{i_0}  ( -v^{-e}_{i_0} F_{i_0, e}) )^{(r)} F_k ( - v^e_{i_0} (-v^{-e}_{i_0} F_{i_0,e}))^{(s)}\\
& =  \sum_{r+s=- \langle i_0,k'\rangle} (-1)^r v^{er}_{i_0} F^{(r)}_{i_0,e} F_k F^{(s)}_{i_0, e}.
\end{split}
\end{align*}
The above computations show that the remaining two equalities in the proposition hold. This finishes the proof of the proposition. 
\end{proof}

\subsection{The operators $T'_{i_0, e}$ and $T''_{i_0, e}$}

\index{ $T'_{i_0, e}$} \index{$T''_{i_0, e}$}

For any  $\ve\in \{ \pm 1\}$, 
let $\V_{\bI, \ve}$ be the subalgebra of $\U_I$ generated by $E_{i_0, \ve}$, $F_{i_0, \ve}$, $E_i,  F_i, K_\mu$ 
for all $i\in \bI-\{ i_0\}$ and $\mu \in Y$. In other words, $\V_{\bI,\ve}$ is the image of $\U_{\bI}$ under $\Psi_\ve$.
\index{$\V_{\bI, \ve}$}
We define an isomorphism \index{$\chi_{i_0, -1}$}
\[
\chi_{i_0,-1} : \V_{\bI, -\ve} \to \V_{\bI, \ve}
\]
by
\begin{align*}
& E_{i_0, -\ve} \mapsto - v^{\ve}_{i_0} E_{i_0,\ve} , \
 F_{i_0, -\ve} \mapsto - v^{-\ve}_{i_0} F_{i_0, \ve}, \\
& E_i\mapsto E_i, \  F_i\mapsto F_i, \
 K_\mu\mapsto K_\mu,
\forall i\in \bI-\{i_0\}, \mu \in Y.
\end{align*}
We define an isomorphism  \index{$\chi_{i_0, 1}$}
\[
\chi_{i_0, 1} : \V_{\bI, -\ve} \to \V_{\bI, \ve}
\]
by
\begin{align*}
& E_{i_0, -\ve} \mapsto - v^{-\ve}_{i_0} E_{i_0,\ve} , \
F_{i_0, -\ve} \mapsto - v^{\ve}_{i_0} F_{i_0, \ve}, \\
& E_j \mapsto (-v_{i_0})^{- \ve \langle i_0, j'\rangle} E_j,\
F_j \mapsto (-v_{i_0})^{\ve \langle i_0, j'\rangle } F_j,\
K_\mu \mapsto K_\mu, \quad \forall j\in \bI-\{ i_0\}, \mu \in Y.
\end{align*}

We have the following compatibility of the automorphism $\tT'_{i_0,e}$ with the similar one on $\U_{\bI}$.

\begin{prop}
\label{Braid-3}
For any  $e, \ve\in \{\pm 1\}$, we assume either that
$e\ve = -1$ and $i_+$ is an end vertex or that 
$e\ve =1$ and $i_-$ is an end vertex. Then the composition $T'_{i_0, e}= \chi_{i_0,e\ve} \tT'_{i_0,e}$ defines an automorphism 
of $\V_{\bI, \ve}$. Moreover, the automorphism $T'_{i_0,\ve}$ coincides with the automorphism in the same notation in $\U_{\bI}$ under the embedding $\Psi_\ve$.
\end{prop}

\begin{proof}
Assume that
$e\ve = -1$ and $i_+$ is an end vertex. 
Because $i_+$ is an end vertex, the assumption (\ref{Braid}) is satisfied. 
Then from (\ref{T1}), (\ref{T2}), (\ref{T5}) and (\ref{T6}), we have 
\begin{align*}
T'_{i_0, e} (E_{i_0,\ve}) & = \chi_{i_0,e\ve} \tT'_{i_0,e} (E_{i_0,\ve})\\
& = \chi_{i_0, -1}( - \tilde K_{ei_0} ( - v^{-e}_{i_0} F_{i_0, - \ve}))\\
&=- \tilde K_{ei_0} ( - v^{-e}_{i_0} (- v^{-\ve}_{i_0} F_{i_0, \ve}) )\\
&= - \tilde K_{ei_0} F_{i_0, \ve} 
\end{align*}
\begin{align*}
T'_{i_0, e} (F_{i_0,\ve}) & = \chi_{i_0,e\ve} \tT'_{i_0, e} (F_{i_0,\ve})\\
&=  \chi_{i_0, -1} (   - (-v^e_{i_0} E_{i_0, -\ve} ) \tilde K_{- ei_0})\\
&= - (-v^e_{i_0} ( - v^{\ve}_{i_0} E_{i_0,\ve})) \tilde K_{- ei_0})
\\
&= 
- E_{i_0,\ve}\tilde K_{- e i_0}
\end{align*}
\begin{align*}
T'_{i_0, e} (E_k) & = \chi_{i_0, -1} \tT'_{i_0, e}  (E_k ) \\
&= \chi_{i_0, -1} (
 \sum_{r+s=-\langle i_0, k'\rangle} (-1)^r v^{er}_{i_0} ( - v^e_{i_0} E_{i_0, e})^{(r)} E_k (- v^e_{i_0} E_{i_0, e})^{(s)}, 
)\\
&=  \sum_{r+s=-\langle i_0, k'\rangle} (-1)^r v^{er}_{i_0} E^{(r)}_{i_0, \ve} E_k  E^{(s)}_{i_0, \ve}
\end{align*}
\begin{align*}
T'_{i_0, e} (F_k) & = \chi_{i_0, -1} \tT'_{i_0, e}  (F_k ) \\
& = \chi_{i_0, -1} (
\sum_{r+s=-\langle i_0, k'\rangle} (-1)^r v^{-er}_{i_0} ( - v^{-e}_{i_0} F_{i_0, e})^{(s)} F_k (- v^{-e}_{i_0} F_{i_0, e})^{(r)}
)\\
&= \sum_{r+s=-\langle i_0, k'\rangle} (-1)^r v^{-er}_{i_0} F^{(s)}_{i_0, \ve} F_k F^{(r)}_{i_0, \ve},\quad \forall k\neq i_0.
\end{align*}
Clearly $T'_{i_0,e}(K_\mu) = K_{s_{i_0}(\mu)}$ for all $\mu\in Y$.
Therefore, we see that $T'_{i_0,e}$ is well-defined on $\V_{\bI,\ve}$ if $e\ve=-1$ and $i_+$ is an end vertex.
Moreover, the formulas we obtained are compatible with the formulas of $T'_{i,e}$ on $\U_{\bI}$. Hence the statements
in the proposition hold under the assumption that $e\ve =-1$ and $i_+$ is an end vertex.

Assume now that $e\ve=1$ and $i_-$ is an end vertex. 
Since $i_-$ is an end vertex, we see that the assumption (\ref{Braid}) is satisfied.
By the formulas (\ref{T1})-(\ref{T4}), we have
\begin{align*}
T'_{i_0, e} ( E_{i_0,\ve}) & = \chi_{i_0, 1} \tT'_{i_0, e} ( E_{i_0, \ve})\\
& = \chi_{i_0,1} ( - \tilde K_{ei_0} ( - v^{-e}_{i_0} F_{i_0, -\ve}) ) \\
&=  - \tilde K_{ei_0}  ( - v^{-e}_{i_0}  ( - v^{\ve}_{i_0} F_{i_0,\ve})) \\
&= - \tilde K_{ei_0} F_{i_0, \ve}
\end{align*}
\begin{align*}
T'_{i_0,e} (F_{i_0,\ve}) & = \chi_{i_0, 1} \tT'_{i_0, e} (F_{i_0,\ve})\\
&=  \chi_{i_0, 1} ( - (-v^e_{i_0} E_{i_0, -\ve}) \tilde K_{-ei_0})\\
&= - (-v^e_{i_0} (-v^{-\ve}_{i_0} E_{i_0,\ve})) \tilde K_{-ei_0}\\
&= - E_{i_0,\ve} \tilde K_{-ei_0}
\end{align*}
\begin{align*}
T'_{i_0,e} (E_j) & = \chi_{i_0, 1} \tT'_{i_0,e} (E_j) \\
& = \chi_{i_0, 1} ( \sum_{r+s= - \langle i_0, j'\rangle} (-1)^r v^{er}_{i_0} E^{(r)}_{i_0, -e} E_j E^{(s)}_{i_0, -e})\\
&=  \sum_{r+s= - \langle i_0, j'\rangle} (-1)^r v^{er}_{i_0}  
(- v^{-\ve}_{i_0} E_{i_0,\ve})^{(r)} ((-v_{i_0})^{-\ve \langle i_0, j'\rangle} E_j) ( -v^{-\ve}_{i_0} E_{i_0,\ve})^{(s)}\\
&=  \sum_{r+s= - \langle i_0, j'\rangle} (-1)^r v^{er}_{i_0}   (-v_{i_0})^{-\ve (r+s) -\ve \langle i_0, j'\rangle} 
E^{(r)}_{i_0, \ve} E_j E^{(s)}_{i_0,\ve}\\
&=   \sum_{r+s= - \langle i_0, j'\rangle} (-1)^r v^{er}_{i_0}  E^{(r)}_{i_0, \ve} E_j E^{(s)}_{i_0,\ve}
\end{align*}
\begin{align*}
T'_{i_0,e}(F_j) & = \chi_{i_0,1}\tT'_{i_0,e} (F_j) \\
&= \chi_{i_0,1} ( \sum_{r+s= - \langle i_0, j'\rangle} (-1)^r v^{-er}_{i_0}F^{(s)}_{i_0, -e} F_j F^{(r)}_{i_0, -e})\\
&= \sum_{r+s= - \langle i_0, j'\rangle} (-1)^r v^{-er}_{i_0}  ( -v^{\ve}_{i_0} F_{i_0,\ve})^{(s)} ( (-v_{i_0})^{\ve \langle i_0, j'\rangle} F_j) (-v^{\ve}_{i_0} F_{i_0, \ve})^{(r)}\\
& = \sum_{r+s= - \langle i_0, j'\rangle} (-1)^r v^{-er}_{i_0}  (-v_{i_0})^{\ve(r+s) + \ve \langle i_0,j'\rangle} 
F_{i_0, \ve}^{(s)} F_j F^{(r)}_{i_0,\ve}\\
& = \sum_{r+s= - \langle i_0, j'\rangle} (-1)^r v^{-er}_{i_0} F_{i_0, \ve}^{(s)} F_j F^{(r)}_{i_0,\ve}.
\end{align*}
These formulas are compatible with the formula of $T'_{i_0,e}$ on $\U_{\bI}$ and hence the statements in the proposition hold under the assumption that $e\ve=1$ and $i_-$ is an end vertex. The proposition is therefore proved. 
\end{proof}

Similarly we have the following compatibility of the operator $T''_{i_0,\ve}$ and $\tT''_{i_0,\ve}$. 

\begin{prop}
\label{Braid-4}
For any  $e, \ve\in \{\pm 1\}$, we assume either that
$e\ve = -1$ and $i_+$ is an end vertex or that 
$e\ve =1$ and $i_-$ is an end vertex. Then the composition $T''_{i_0, e}= \chi_{i_0,-e\ve} \tT''_{i_0,e}$ defines an automorphism 
of $\V_{\bI, \ve}$. Moreover, the automorphism $T''_{i_0,\ve}$ coincides with the automorphism in the same notation in $\U_{\bI}$ under the embedding $\Psi_\ve$.
\end{prop}

\begin{proof}
The proof is similar to the proof of Proposition~\ref{Braid-3} by using Proposition~\ref{Braid-2}. We leave it to the readers. 
\end{proof}

Note that the involutions $ \mbox{\textomega}  $ on $\U_{I}$ and $\U_{\bI}$ are not compatible with $\Psi$ and so we do not have 
$T''_{i_0,e} =  \mbox{\textomega}  T'_{i_0, e}  \mbox{\textomega}  $ in $\V_{I, \ve}$.


\begin{prop}
\label{Braid-5}
Assume that the condition (\ref{Braid}) holds. 
For any $j\in \bI-\{i_0\}$ and $e\in \{ \pm 1\}$, the automorphisms $T'_{j,e}$ (resp. $T''_{j, e}$) on $\U_I$ and $\U_{\bI}$ are compatible under $\Psi_\ve$. 
In particular, we have
\begin{align*}
T'_{j, e} (E_{i_0,\ve}) &= \sum_{r+s= - \langle j, i_0\rangle} (-1)^r v^{er}_j E^{(r)}_j E_{i_0,\ve} E^{(s)}_j,\\
T'_{j, e}(F_{i_0, \ve}) & = \sum_{r+s=-\langle j, i_0 \rangle} (-1)^r v^{-er}_j F^{(s)}_j F_{i_0,\ve} F^{(r)}_j,\\
T''_{j, e} (E_{i_0,\ve}) & =  \sum_{r+s= - \langle j, i_0\rangle} (-1)^r v^{er}_j E^{(s)}_j E_{i_0,\ve} E^{(r)}_j,\\
T''_{j, e}(F_{i_0, \ve}) & = \sum_{r+s=-\langle j, i_0 \rangle} (-1)^r v^{-er}_j F^{(r)}_j F_{i_0,\ve} F^{(s)}_j.
\end{align*}
\end{prop}

\begin{proof}
By definition, we see that the automorphisms $T'_{j, e}$ on $\U_I$ and $\U_{\bI}$ are compatible on all generators, except $E_{i_0,\ve}$ and $F_{i_0,\ve}$. If we can show the equalities in the proposition, then the compatibility holds.
Under the assumption, we know that $j\cdot i_+=0$ or $j\cdot i_-=0$. Without lost of generalities, we assume that $j\cdot i_-=0$. 
Then we have 
\begin{align*}
T'_{j, e} (E_{i_0, \ve}) & = T'_{j, e} (E_{i_+}) E_{i_-} - v^{-\ve}_{i_0} E_{i_-} T'_{j,e} (E_{i_+}) \\
&=\sum_{r+s = -\langle j, i'_+\rangle} (-1)^r v^{er}_j E^{(r)}_j E_{i_+} E^{(s)}_j E_{i_-} 
- v^{-\ve}_{i_0} E_{i_-} \sum_{r+s = -\langle j, i'_+\rangle} (-1)^r v^{er}_j E^{(r)}_j E_{i_+} E^{(s)}_j \\
& =\sum_{r+s = -\langle j, i'_0\rangle} (-1)^r v^{er}_j E^{(r)}_j  ( E_{i_+} E_{i_-} - v^{-\ve}_{i_0} E_{i_-} E_{i_+}) E^{(s)}_j \\
&= \sum_{r+s= - \langle j, i_0\rangle} (-1)^r v^{er}_j E^{(r)}_j E_{i_0,\ve} E^{(s)}_j.
\end{align*}
The above computation verifies the first equality. The second one can be checked in exactly the same manner. This shows that $T'_{j,e}$ on 
both $\U_I$ and $\U_{\bI}$ are compatible under $\Psi_\ve$. 

The statement on $T''_{j,e}$ can be proved in the same way as that on $T'_{j, e}$. The proposition is thus proved.
\end{proof}

By Propositions~\ref{Braid-3}--\ref{Braid-5}, we have 

\begin{thm}
For any  $e, \ve\in \{\pm 1\}$, we assume either that
$e\ve = -1$ and $i_+$ is an end vertex or that 
$e\ve =1$ and $i_-$ is an end vertex. 
The operators $T'_{i,e}$ (resp. $T''_{i,e}$)  for all $i\in \bI$ and $e\in \{\pm 1\}$ on $\U_I$ and $\U_{\bI}$ are compatible under $\Psi_\ve$. 
\end{thm}

\begin{rem}
\begin{enumerate}
\item When specialized to $v=1$, the braid group actions on $\U_I$ and $\U_{\bI}$ descend to Weyl group actions on the associated enveloping algebras. They are compatible up to a sign. 
\item It is interesting to see if the restriction of the operators $T'_{i, e}$ and $T''_{i, e}$ to $\V_{\bI,\ve}$ satisfy the braid relations of $W_{\bI}$. 
\end{enumerate}
\end{rem}

\subsection{Braid group actions on the subquotient  $\U_{I,\bI}/\mathcal J_{I, \bI}$}

Recall that $\U_{I,\bI}$ be the subalgebra of $\U_I$ generated by the elements $E_i, F_i, K_{\mu}$ for $i\in \bI-\{i_0\}$ and $\mu\in Y$, and $E_{i_+} E_{i_-}$, $E_{i_-}E_{i_+}$,
$F_{i_-}F_{i_+}$ and $F_{i_+}F_{i_-}$.  
And recall that  $\mathcal J_{I, \bI}$ is the two-sided ideal of $\U_{I, \bI}$ generated by $E_{i_-}E_{i_+}$ and $F_{i_+}F_{i_-}$. 
By definition, we have
\begin{align}
\begin{split}
E_{i_-} E_{i_+} = \frac{E_{i_0,\ve} - E_{i_0, -\ve}}{v^{\ve}_{i_0} - v^{-\ve}_{i_0}},\
E_{i_+} E_{i_-} = \frac{ v^{\ve}_{i_0} E_{i_0,\ve} - v^{-\ve}_{i_0} E_{i_0, -\ve}}{v^{\ve}_{i_0} - v^{-\ve}_{i_0}} , \\
F_{i_+} F_{i_-} = \frac{ F_{i_0, -\ve}- F_{i_0,\ve}}{v^{\ve}_{i_0} - v^{-\ve}_{i_0}} , 
F_{i_-} F_{i_+} = \frac{ v^{\ve}_{i_0} F_{i_0,-\ve} - v^{-\ve}_{i_0} F_{i_0,\ve}}{v^{\ve}_{i_0} - v^{-\ve}_{i_0}}. 
\end{split}
\end{align}
This implies that $\U_{I, \bI}$ is generated by $E_i, F_i, K_{\mu}$ for $i\in \bI-\{i_0\}$ and $\mu\in Y$ and $E_{i_0, \ve}, E_{i_0, -\ve}$,
$F_{i_0,\ve}$ and $F_{i_0, -\ve}$. 
By a straightforward computation, we get 
\begin{align}
\begin{split}
\tT'_{i_0, e} ( E_{i_-} E_{i_+} ) = \tilde K_{e i_0} v^{-e}_{i_0} F_{i_+} F_{i_-}, \
\tT'_{i_0, e} (E_{i_+} E_{i_-}) = \tilde K_{ei_0} v^{-e}_{i_0} F_{i_-} F_{i_+} , \\
\tT'_{i_0, e} (F_{i_+} F_{i_-}) = v^e_{i_0} E_{i_-} E_{i_+} \tilde K_{- ei_0} ,\
\tT'_{i_0, e} ( F_{i_-} F_{i_+}) = v^e_{i_0} E_{i_+} E_{i_-} \tilde K_{-ei_0}. 
\end{split}
\end{align}
This implies that 
\[
\tT'_{i_0, e} (\U_{I, \bI}) \subseteq \U_{I, \bI}, \tT'_{i_0, e} ( \mathcal J_{I, \bI}) \subseteq \mathcal J_{I, \bI}. 
\]
Therefore it induces an automorphism on the quotient algebra $\U_{I,\bI}/\mathcal J_{I, \bI}$, still denoted by $\tT'_{i_0, e} : \U_{I,\bI}/\mathcal J_{I, \bI} \to \U_{I,\bI}/\mathcal J_{I, \bI}$. 
Since $E_{i_0, \ve} = E_{i_0,-\ve}$ and $F_{i_0,\ve} =F_{i_0, -\ve}$ in  $\U_{I,\bI}/\mathcal J_{I, \bI}$,  the formula in Proposition~\ref{Braid-1} can be rewritten as  follows. 
For any $e, \ve \in \{ \pm 1\}$,  we have 
\begin{align*}
\begin{split}
\tT'_{i_0, e} (E_{i_0, \ve}) & = - \tilde K_{ei_0} ( - v^{-e}_{i_0} F_{i_0,  \ve}), \\
\tT'_{i_0, e} (F_{i_0, \ve} ) & = - (-v^e_{i_0} E_{i_0, \ve} ) \tilde K_{- ei_0},\\
\tT'_{i_0,e} (E_j) & = \sum_{r+s= - \langle i_0, j'\rangle} (-1)^r v^{er}_{i_0} E^{(r)}_{i_0, \ve} E_j E^{(s)}_{i_0, \ve}, \quad  \forall j\cdot i_-=0,\\
\tT'_{i_0, e} (F_j) & = \sum_{r+s= - \langle i_0, j'\rangle} (-1)^r v^{-er}_{i_0}F^{(s)}_{i_0, \ve} F_j F^{(r)}_{i_0, \ve}, \quad  \forall j\cdot i_- =0,\\
\tT'_{i_0, e} (E_k) & = \sum_{r+s=-\langle i_0, k'\rangle} (-1)^r v^{er}_{i_0} ( - v^e_{i_0} E_{i_0, \ve})^{(r)} E_k (- v^e_{i_0} E_{i_0, \ve})^{(s)}, 
\quad  \forall k \cdot i_+=0,\\
\tT'_{i_0, e} (F_k) & = \sum_{r+s=-\langle i_0, k'\rangle} (-1)^r v^{-er}_{i_0} ( - v^{-e}_{i_0} F_{i_0, \ve})^{(s)} F_k (- v^{-e}_{i_0} F_{i_0, \ve})^{(r)}, 
\quad \forall k\cdot i_+=0.
\end{split}
\end{align*}

Now define an automorphism $\chi'_{i_0, e}$, for $e\in \{\pm 1\}$,  on $\U_{I, \bI}$ by the following rules \index{$\chi'_{i_0, e}$}
\begin{align*}
&K_{\mu} \mapsto K_\mu, \quad \forall \mu \in Y, \\
 &E_{i_0, \ve} \to - v^{-e}_{i_0} E_{i_0, \ve}, F_{i_0,\ve} \mapsto - v^{e}_{i_0} F_{i_0, \ve}, \\
 &E_j \mapsto - v^{\langle i_0, j'\rangle}_{i_0} E_j, F_j \mapsto - v^{-\langle i_0, j'\rangle} F_j, \quad \langle i_0, j'\rangle =0,\\
 &E_j\mapsto E_j, F_j\mapsto F_j, \quad  \langle i_0, j'\rangle \neq 0
\end{align*}
Clearly, the automorphism $\chi'_{i_0, e}$ leaves the ideal $\mathcal J_{I, \bI}$ stable, and thus induces an automorphism on the quotient
$\U_{I, \bI}/\mathcal J_{I, \bI}$, still denoted by the same notation.  Let
$T'_{i_0, e} = \chi'_{i_0, e} \tT'_{i_0, e}$.

For $j\in \bI-\{i_0\}$, the braid group action $T'_{j, e}$ on $\U_I$ descends to an automorphism on $\U_{I, \bI}/\mathcal J_{I, \bI}$, still denoted by the same notation.

In an entirely similar manner, the operator $\tT''_{i_0, e}$ induces an automorphism on $\U_{I, \bI}/\mathcal J_{I, \bI}$, still denoted by the same notation.
Define an automorphism $\chi''_{i_0, e}$ on $\U_{I, \bI}/\mathcal J_{I, \bI}$ by  \index{$\chi''_{i_0, e}$} 
\begin{align*}
&K_{\mu} \mapsto K_\mu, \quad \forall \mu \in Y, \\
 &E_{i_0, \ve} \to - v^{e}_{i_0} E_{i_0, \ve}, F_{i_0,\ve} \mapsto - v^{-e}_{i_0} F_{i_0, \ve}, \\
 &E_j \mapsto - v^{-\langle i_0, j'\rangle}_{i_0} E_j, F_j \mapsto - v^{\langle i_0, j'\rangle} F_j, \quad \langle i_0, j'\rangle \neq 0,\\
 &E_j\mapsto E_j, F_j\mapsto F_j, \quad  \langle i_0, j'\rangle = 0
\end{align*} 
Let $T''_{i_0, e} = \chi''_{i_0, e} \tT''_{i_0, e}$.

For $j\in \bI-\{i_0\}$, the braid group action $T''_{j, e}$ on $\U_I$ descends to an automorphism on $\U_{I, \bI}/\mathcal J_{I, \bI}$, still denoted by the same notation.
By  tracing the generators, we have

\begin{thm}
Assume that the condition (\ref{Braid}) holds. 
For any  $e, \ve\in \{\pm 1\}$, 
the operators $T'_{i,e}$ (resp. $T''_{i,e}$)  for all $i\in \bI$ and $e\in \{\pm 1\}$ on $\U_{I, \bI}/\mathcal J_{I, \bI}$ and $\U_{\bI}$ are compatible under $\Phi$ in (\ref{phi}). 
\end{thm}

\section{Examples: linear trees}
\label{Linear}
In this section, we show that the embeddings induced by an edge contraction along an edge in a linear tree can be related to a standard embedding 
via Lusztig's braid group actions on quantum groups. 

\subsection{Standard embeddings}

In this section, we shall show that when the pair $\{ i_+, i_-\}$ is on a linear tree, the embedding $\Psi_\ve$ can be obtained from a naive embedding by applying repetitively
Lusztig's symmetries. To avoid confusion, we shall write $\Psi_{\ve;\{ i_+, i_-\}}$ for $\Psi_\ve$ whenever necessary.

Assume that we have three vertices $i_1, i_2, i_3$ in $I$ such that 
\[
i_1\cdot i_1=i_2\cdot i_2=i_3\cdot i_3= -2 i_1\cdot i_2=-2i_2\cdot i_3,
\]
and that $i_1\cdot i_3=0$ and $i_2\cdot j=0$ for all $j \neq i_1, i_2, i_3$.  
We can consider the edge contractions $(I_1, \cdot)$ and $(I_2, \cdot)$  of the Cartan datum $(I, \cdot)$ alone $\{i_1, i_2\}$ and $\{i_2, i_3\}$ respectively. 
For a fixed root datum of $(I, \cdot)$, it can both be regarded naturally as the root data of $(I_1, \cdot)$ and $(I_2, \cdot)$. 
One can check that $(I_1, \cdot ) $ and $(I_2, \cdot)$ is isomorphic, where $i_1+i_2 , i_3$ correspond to $i_1$, $i_2+i_3$ respectively.
This isomorphism is defined by the simple reflection $s_{i_2}$.  
Let $\U_{I_1} $ and $\U_{I_2}$ be the associated quantum group with respect to the root data. 
Induced by the isomorphism $(I_1,\cdot) \to (I_2, \cdot)$, there exists an isomorphism $S_{i_2}: \U_{I_2} \to \U_{I_1}$ given by $K_\mu \mapsto K_{s_{i_2} \mu}$, $E_i\mapsto E_{s_{i_2} (i)}$ and $F_i\mapsto F_{s_{i_2}(i)}$ for all  $\mu\in Y $, $i\in I_2$; in particular, $E_{i_2+i_3} \mapsto E_{i_3}$, $F_{i_2+i_3} \mapsto F_{i_3}$, 
$E_{i_1} \mapsto E_{i_1+i_2}$, and $F_{i_1} \mapsto F_{i_1+ i_2}$. 
We have 

\begin{lem}
\label{Naive-1}
Under the above setting, we have the following commutative diagram.
\[
\begin{CD}
\U_{I_2} @> \Psi_{\ve;\{i_2, i_3\}} >> \U_I\\
@VS_{i_2} VV @VV T'_{i_2, -\ve} V\\
\U_{I_1} @>\Psi_{\ve; \{i_1, i_2\}}>> \U_I
\end{CD}
\]
\end{lem}

\begin{proof}
When restricting to  Cartan parts, the horizontal morphisms are identities while the vertical morphisms coincide by definition. Hence the diagram commutes when restricts to Cartan parts. There are no effects of the morphisms on $E_i$ and $F_i$ for $i \neq i_1$, $i_2+i_3$. For $i=i_1$, the application of the morphisms in the bottom left path is 
$E_{i_1} \mapsto E_{i_1+i_2} \mapsto E_{i_1} E_{i_2} - v^{-\ve}_{i_0} E_{i_2} E_{i_1} $ and the application of the morphisms in the top right path is 
$E_{i_1} \mapsto E_{i_1} \mapsto E_{i_1} E_{i_2} - v^{-\ve}_{i_0} E_{i_2} E_{i_1} $; hence coincide. 
Similarly, the application of the morphisms in the bottom left path on $E_{i_1+i_2}$ is $E_{i_2+i_3} \mapsto E_{i_3} \mapsto E_{i_3}$ and the application of the morphisms in the upper right path is $E_{i_2+i_3} \mapsto E_{i_2} E_{i_3} - v^{-\ve}_{i_0} E_{i_3} E_{i_2} \mapsto E_{i_3}$; hence coincide. 
The commutativity with respect to the generators $F_i$ can be checked in a similar manner, and  the detail is skipped. This finishes the proof. 
\end{proof}

Assume that $\{i_1, i_2\}$ satisfies the condition (\ref{i-comp}), i.e., $i_1\cdot i_1=i_2\cdot i_2=-2i_1\cdot i_2$ and that $i_2$ is an end vertex, i.e., 
$i_2\cdot j=0$ for all $j\neq i_1, i_2$. 
Let $(I_1,\cdot)$ be the edge contraction of $(I, \cdot)$ along $\{i_1, i_2\}$. 
Let $(I_2, \cdot)$ be the Cartan datum obtained from $(I, \cdot)$ by throwing away $i_2$ and data related to $i_2$. 
Then we still have an isomorphism from $(I_1,\cdot)$ to $(I_2,\cdot)$ via $s_{i_2}$. 
Fix a root data for $(I_1,\cdot)$ and $(I_2,\cdot)$, respectively, induced from one of $(I,\cdot)$. 
Let $\U_{I_1}$ and $\U_{I_2}$ be the respective quantum group associated to the above root data. 
Then we have an isomorphism $S_{i_2} : \U_{I_2} \to \U_{I_1}$ by sending $K_\mu$ to $K_{s_{i_2} \mu}$ and $E_{i} \mapsto E_{s_{i_2} (i)}$ and
$F_{i} \mapsto F_{s_{i_2}(i)}$ for all $\mu\in Y$ and $i\in I_2$.
Let $\Upsilon_{i_2}: \U_{I_2} \to \U_I$ be the standard embedding defined by $K_\mu \mapsto K_\mu$, $E_i\mapsto E_i$ and $F_i\mapsto F_i$ for all $\mu \in Y$ and $i\in I_2$. 

\begin{lem}
\label{Naive-2}
Retaining the above assumptions, we have the following commutative diagram.
\[
\begin{CD}
\U_{I_2} @> \Upsilon_{i_2} >> \U_I\\
@VS_{i_2} VV @VVT'_{i_2, -\ve}V\\
\U_{I_1} @> \Psi_{\ve, \{i_1, i_2\}} >>  \U_I
\end{CD}
\]
\end{lem}

\begin{proof}
The proof is a simplified version of that of Lemma~\ref{Naive-1}. We leave it to the reader. 
\end{proof}

Now we can state the main result in this section.

\begin{prop}
\label{Naive-3}
Assume that $i_1, i_2, \cdots, i_n$ for $n\geq 2$ forms a linear tree, i.e.,
\[
i_a \cdot i_a = -2 i_{a} \cdot  i_{a+1}, i_a \cdot i_b =0, \forall |a-b| \neq 0, 1, i_a \cdot j=0,\forall a\geq 2, j \neq i_1,\cdots, i_n.
\]
Assume further that $i_n$ is an end vertex. Then we have
\begin{align}
\label{Naive-4}
\Psi_{\ve, \{i_1, i_2\}} = T'_{i_2,-\ve} T'_{i_3, -\ve} \cdots T'_{i_n, -\ve} \Upsilon_{i_n} S^{-1}_{i_n} S^{-1}_{i_{n-1}} \cdots S^{-1}_{i_2}.  
\end{align}
\end{prop}

\begin{proof}
This is resulted from applying Lemma~\ref{Naive-1} repetitively and applying Lemma~\ref{Naive-2} in the last step. 
\end{proof}

We end this section with a remark.

\begin{rem}
\label{Naive-5}
\begin{enumerate}
\item Proposition~\ref{Naive-3} gives an alternative proof of Theorem~\ref{Psi-U} when $\{i_+, i_-\}$ lies on a linear tree. 

\item The commutativity is compatible with the integral forms of various quantum groups involved. 

\item There is a similar commutativity for modified quantum groups. 

\item In light of Proposition~\ref{Naive-3}, when the edge contraction is along a linear tree, Problem~\ref{branching} is equivalent to the branching rule with respect to the embedding $\Upsilon$.
The latter problem is solved when the data $(I, \cdot)$ is of classical type.

\end{enumerate}
\end{rem}

\section{Examples: cyclic quivers}
\label{cyclic}
In the section, we consider the cyclic quiver $C_n$ of $n$ vertices. 
We establish a connection of the quotient $\mu_\nu$ in (\ref{quotient-mu})   with a natural embedding in  affine flag varieties. This connection inspired this work. 
We further show that $\f^{\widehat I}_{C_{n+1}} /\f^c_{C_{n+1}}$ is bigger than $\f_{C_n}$.

\subsection{Affine flags}
\label{flag}

Let $\k((t))$ be the field of formal Laurent polynomials with coefficients in $\k$. 
Let $\k[[t]]$ be the subring of $\k((t))$ consisting of all formal power series. 
Let $\W$ be a $\k((t))$-vector space of dimension $w$. 
A lattice $L$ in $\W$ is a free $\k[[t]]$-module of rank $w$ such that $\k((t)) \otimes_{\k[[t]]} L= \W$. 
Fix $n$ and let $\mathscr F_{n,w}$ be the collection of chains of lattices $L_\bullet= (L_1\subseteq L_2\subseteq \cdots \subseteq L_n\subseteq t^{-1} L_1)$. 
Let $\mathcal F^{2,+}_{n,w}$ be the set of all pairs $(L_\bullet, L'_{\bullet})$ such that $L_i\supseteq L'_i$ for all $1\leq i\leq n$.
To each pairs $(L_\bullet, L'_{\bullet})$, we set $\V_i = L_i/L'_i$ for all $1\leq i\leq n$. 
The two inclusions $L_i\subseteq L_{i+1}, L'_i\subseteq L'_{i+1}$ induces a linear map $x_{i\to i+1}: \V_i\to \V_{i+1}$ for $i=1,\cdots, n-1$. 
The inclusions $L_n\subseteq t^{-1}L_1$ and $L'_n\subseteq t^{-1}L'_1$ defines a linear map $y_{n\to 1}: \V_n\to t^{-1} \V_1$. 
Set $x_{n\to 1} = ty_{n\to 1}$.  Then the collection $x_{L_\bullet,L'_{\bullet}}= (x_{i\to i+1})_{i\in \mbb Z/n\mbb Z}$ is an element in $\E_{\V, C_n}$ where $C_n$
is the cyclic quiver $\{ i\to i+1| i\in \mbb Z/n\mbb Z\}$. 
Fix $i_+$ in $\mbb Z/n\mbb Z$ such that $i_-=i_++1$. 
Let us consider the subset $\mathscr F^{2,+}_{n,w}|_{n-1}$ of $\mathscr F^{2,+}_{n,w}$ consisting of all pairs $(L_\bullet,L'_\bullet)$ such that 
$L_{i_+}=L_{i_-}$ and $L'_{i_-}=L'_{i_-}$. 
Clearly we have that the set $\mathscr F^{2,+}_{n,w}|_{n-1}$ is in bijection with $\mathscr F^{2, +}_{n-1, w}$. 
By definition we also have that if $(L_{\bullet}, L'_{\bullet}) \in \mathscr F^{2,+}_{n, w}$ then $x_{L_\bullet, L'_\bullet}\in \E^{\heartsuit}_{\V,C_n}$
because the associated $x_{i_+\to i_-}$ is the identity map.  
The edge contraction of $C_n$ along $\{ i_+,i_-\}$ is  the cyclic quiver $C_{n-1}$. 
So the quotient map $\E^{\heartsuit}_{\V,C_n} \to \E_{\widehat{\V}, C_{n-1}}$  is corresponding to the bijection $\mathscr F^{2,+}_{n,w}|_{n-1}=\mathscr F^{2, +}_{n-1, w}$, which was used extensively in~\cite{LS20, FLLLW}. 

\subsection{Hall algebras of cyclic quivers}
\label{cyclic}
Fix $i_+$ in $C_n$. Let $i_-=i_++1$. Let $C_{n-1}$ be the edge contraction of $C_n$ along $\{  i_+, i_-\}$. 
Fix $a=(a_i)\in \mbb N^{\bI}$. Consider the following monomial
$$
\theta_{a,C_n}=\theta_{i_+, C_n}^{(a_{i_+})} \theta_{i_+ -1, C_n}^{(a_{i_+-1})} \theta_{i_+ -2, C_n}^{(a_{i_+-2})} \cdots \theta_{i_- +1, C_n}^{(a_{i_-+1})} \theta_{i_-,C_n}^{(a_{i_+})}
$$
Let $\mathscr O_a$ be the zero orbit in $\E_{\widehat V, C_{n-1}}$ for $\widehat V$ of dimension vector $a$. One has
\begin{align}
\label{j-gene}
j^*(\theta_{a, C_n} ) = 1_{\mathscr O_a}.
\end{align}
The quiver representation corresponding to $\mathscr O_a$ is periodic if all entries in $a$ are nonzero. 
As such, $1_{\mathscr O_a}$ is not in $\f_{C_{n-1}}$ as long as all entries in $a$ are nonzero. 
This implies that $\f_{C_{n-1}}$ is a proper subalgebra of $\f^{\bI}_{C_n}/\f^c_{C_n}$. 

Further, it is known that $1_{\mathscr O_a}$ for various $a$ forms a generating set of the Hall algebra $\H_{C_{n-1}}$ and so
from (\ref{j-gene}, we see that 

\begin{prop}
\label{cyclic-1}
We have
$$
\f^{\bI}_{C_n} /\f^c_{C_n} \cong  \H_{C_{n-1}}.
$$
\end{prop}
Note that $\psi_\Omega (\H_{C_{n-1}})$ is not contained in $ \f^{\bI}_{C_n}$. 
Here is an example. When $n=2$, $1_{\mathscr O_2}$ gets sent  via $\psi_\Omega$ to $1_{\mathscr O}$ where $\mathscr O$ is the orbit of $(x_a, x_b)$ where $a: i_+ \to i_-$ and $b: i_-\to i_+$ and $x_a=1$ and $x_b=0$. Now one observes that $1_{\mathscr O}$ is not in $\f^{\bI}_{C_2}$ by considering the support of the canonical basis elements in $\f_{C_2}$ of dimension vector $(2, 2)$. 
So it is not clear yet if the quotient $\f^{\bI}_{C_n}\to \H_{C_{n-1}}$ splits.

Finally, by using (\ref{j-gene}) and multiplication formulas for aperiodic semisimple generators in $\H_{C_n}$
one can deduce the multiplication formula for a periodic semisimple generator in $\H_{C_{n-1}}$ in~\cite[Theorem 2.1]{DZ18}.
This proof is the counterpart of the proof given in~\cite{LS20}.


In light of Proposition~\ref{cyclic-1}, it is interesting to see if quantum affine $\mathfrak{gl}_n$ is a subquotient of a quantum affine $\mathfrak{sl}_{n+1}$.

\subsection{Quivers with loops}

Note that when $n=2$, the edge contraction of $C_2$ is the Jordan quiver. Strictly speaking, this case is not covered in previous sections. However, with slight modifications, this case can be included.   Moreover, the treatment in Section~\ref{Hall} remains valid if we consider the more general case when $a=1$ and there are more than one arrow in between $i_+$ and $i_-$. 
In this case, one only needs to modify the definition $\mu_\nu$ in (\ref{quotient-mu}) slightly. 
That is if the edge contraction is along an arrow, say $e$, such that $e'=i_+$ and $e''=i_-$, then 
the map $\mu_\nu$ sends $x$ to $\widehat x$ such that $\widehat x_{h \bar e} = x_e^{-1} x_h $ if $h'=i_+$ and $h''=i_-$, and $\widehat x_{h e}= x_h x_e$ if $h'=i_-$ and $h''=i_+$.  
Then Theorem~\ref{Hall-emb} remains valid in this setting, and so the analysis in the above section makes sense. 
Furthermore, Proposition~\ref{H-subq} holds. In particular,
the image of the subalgebra generated by the elements $\theta^{(n)}_{i, \Omega}$ for $i\neq i_+, i_-$ and $\theta^{(n)}_{i_+,\Omega} \theta^{(n)}_{i_-,\Omega}$ 
for various $n$ under the quotient $j^*$ was previous studied by Lusztig, Li-Lin and Bozec.  A detailed analysis will be given elsewhere.

\printindex

\begin{thebibliography}{99999}\frenchspacing

\bibitem[BKLW]{BKLW}
H. Bao, J. Kujawa, Y. Li and W. Wang,
{\em Geometric Schur duality of classical type},
Transform. Groups {\bf 23} (2018), no. 2, 329--389.

\bibitem[D03]{D03}
S. Doty, 
{\em Presenting generalized q-Schur algebras}. Represent. Theory 7 (2003), 196--213.



 \bibitem[DZ18]{DZ18}
 J. Du and Z. Zhao,  
 {\em Multiplication formulas and canonical bases for quantum affine $\mathfrak{gl}_n$}. Canad. J. Math. {\bf 70} (2018), no. 4, 773--803.
 
 




\bibitem[FLLLW]{FLLLW}
Z. Fan, C. Lai, Y. Li, L. Luo and W. Wang,
{\em Affine flag varieties and quantum symmetric pairs}, 
 Mem. Amer. Math. Soc. {\bf 265} (2020), no. 1285, v+123 pp.


\bibitem[GV93]{GV93}
V. Ginzburg and E. Vasserot,
{\em Langlands reciprocity for affine quantum groups of type $A_n$}. Internat. Math. Res.
Notices 1993, no. 3, 67--85.

\bibitem[G99]{G99} 
R.M. Green,
{\em The affine $q$-Schur algebra,} J. Algebra {\bf 215} (1999), 379--411.


\bibitem[Li21]{Li21}
Y. Li,
{\em Embeddings among quantum affine $\mathfrak{sl}_n$},
Acta Mathematica Sinica, English Series, to appear.


\bibitem[LS20]{LS20}
Y. Li and A. Samer,
{\em On multiplication formulas of affine $q$-Schur algebras}, 
 J. Pure Appl. Algebra {\bf 226} (2022), no. 7, Paper No. 106943.

\bibitem[Lu96]{L96}
G. Lusztig,
{\em Braid group action and canonical bases},
Adv. Math. {\bf 122}, 237--261 (1996).

\bibitem[Lu98]{L98}
G. Lusztig, 
{\em Canonical bases and Hall algebras}. 
{\em Representation theories and algebraic geometry} (Montreal, PQ, 1997), 365--399, NATO Adv. Sci. Inst. Ser. C Math. Phys. Sci., {\bf 514}, Kluwer Acad. Publ., Dordrecht, 1998.

\bibitem[Lu99]{L99}
G. Lusztig,
{\em Aperiodicity in quantum affine $\mathfrak{gl}_n$},
Asian J. Math. {\bf 3} (1999), 147--177.


\bibitem[Lu09]{L09}
G. Lusztig, 
{\em Study of a Z-form of the coordinate ring of a reductive group.} 
J. Amer. Math. Soc. {\bf 22} (2009), no. 3, 739--769. 

\bibitem[Lu10]{L10}
G. Lusztig,
{\em Introduction to quantum groups},
Reprint of the 1994 edition. Modern Birkh\"{a}user Classics. Birkh\"{a}user/Springer, New York, 2010. xiv+346 pp.

\bibitem[M15]{M15}
R. Maksimau,
{\em Categorical representations, KLR algebras, and Koszul duality}, 
\href{https://arxiv.org/pdf/1512.04878v1.pdf}{arXiv:1512.04878}.

\bibitem[M18]{M18}
R.  Maksimau, 
{\em Categorical representations and KLR algebras.} Algebra Number Theory {\bf 12} (2018), no. 8, 1887--1921. 

\bibitem[RW18]{RW18}
S. Riche and G. Williamson, 
{\em Tilting modules and the p-canonical basis}, Ast\'{e}risque {\bf 397}, 2018.


\bibitem[R90]{R90}
 C. M. Ringel, 
 {\em Hall algebras and quantum groups.} Invent. Math. 101 (1990), no. 3, 583--591.
 
\bibitem[X97]{X97}
 J. Xiao,  
 {\em Drinfeld double and Ringel-Green theory of Hall algebras}. J. Algebra 190 (1997), no. 1, 100-144.
 


\end{thebibliography}
\end{document}